\documentclass[12pt,a4paper,openright,twoside,titlepage]{report}

\usepackage[spanish,british]{babel}
\usepackage[utf8]{inputenc}
\usepackage{amsmath}
\usepackage{amsfonts}
\usepackage{amssymb}
\usepackage{amsthm}
\usepackage{graphicx}
\usepackage{fancyhdr}
\usepackage{makeidx}
\usepackage{tikz}
\usepackage{subfigure}
\usepackage{hyperref}
\usepackage{dsfont}
\usepackage{parskip}
\usepackage{aliascnt}
\usepackage[round]{natbib}
\usepackage{enumerate}
\usepackage{mathrsfs}

\hypersetup{
    colorlinks=false,
    pdfborder={0 0 0},
}

\begingroup
    \makeatletter
    \@for\theoremstyle:=definition,remark,plain\do{%
        \expandafter\g@addto@macro\csname th@\theoremstyle\endcsname{%
            \addtolength\thm@preskip\parskip
            }%
        }
\endgroup

\usepackage{boxedminipage}
\makeindex

\DeclareMathAlphabet{\mathpzc}{OT1}{pzc}{m}{it}
\newcommand{\desc}{\alpha} %parámetro de descuento
\newcommand{\F}{\mathcal{F}}
\newcommand{\FF}{\{\F_t\colon t\geq 0\}}%{\mathbb{F}}
\newcommand{\T}{\bf T}
\newcommand{\R}{\mathds{R}}

\newcommand{\B}{\mathcal{B}}
\newcommand{\CR}{\mathcal{C}} %continuation region
\newcommand{\SR}{\mathcal{S}} %stopping region

\newcommand{\Ga}{G_{\desc}}
\newcommand{\Ra}{R_{\desc}}
\newcommand{\Va}{V_{\desc}}
\newcommand{\Wa}{W_{\desc}}
\newcommand{\phia}{\varphi_\desc}
\newcommand{\psia}{\psi_\desc} 
\newcommand{\wa}{w_\desc}
\newcommand{\al}{(\desc-L)}

\newcommand{\hit}[1]{\mathpzc{h}_{#1}}
\newcommand{\ind}[1]{\mathds{1}_{#1}}
\newcommand{\Ig}[1]{\I_{>#1}}
\newcommand{\Ige}[1]{\I_{\geq #1}}
\newcommand{\Il}[1]{\I_{<#1}}
\newcommand{\Ile}[1]{\I_{\leq #1}}
\newcommand{\g}{{g}}
\newcommand{\gi}{\tilde{g}}
\newcommand{\gpsi}{\frac{g}{\psia}}

\newcommand{\D}{\mathcal{D}}
\newcommand{\Da}{\mathcal{D}_{\desc}}
\newcommand{\CC}{\mathscr{C}} %funciones continuas C^2...
\newcommand{\I}{\mathcal{I}}
\newcommand{\ea}[1]{e^{-\desc{#1}}}
\newcommand{\sfp}{(SF)}
\newcommand{\ssfp}{(SSF)}  %% smooth fit
\newcommand{\psfp}{($\desc$-SF)} %% alpha smooth fit
\newcommand{\estados}{\mathcal{E}}
\newcommand{\sigalg}{\mathscr{E}}
\newcommand{\Aa}{A_{\desc}}
\newcommand{\Ex}[2]{\E_{#1}\left(#2\right)}
\newcommand{\be}{\begin{equation}}
\newcommand{\ee}{\end{equation}}
\newcommand{\bd}{\begin{equation*}}
\newcommand{\ed}{\end{equation*}}
\newcommand{\ve}{\mathbf{v}}
\newcommand{\vx}{x_v}
\newcommand{\vy}{y_v}
\newcommand{\vz}{z_v}
\newcommand{\we}{\mathbf{w}}

\newcommand{\se}{\mathbf{s}}
\newcommand{\sx}{x_s}
\newcommand{\sy}{y_s}
\newcommand{\sz}{z_s}

\newcommand{\mug}{\nu}

\def\P{\operatorname{\mathds{P}}}
\def\E{\operatorname{\mathds{E}}}
\def\erf{\operatorname{erf}}

\newtheorem{teo}{Theorem}[chapter]

\newaliascnt{cor}{teo}
\newaliascnt{remark}{teo}
\newaliascnt{remarks}{teo}
\newaliascnt{lem}{teo}
\newaliascnt{prop}{teo}
\newaliascnt{obs}{teo}
\newaliascnt{defn}{teo}
\newaliascnt{example}{teo}
\newaliascnt{notation}{teo}
\newaliascnt{cond}{teo}

\newtheorem{cor}[cor]{Corolary}
\newtheorem{lem}[lem]{Lemma}
\newtheorem{prop}[prop]{Proposition}

\theoremstyle{definition}
\newtheorem{remark}[remark]{Remark}
\newtheorem{remarks}[remarks]{Remarks}
\newtheorem{defn}[defn]{Definition}
\newtheorem{example}[example]{Example}

\newtheorem{cond}[cond]{Condition}
\newtheorem{algorithm}{Algorithm}[chapter]
\newtheorem{hip}{\quad H.\hspace{-4pt}}[chapter]
\aliascntresetthe{cor}
\aliascntresetthe{remark}
\aliascntresetthe{remarks}
\aliascntresetthe{lem}
\aliascntresetthe{prop}
\aliascntresetthe{obs}

\aliascntresetthe{defn}
\aliascntresetthe{example}
\aliascntresetthe{notation}
\aliascntresetthe{cond}

\begin{document}

\pagestyle{empty}

\quad
\vspace{.5cm}
\begin{center}
\Large {\textsc{Doctoral Thesis}} \\

\vspace{3.5cm}

\begin{LARGE}
 \textbf {Optimal Stopping for Strong Markov Processes:}
\end{LARGE}

\begin{large}
\emph{Explicit solutions and verification theorems for diffusions,\\
\vspace{-.2cm}
multidimensional diffusions, and jump-processes.}
\end{large}

\vspace{1.8cm}

\large{Fabián Crocce}\\

\vspace{0.6cm}

\large{\textit{Advisor: Ernesto Mordecki\\
Centro de Matemática}}

\end{center}
\vspace{3.6cm}
\begin{normalsize}
\begin{tabular*}{20in}{lr}
&\textsc Doctorado en Matemática (PEDECIBA)\\
&Facultad de Ciencias\\
&Universidad de la República\\
7th December 2012 \hspace{2.4cm} & Uruguay
\end{tabular*}

\end{normalsize}
\newpage
\quad
\newpage
\newpage
\begin{center} \textbf{Abstract} \end{center}

We consider the optimal stopping problem consisting in, given a strong Markov process $X=\{X_t\}$ taking values in $\estados$, and a reward function $g\colon \estados \to \R$, finding the optimal stopping time $\tau^*$ and the value function $\Va$ such that:
\bd
\Va(x)=\Ex{x}{\ea{\tau^*} \g(X_{\tau^*})} = \sup_{\tau}\Ex{x}{\ea{\tau} \g(X_{\tau})},
\ed
where the supremum is taken over the class of all stopping times, $\desc$ is a positive discount rate, and $x$ is the starting point of $X$. The approach we follow, has two main components: the Dynkin's characterization of the value function as the smallest $\desc$-excessive function dominating $g$; and the Riesz representation of $\desc$-excessive functions in terms of the Green kernel, the main reference being \cite{salminen85}.
In the context of one-dimensional diffusions we give a complete characterization of the solution under some assumptions on $g$. If the optimal stopping problem is one-sided (the optimal stopping region is $[x^*,\infty)$ or $(-\infty,x^*]$) we provide a simple equation to find the threshold $x^*$ and discuss the validity of the smooth fit principle. 
We include some new examples as the optimal stopping of the skew Brownian motion and the sticky Brownian motion. In particular, we consider cases in which the smooth fit principle fails. In the general case, we propose an algorithm that finds the optimal stopping region when it is a disjoint union of intervals. We also give a simple formula for the value function. Using this algorithm we solve some examples including polynomial rewards. 
For general Markov processes with continuous sample paths (for instance multidimensional diffusions) we provide a verification theorem and use it to solve a particular problem. Finally we consider one-dimensional strong Markov processes with only positive (or only negative) jumps, and provide another verification theorem for right-sided (left-sided) problems. As applications of our results we address the problem of pricing an American put option in a Lévy market, and also solve an optimal stopping problem for a Lévy driven Ornstein-Uhlenbeck process.

%\selectlanguage{spanish}
%\begin{abstract}
\newpage

\quad
\newpage
\begin{center} \textbf{Resumen} \end{center}

Consideramos el problema de parada óptima que consiste en, dados un proceso de Markov fuerte $X=\{X_t\}$ a valores en $\estados$, y una función de pago $g\colon \estados \to \R$, encontrar el tiempo de parada óptima $\tau^*$ y la función de valor $\Va$ que verifican:
\bd
\Va(x)=\Ex{x}{\ea{\tau^*} \g(X_{\tau^*})} = \sup_{\tau}\Ex{x}{\ea{\tau} \g(X_{\tau})},
\ed
donde el supremo es tomado sobre la clase de todos los tiempos de parada, $\desc$ es una tasa de descuento positiva, y $x$ es el estado del que parte el proceso. El enfoque que seguimos se basa en dos componentes: la caracterización de  Dynkin de la función de valor como la mínima función $\desc$-excesiva que domina $g$; y la representación de Riesz de las funciones $\desc$-excesivas en términos del núcleo de Green. La principal referencia para este enfoque es \cite{salminen85}.
En el contexto de las difusiones unidimensionales damos una caracterización completa de la solución, asumiendo algunas condiciones sobre $g$. Si el problema de parada óptima es tal que la región de parada es de la forma $[x^*,\infty)$ o de la forma $(-\infty,x^*]$, damos una ecuación sencilla para encontrar el valor crítico $x^*$ y discutimos la validez del principio de pegado suave.
También incluimos algunos ejemplos nuevos como ser la parada óptima del movimiento browniano asimétrico (skew) y del movimiento browniano pegajoso (sticky); en particular damos ejemplos en que no vale el principio de pegado suave. En el caso general, proponemos un algoritmo que encuentra la región de parada óptima cuando ésta es una unión disjunta de intervalos, dando también una fórmula sencilla para la función de valor. Usando el algoritmo mencionado resolvemos algunos ejemplos que incluyen funciones de pago polinomiales.
Para procesos de Markov generales con trayectorias continuas (como ser las difusiones multidimensionales) damos un teorema de verificación que luego usamos para resolver un problema concreto. Por último, consideramos un proceso de Markov fuerte unidimensional solo con saltos positivos (o solo con saltos negativos) y damos un teorema de verificación para problemas en que la región de parada es de la forma $[x^*,\infty)$ (de la forma $(-\infty,x^*]$). Como aplicación de los resultados obtenidos consideramos el problema de la valuación de una opción americana de tipo put en un mercado de Lévy, y también resolvemos un problema de parada óptima en que el proceso subyacente es un proceso de Ornstein-Uhlenbeck combinado con un proceso de Lévy.

\newpage
\pagestyle{plain}

\thispagestyle{empty}
\ \\[10em]
{\large a Karen\ldots}
\newpage
\chapter*{Agradecimientos}\addcontentsline{toc}{chapter}{Agradecimientos}
Empiezo por expresar mi más profundo agradecimiento a Ernesto Mordecki, quien desde mis primeros días en la facultad siempre ha estado presente.
Como orientador ha sido incondicional, dispuesto a trabajar hasta un domingo o en vacaciones con la misma pasión y dedicación de siempre. 
En los momentos de sequía, en que el avance en la tesis resultaba difícil, nunca faltó su confianza y su aliento, sin duda imprescindibles para para seguir adelante. 
Su vocación y entusiasmo por la matemática siempre me resultaron admirables y motivadores. Trabajar con el Cacha ha sido, y confío en que seguirá siendo, un gran placer.

Por acompañarme en todo, por apoyarme en cada decisión, por alegrarme cada día y porque sin ella nada tendría sentido, no alcanzan las palabras para agradecerle a Karen.

Le agradezco infinitamente a mis padres. A mi madre por su continua lucha para que cada uno de nosotros consiga siempre lo que quiere, por hacerme sentir que nada es imposible, por su optimismo, por la confianza y por la generosidad. A mi padre por enseñarme a vivir con alegría, por motivarme a hacer lo que me gusta y por compartir conmigo sus pasiones.

Parte de hacer este trabajo con entusiasmo tiene que ver con la gente que me rodea día a día, que se enorgullece de mis logros y con eso me da ánimo para seguir adelante. En ese sentido le quiero agradecer a mi familia y a mis amigos. En particular a mis hermanas y cuñados porque son el estandarte de la alegría cotidiana. También quiero nombrar en particular a la abuela, a Ana, a los padres de Karen, a la tía Elsa y a la tía Nelly porque siempre están presentes. A la barra del 12, que es como mi familia, y mis amigos de ingeniería con quienes he compartido y espero seguir compartiendo momentos inolvidables. A Pilar por su generosidad y alegría contagiosa. Otra mención especial se merece Rodri, mi primer sobrino, que en este momento de gran presión nos inundó de felicidad.

Aprovechar para reconocer que mi acercamiento a la matemática se lo debo a Fernando Peláez, que en febrero de 2002 en San Luis, mientras lo acribillaba a preguntas para el examen de matemática A (con Parisi), me recomendó hacer las matemáticas en la Facultad de Ciencias, cuando yo no dudaba en estudiar Ingeniería. 

Quiero agradecerle a mis compañeros de trabajo, y en especial a mis compañeros del seminario de Finanzas, por escucharme tantas veces, por los consejos y por hacer que uno se sienta parte de un equipo. A Claudia, Sandra y Lydia, les agradezco la paciencia que me tienen y el hacer más simple y alegre el día a día en el CMAT. 

El desarrollo de este trabajo no habría sido posible sin el apoyo económico de la Agencia Nacional de Investigación e Innovación. También quiero reconocer el apoyo del Programa de Desarrollo de Ciencias Básicas y del Centro de Matemática.

Quiero agradecer especialmente a Paavo Salminen. En primer lugar por haberme dado un borrador de su trabajo sobre difusiones \citep{notasPaavo}, que resultó de gran ayuda. En segundo lugar por haberme recibido tan cálidamente en Turku. En mi visita a \r{A}bo Akademi avancé mucho en mi trabajo sobre difusiones, las discusiones con Salmien fueron de gran provecho, en particular él me propuso considerar el interesante ejemplo del “sticky Brownian motion”.  Por último quiero agradecerle su visita a Uruguay y el haber aceptado venir nuevamente desde Finlandia para participar del tribunal de tesis.

Por último les agradezco a  Enrique Cabaña, Ricardo Fraiman, Gustavo Guerberoff, Matilde Martínez, Gonzalo Perera, Raúl Tempone por haber aceptado integrar el tribunal evaluador.
\selectlanguage{british}
\tableofcontents 
\listoffigures 

\chapter*{Introduction}\addcontentsline{toc}{chapter}{Introduction}

\subsection*{About optimal stopping}\addcontentsline{toc}{section}{About optimal stopping}
The theory of optimal stopping is concerned with the problem of choosing the moment
to take certain action with the purpose of maximizing an expected reward or minimizing 
an expected cost. 
It has applications in many fields such as theory of probability, mathematical statistics, economics, mathematical finance and control theory.

One of the most famous problems in this area is the \emph{secretary problem} (also known as \emph{marriage problem}, \emph{the sultan's dowry problem}, \emph{the fussy suitor problem}, etc.) in which an administrator aims to maximize the probability of hiring the best of $N$ applicants to the position of secretary. 
In this problem the applicants are interviewed sequentially in random order and they can be ranked during the interview and compared only 
with those already interviewed; after each interview, the administrator has to decide whether the interviewee is chosen (without this assumption the problem is trivially solved by interviewing all the applicants and selecting the best). 
According to \cite{ferguson1989secretary} the secretary problem appears for the first time in print, in Martin Gardner's February 1960 column in Scientific American, where it was called the game of googol.

\subsubsection*{Historical comments}
%\section*{Related works}\addcontentsline{toc}{section}{Related works}
Optimal stopping problems have a long %va long no large
 history in literature. 
The first problems arose 
%In the context of continuous time processes the story began 
within the framework of statistics in the late 40s, %framework va con within, en vez de in
 more precisely, in the context of sequential analysis with the works of \cite{wald,wald2}, and \cite{wald3,wald4}. %(we refer to the book by Wald \cite{wald} for discrete time observations), 
A few years later, \cite{snell1952applications} proposed a general optimal stopping problem, for discrete-time stochastic processes, and characterized its solution as the smallest supermartingale dominating the gain process --known as \emph{Snell envelope}--; 
this result, which stands relevant today, can be considered one of the first major theoretical results in optimal stopping. 
This approach, sometimes called \emph{the martingale approach}, is also treated in the book by \cite{chow1971great}.

A later impulse to optimal stopping theory is related to mathematical finance, where arbitrage considerations show that 
in order to price an American option, 
one has to solve an optimal stopping problem. 
The first results in this direction were provided by \cite{mckean1965appendix} and \cite{merton}, 
who respectively solved the perpetual put and call option pricing problem, by solving the corresponding optimal stopping problems in the
context of the Black and Scholes model \citep{bs}. 
Hedging arguments justifying the use of optimal stopping in option pricing were provided by \cite{bensoussan1984theory} and \cite{karatzas1988pricing}.
The work by \cite{jacka1991optimal} also treats the relationship between the option pricing and optimal stopping.
Mathematical finance problems continue to motivate works on optimal stopping, also regarding processes with jumps, which intend to model  turbulences of the markets. Without meaning to be exhaustive, we cite the books: \cite{boyarchenko2002non,boy-lev, cont-tan}; and the articles:  \citet{gerber1999ruin,mordeckiOSPOLP,boyarchenko2002perpetual,alili-kyp, christensen2009note}.

\subsubsection*{Verification vs Theoretical approach}
When considering works on optimal stopping problems we typically find two approaches.
In the first approach, in order to solve a concrete optimal stopping problem, 
one has to somehow guess the solution and prove that the guessed candidate, 
in fact, solves the optimization problem; this approach is known as \emph{verification}.
The second one is the \emph{theoretical} approach, that typically includes results about properties of the solution. 
But these two approaches seldom meet, as frequently in concrete problems the assumptions of the theoretical studies are not fulfilled, 
and, what is more important, 
the theoretical studies do not provide concrete ways to find solutions. 
%In what concerns the first approach, 

As far as the first approach is concerned, the relationship between optimal stopping and free boundary problems is of key importance. A common procedure, to delimit the region in which is optimal to stop, is to apply the \emph{principle of smooth fit}, used for the first time by \cite{mikhalevish1958Ukranian}, that 
generally leads to the solution of two equations: 
the \emph{continuous fit}  and the \emph{smooth fit} equations. 
Once these equations are solved, 
a verification procedure is needed to prove that the candidate is the 
%effective
actual solution to the problem itself \citep[for more details see][chapter IV]{ps}. This approach, when an explicit solution can be found, is very effective.
%In what concerns 

With regard to the second approach, the most important results may be: Dynkin's characterization of the value function $\Va$ as the least $\desc$-excessive (or $\desc$-superharmonic) dominating the reward \citep[see][]{dynkin:1963} in the continuous-time case; and the already exposed Snell's characterization in discrete-time. 
The book by \cite{shiryaev} (which is a reprint of \cite{shirBook1978}, whose first version in English is \cite{shirBook1973}) provides a comprehensive treatment of Dynkin's characterization.

Other ways of classifying approaches in the study of optimal stopping problems include the Martingale-Markovian dichotomy. The Martingale approach is treated in the monograph \citep{chow1971great} for discrete time processes. As for the Markovian approach, the monograph \citep{shiryaev} was, for a long time, the main reference. In the more recent book \citep{ps}, both approaches are extensively analysed and compared.

% a good exposition of this application to statistics may be found in the book by \citet{shiryaev} and the more recent book by \cite{ps}.

\subsection*{The problem}\addcontentsline{toc}{section}{The problem}
A general continuous time optimal stopping problem can be stated as follows: given a gain process $\{G_t\}_{t\geq 0}$, defined in a filtered probability space $(\Omega,\F,\{\F_t\},\P)$, find the stopping time $\tau^*$ that maximizes the expected gain, i.e.
\bd
\E(G_{\tau^*})=\sup_{\tau} \E(G_\tau),
\ed
where the supremum is taken over all stopping times that are less than or equal to certain horizon $T$ that could be either a positive real number, or $T=\infty$. In the former case, the problem is said to be of \emph{finite horizon}, while in the latter, it is said to be an \emph{infinite-horizon} or \emph{perpetual} problem.

In this work (that follows the Markovian approach) we consider continuous-time infinite-horizon problems, where $G_t$ is given by $\ea{t}\g(X_t)$ with:
\begin{itemize}
\item a Markov process $\{X_t\}$ taking values in a certain topological state  space $\estados$; 
\item a Borel function $\g\colon \estados \to \R$, called the reward function; and
\item a strictly positive discount rate $\desc$.
\end{itemize}
%\subsection*{The problem}\addcontentsline{toc}{subsection}{The problem}
We aim to find the optimal stopping time $\tau^*$ and the value function $\Va\colon \estados \to \R$ that satisfy
\bd
\Va(x)=\Ex{x}{\ea{\tau^*}\g(X_{\tau^*})}=\sup_{\tau}\Ex{x}{\ea{\tau}\g(X_{\tau})},
\ed
where the supremum is taken over the class of all stopping times.
%where $\{X_t\}_{t\geq 0}$ is a given time-homogeneous Markov process with state space $\estados$, and $\g\colon \estados\to \R$ is the reward function.

\subsection*{The approach}\addcontentsline{toc}{section}{The approach}
We start by giving an informal description of our approach.
\begin{itemize}

\item Our starting point is Dynkin's characterization of the value function $\Va$ as the least $\desc$-excessive function that dominates the reward $\g$ \citep{dynkin:1963}.
Once the value function $\Va$ is known, so is the optimal stopping time $\tau^*$, which is the first time at which the process hits the set $\SR$, so-called \emph{stopping region}, given by
\bd
\SR:=\{x\in \estados:g(x)=\Va(x)\}.
\ed

\item The second step uses Riesz decomposition of an $\desc$-excessive function, according to which
\begin{equation}\label{eq:green2}
\Va(x)=\int_{\estados} \Ga(x,y)\nu(dy) + \text{(a harmonic function)},
\end{equation}
where $\Ga$ is the Green function of the process with respect to some reference measure $m(dy)$, and
$\nu$ is a non-negative Radon measure \citep{kw}. In this introduction we assume that the harmonic function in \eqref{eq:green2} is zero.

\item The third step is based on the fact that the resolvent and the infinitesimal generator of a Markov process are inverse operators. With this idea in mind, suppose that we could write
\begin{equation}\label{eq:inversionV}
\Va(x)=\int_{\estados} \Ga(x,y)\al \Va(y)m(dy),
\end{equation}
where $L$ is the infinitesimal generator, and $m(dy)$ some reference measure. 

\item Finally, we observe that $\Va=\g$ within the stopping region $\SR$, being $\desc$-harmonic in its complement. Therefore
\bd
\Va(x)=
\begin{cases}
\g(x),& x\in \SR,\\
\desc\text{-harmonic},& \mbox{else.}
\end{cases}
\ed
Assuming that the infinitesimal generator at $x$ depends only on the values of the function in a neighbourhood of $x$ we would get
\bd
\al \Va(x)=
\begin{cases}
\al \g(x),& x\in \SR,\\
0,& \mbox{else.}
\end{cases}
\ed
Comparing \eqref{eq:green2} and \eqref{eq:inversionV}, and considering the previous equation, we would obtain that $\nu$ is supported on $\SR$, which is in fact a known result in some cases \citep[see][]{mordecki-salminen}, and $\nu(dx)=\al g(x) m(dx)$.
%(the same holds for $\kappa$ in \eqref{eq:martin}).
\end{itemize}

From the previous considerations we obtain that
\be
\label{eq:candidateVa}
\Va(x):=\int_\SR \Ga(x,y)\al g(y) m(dy),
\ee
with $\SR$ a stopping region, is a plausible value function.

This approach was initiated by \cite{salminen85}, and used by \cite{mordecki-salminen}. In \cite{alvarez1998exit} some applications can be found. These articles are in fact the starting point of our work. 
According to Salminen's approach,
once the excessive function is represented as an integral with respect to the Martin Kernel --related with the Green Kernel--, 
\begin{equation}\label{eq:martin}
\Va(x)=\int_{\I} M(x,y)\kappa(dy)
\end{equation}
one has to find the representing measure $\kappa$. In \citep{salminen85} the author provides a way to express $\kappa(dy)$ in terms of the derivatives (with respect to the scale function of a diffusion) of the value function.

\subsection*{Our contributions}\addcontentsline{toc}{section}{Our contributions}
 In this work we apply the exposed methodology in three different 
situations: (i) one-dimensional diffusions; (ii) general (multidimensional) strong Markov processes with continuous sample paths; and (iii) strong Markov processes with one-sided (only positive or only negative) jumps.
%kind of Markov processes, trying to find out whether the previous formula holds. 

In all these three cases formula \eqref{eq:candidateVa} is used to obtain solutions of the corresponding optimal stopping problem. The more general the problems are, the harder is to solve them explicitly. Hence, different degrees of explicitness are obtained, depending on the process and on the regularity of reward function.

In the case (i) of one-dimensional diffusions, we give a quite comprehensive explicit solution of the optimal stopping problem. The corresponding stopping regions can be either a half-line (one-sided case) or also a union of disjoint intervals. 

In the second and third cases we provide verification theorems in a quite general framework.

As applications, several new concrete examples are explicitly solved. In what follows we discuss each of the three cases in detail.

\subsubsection*{On one-dimensional diffusions}\addcontentsline{toc}{subsection}{On one-dimensional diffusions}
In the context of one-dimensional diffusions whose state space is an interval $\I$ of $\R$, we provide a complete solution, under mild regularity conditions, for problems that are \emph{one-sided}, i.e. the optimal stopping rule has either the form 
\bd
\tau^*=\inf\{t\geq 0\colon X_t\geq x^*\}
\ed
or the form
\bd
\tau^*=\inf\{t\geq 0\colon X_t\leq x^*\}
\ed
for some optimal threshold $x^*$. In the former case, the problem is said to be \emph{right-sided}, while in the latter, is said to be \emph{left-sided}. We prove that the threshold, in the right-sided case, is the solution of a simple equation, which with enough regularity is
\be \label{eq:introx*}
\frac{\g(x)}{\psia(x)}=\frac{g'(x)}{\psia'(x)},
\ee
where $\psia$ is the increasing fundamental solution of $\al g=0$; for left-sided problems the equation is analogous (with the decreasing fundamental solution $\phia$ of $\al g=0$, instead of $\psia$). An equation close to \eqref{eq:introx*} is provided in \cite{salminen85}.
It should be observed that \eqref{eq:introx*} is a combination of both, continuous fit, and smooth fit equations.
An interesting by-product of our work has to do with the \emph{smooth fit principle}.
Our results are in fact independent of the smooth fit principle, although
we obtain sufficient conditions in order to 
guarantee it. Some examples of application of our results are also provided: in particular we solve optimal stopping problems for both the Skew Brownian motion and the Sticky Brownian motion.

Also in the context of one-dimensional diffusions, we consider problems in which the optimal stopping region is two-sided or has a more general form. We prove that the value function has the form given in \eqref{eq:candidateVa}; where the continuation region $\CR=\I \setminus \SR$ is a union of disjoint intervals $J_k$ that satisfy
\bd 
\int_{J_k} \psia(y)\al \g(y) m(dy)=\int_{J_k} \phia(y)\al \g(y) m(dy)=0.
\ed
We also provide an algorithm to compute the continuation region, which then use to solve an example with a polynomial reward function. Some examples with non-differentiable reward are also provided.

\subsubsection*{On general continuous-paths Markov processes} \addcontentsline{toc}{subsection}{On general continuous-paths Markov processes}
Markov processes with continuous sample paths, taking values in abstract topological spaces, including multidimensional diffusions, are considered in \autoref{chap:contMarkov}. In this context, we prove some verification theorems, which are similar in certain sense to the results given for one-dimensional diffusions, but, as expected, given the generality of the framework, considerably weaker. We also solve an optimal stopping problem for a three-dimensional Brownian motion.

\subsubsection*{On Markov processes with jumps} \addcontentsline{toc}{subsection}{On Markov processes with jumps}

We consider processes with one-sided jumps, proving that the kind of representation given in \eqref{eq:candidateVa} can be used in the case of right-sided problems with positive jumps (and also in the case of left-sided problems with negative jumps). It should be noted that right-sided optimal stopping problems (e.g. call options) for processes with negative jumps are easier to solve, as the process hits the stopping region at the border, without overshot.
As applications of this result we consider the problem of pricing American put options in a Lévy market with \emph{positive} jumps. We also solve an optimal stopping problem for a diffusion with jumps, used to model prices in electricity markets \citep{electricity}, that is not a Lévy process.

The organization of this work is as follows: 
In \autoref{chap:markov} we briefly present Markov processes and the subclasses of them in which we are interested. In particular we include some results on potential theory and on optimal stopping, and also some preliminary results further needed. 
Chapters \ref{chap:diffonesided} and \ref{chap:diffGeneral} treat optimal stopping problems for one-dimensional diffusions. 
Problems whose solution is one-sided are solved in \autoref{chap:diffonesided}, while the general case is considered in \autoref{chap:diffGeneral}. 
In \autoref{chap:contMarkov} Markov processes with continuous sample paths in general topological state spaces are considered. 
Chapter \ref{chap:jumps} has results concerning spectrally one-sided Markov processes.
\chapter{Introduction to Markov processes and optimal stopping }
\label{chap:markov}

In the previous introduction we provided a heuristic basis to illustrate our approach for solving optimal stopping problems. This first chapter aims to formalize the concepts already exposed and also present some necessary preliminary results for this monograph. Some of the results of this chapter are well-known, but we include it for the reader's convenience.

\section{Markov processes}
Given a probability space $(\Omega,\F,\P)$ and a measurable space $(\estados,\sigalg)$, such that for all $x\in \estados$ the unitary set $\{x\}$ belongs to $\sigalg$, consider a family of random variables $X=\{X_t\colon t\in Z\}$, where $X_t\colon\Omega\to \estados$. We call $X$ a \emph{stochastic process} with state space $(\estados,\sigalg)$. In this work we only consider continuous time stochastic processes indexed in $Z=[0,\infty)$.

The family of $\sigma$-algebras $\{\F_t: t\geq 0\}$ (also denoted by $\{\F_t\}$) is said to be a \emph{filtration} of the $\sigma$-algebra $\F$ if the following inclusions holds:
\begin{equation*}
 \F_s \subseteq \F_t \subseteq \F \quad \text{for every $s\leq t$}.
\end{equation*}
The stochastic process $X$ is said to be \emph{adapted} to the filtration $\{\F_t\}$ if for every $t\in Z$ the random variable $X_t$ is $\F_t$-measurable.

Consider a stochastic process $X$, defined on $(\Omega,\F,\P)$, with state space $(\estados,\sigalg)$, adapted to $\{\F_t\}$, a filtration of $\F$. Let $\{\P_x:x\in \estados\}$ be a family of probability measures defined on $(\Omega,\F)$. The system $(\{X_t\},\{\F_t\},\{\P_x\})$ is called a (time-homogeneous, non-terminating) \emph{Markov process} if the following conditions are fulfilled:
\begin{enumerate}[(i)]
 \item For every $A \in \F$, the map $x\mapsto \P_x(A)$ is $\sigalg$-measurable.
 \item \label{markovProperty} For all $x\in \estados$, $B \in \sigalg$, $s,t \geq 0$,
    \bd
      \P_x(X_{t+s}\in B|\F_t)=\P_{X_t}(X_s \in B) \quad (P_x-a.s.).
    \ed
 \item For every $x\in \estados$, $\P_x(X_0=x)=1$.
 \item For each $t>0$, and for all $\omega \in \Omega$, there exists a unique $\omega'\in \Omega$ such that
 \begin{equation*}
   X_s(\omega')=X_{t+s}(\omega)\quad (\forall s\geq 0).
 \end{equation*}
\end{enumerate}

Condition \eqref{markovProperty} is known as Markov property, being its intuitive meaning that the future of the process depends only on the present, but not on the past behaviour.

A stochastic process $X$ is said to be \emph{progressively-measurable} with respect to a filtration $\{\F_t\}$ of $\F$ if the map
\begin{equation*}
(t,\omega) \mapsto X_t(\omega)
\end{equation*}
is measurable with respect to $\B \times \F_t$, where $\B$ denotes the Borel $\sigma$-algebra of $[0,\infty)$.

Given the filtration $\{\F_t\}$ of $\F$, a random variable $\tau$ in $(\Omega,\F)$, taking values in $[0,\infty]$ and such that $\{\omega:\tau(\omega)<t\}\in \F_t$, for all $t\geq 0$ is known as a \emph{stopping time} with respect to the filtration $\{\F_t\}$.

%A \emph{stopping time} with respect to a filtration $\{\F_t\}$ is a stochastic process $\tau$ in $(\Omega,\F)$ taking values in $[0,\infty]$ such that $\{\omega:\tau(\omega)<t\}\in \F_t$, for all $t\geq 0$.

A progressively measurable Markov process $X=(\{X_t\},\{\F_t\},\{\P_x\})$ is said to verify the \emph{strong Markov property} if for all stopping times $\tau$ with respect to $\{\F_t\}$, for all $x\in \estados$, for all $B \in \sigalg$, and for any $s \geq 0$,
    \begin{equation} \label{strongMarkovProperty}
      \P_x(X_{\tau+s}\in B|\F_\tau)=\P_{X_\tau}(X_s \in B) \quad (\P_x-a.s.).
    \end{equation}
    
Note that the \emph{strong Markov property}, as we call the previous condition, includes the Markov property by taking deterministic stopping times ($\tau=t$). A progressively-measurable Markov process that satisfies this stronger condition is called a \emph{strong Markov process}.

A filtration $\{\F_t\}$ is said to be \emph{right-continuous} if for all $t\geq 0$, we have that $\F_t=\F_{t^+}$, where $\F_{t^+}$ is the $\sigma$-algebra defined by
\begin{equation} \label{def:F_t^+}
\F_{t^+}:=\bigcap_{s>0}\F_{t+s}.
\end{equation}

Given a stopping time $\tau$ with respect to a filtration $\{\F_t\}$ of $\F$, the family of sets $\F_\tau$ defined by
\begin{equation*}
 \F_\tau:=\{A \in \F\colon \forall t\geq 0, A\cap\{\omega:\tau(\omega)\leq t\} \in \F_t \}
\end{equation*}
is a sub$\sigma$-algebra of $\F$.

A progressively measurable Markov process $X=(\{X_t\},\{\F_t\},\{\P_x\})$ is said to be \emph{left-quasi-continuous} if for any stopping time $\tau$ with respect to $\{\F_t\}$, the random variable $X_{\tau}$ is $\F_{\tau}$-measurable and for any non-decreasing sequence of stopping times $\tau_n \to \tau$
\begin{equation*}
 X_{\tau_n}\to X_{\tau} \quad (\P_x-a.s\ \text{in the set}\ \{\tau<\infty\}).
\end{equation*}
for all $x\in\estados$.

\begin{defn} \label{def:standardMarkov} A left quasi-continuous strong Markov process $X$
%=(\{X_t\}, \{\F_t\},\{\P_x\})$ 
with state space $(\estados, \sigalg)$ is said to be a \emph{standard Markov process} if:
\begin{itemize}
 \item The paths are right continuous, that is to say, for every $\omega \in \Omega$, and for all $t>0$, \bd \lim_{h\to 0^+} X_{t+h}(\omega)=X_t(\omega). \ed
 \item The paths have left-hand limits almost surely, that is, for almost all $\omega \in \Omega$, the limit $\lim_{h\to 0^+} X_{t-h}(\omega)$ exists for all $t>0$.
 \item The filtration $\{\F_t\}$ is right-continuous, and $\F_t$ is $\P_x$-complete for all $t$ and for all $x$.
 \item The state space $\estados$ is semi-compact, and $\sigalg$ is the Borel $\sigma$-algebra.
\end{itemize}
\end{defn}

All the processes we consider in this work are, in fact, standard Markov processes. As we just did in the previous definition, we use the notation $X$, getting rid of $(\{X_t\},\{F_t\},\{\P_x\})$ when it is not strictly necessary, to denote a standard Markov process.
%Given a filtered probability space $(\Omega,\F,\{\F_t\},\P)$
%A function $k\colon \estados \times \sigalg \to [0,\infty]$ is said to be a \emph{kernel} in $(\estados, \sigalg)$ if 
%\begin{itemize}
% \item for every $x\in \estados$ the map $A \mapsto k(x,A)$ is a positive measure in $\sigalg$;
% \item for every $A \in \sigalg$ the map $x \mapsto k(x,A)$ is $\sigalg$-measurable.
%\end{itemize}

For general reference about Markov processes we refer to \cite{dynkin:books,dynkin1969theorems,karlin-tay,rogers2000diffusions,blumenthal-ge,revuz} 
\subsection{Resolvent and infinitesimal generator}
Given a standard Markov process $X$ and an $\sigalg$-measurable function $f\colon \estados \to \R$, 
we say that $f$ belongs to the domain $\D$ of the extended infinitesimal generator of $X$, if there exists an $\sigalg$-measurable function $Af\colon \estados\to \R$ such that  $\int_0^t|Af(x_s)|ds<\infty$ almost surely for every $t$, and
$$f(X_t)-f(X_0)-\int_0^t Af(X_s) ds$$
is a right-continuous martingale with respect to the filtration $\{\F_t\}$ and the probability $\P_x$, for every $x\in \estados$ \citep[see][chap. VII, sect. 1]{revuz}.

The $\desc$-Green kernel of the process $X$ is defined by
\bd
\Ga(x,H):=\int_0^{\infty} \ea{t} \P_x(X_t\in H)dt,
\ed
for $x\in \estados$ and $H\in \sigalg$.

Consider the operator $\Ra$, given by
\be \label{eq:resolvent}
\Ra f(x):=\int_0^\infty \ea{t}\Ex{x}{f(X_t)}dt,
\ee
which can be defined for all $\sigalg$-measurable functions such that the previous integral makes sense for all $x\in \estados$. Note that if, for instance
\bd
\Ra f(x)=\int_0^\infty \ea{t}\E_x|f(X_t)|dt<\infty \quad (x\in \estados)
\ed
then, using Fubini's theorem we may conclude that
\be
\Ra f(x)=\int_{\estados} f(y)\Ga(x,dy). \label{eq:RafGa}
\ee

Considering $T_\desc$ as a random variable with an exponential distribution with parameter $\desc$ (i.e. $\P(T_\desc\leq t)=1-\ea{t}$ for $t\geq 0$) and independent of $X$, define the process $Y$ with state space $\estados \cup \{\Delta\}$ --where $\Delta$ is an isolated point-- by
\bd Y_t:=
\begin{cases}
X_t & \text{if $t<T_\desc$,}\\
\Delta &\text{else.}
\end{cases}
\ed
Given a function $f:\estados\to\R$, we extend its domain by considering $f(\Delta):=0$. Observe that $\Ex{x}{f(Y_t)}=\ea{t}\Ex{x}{f(X_t)}$. We call $Y$ \emph{the $\desc$-killed process} with respect to $X$.

Consider a function $f$ that belongs to the domain $\Da$ of the extended infinitesimal generator of the $\desc$-killed process $Y$. In this case, there is a function $\Aa f\colon \estados \to \R$ such that
\bd
f(Y_t)-f(Y_0)-\int_0^t \Aa f(Y_s) ds
\ed
is a right-continuous martingale with respect to the filtration $\{\F_t\}$ and the probability $\P_x$, for every $x\in \estados$. Bearing the equality $\Ex{x}{f(Y_t)}=\ea{t}\Ex{x}{f(X_t)}$ in mind, it can be seen that
\bd
\ea{t}f(X_t)-f(X_0)-\int_0^t \ea{s}\Aa f(X_s) ds
\ed
is also a right-continuous martingale with respect to the filtration $\{\F_t\}$ and the probability $\P_x$, for every $x\in \estados$. Then
\bd
\Ex{x}{\ea{t}f(X_t)}-f(x)-\Ex{x}{\int_0^t \ea{s}\Aa f(X_s) ds}=0.
\ed
From the previous equation, and assuming that for all $x\in \estados$
\begin{itemize}
\item $\lim_{t\to \infty}\Ex{x}{\ea{t}f(X_t)}=0$ and
\item $\Ex{x}{\int_0^{\infty} \ea{s}|\Aa f(X_s)| ds}<\infty$,
\end{itemize}
we obtain, by taking the limit as $t\to \infty$ and using Lebesgue dominated convergence theorem, that
\bd
f(x)=\int_0^{\infty} \ea{s}\Ex{x}{-\Aa f(X_s)} ds.
\ed
Note that the right-hand side of the previous equation is $\Ra(-\Aa f)(x)$; from this fact and the previous equation we obtain, by \eqref{eq:RafGa},
\be
\label{eq:invExtended}
f(x)=\int_{\estados} {-\Aa f(y)}\Ga(x,dy).
\ee

It can be proved that if the function $f$ belongs to $\D$, it also belongs to $\Da$ and $\Aa f=Af-\desc f$.

\subsection{Dynkin's formula} 
Given a standard Markov process $X$ and a stopping time $\tau$, if $f=\Ra h$, we have that (see e.g. \cite{dynkin:books}, Theorem 5.1 or \cite{karlin-tay} equation (11.36))
\be \label{eq:dynFormMarkov}
f(x)=\Ex{x}{\int_0^\tau \ea{t}h(X_t) dt}+\Ex{x}{\ea{\tau}f(X_\tau)}.
\ee
As we will see further on, this formula has an important corollary in the analysis of optimal stopping problems.
Observe that it can be written in terms of $\Aa f$, when $f\in\Da$ by 
\bd
\Ex{x}{\ea{\tau}f(X_\tau)}-f(x)=\Ex{x}{\int_0^\tau \ea{t} \Aa f(X_t)dt};
\ed
being its validity a direct consequence of the Doob's optional sampling  theorem.

\subsection{$\desc$-Excessive functions}
\label{sec:excessive}
Consider a standard Markov process X. In a few words it may be said that $\desc$-excessive functions are those $f$ such that $f(X_t)$ is a supermartingale. In optimal stopping theory, the reward being a $\desc$-excessive function, means that, in order to maximize the expected discounted reward, the process should be stopped immediately.

A  non-negative measurable function $f\colon \estados \to \R$ is called \emph{$\desc$-excessive} --with respect to the process $X$-- provided that:
\begin{itemize}
\item $\ea{t} \E_x(f(X_t))\leq f(x)$ for all $x\in \estados$ and $t\geq 0$; and
\item $\lim_{t \to 0} \E_x(f(X_t))= f(x)$ for all $x\in \estados$.
\end{itemize}
A $0$-excessive function is just called \emph{excessive}.

If $h\colon \estados \to \R$ is an $\sigalg$-measurable non-negative function, then
$\Ra f$ is $\desc$-excessive \citep[see for instance][]{dynkin:1969}.
\section{Optimal stopping}

%\subsection{the OSP}
\label{theOSP}
To state the optimal stopping problem we consider in this work, which we call \emph{the optimal stopping problem}, or more briefly \emph{the OSP}, we need: 
\begin{itemize}
\item a standard Markov process $X=\{X_t\}_{t\geq 0}$, whose state space we denote by $\estados$ in general and also by $\I$ when it is an interval of $\R$;
\item a reward function $g\colon \estados \to \R$;
\item a discount rate $\desc$, which we assume to be positive (in some specific cases it can be 0).
\end{itemize}
The problem we face is to find a stopping time $\tau^*$ and a value function $\Va$ such that
\bd
\Va(x)=\Ex{x}{\ea{\tau^*}\g(X_{\tau^*})}=\sup_{\tau}\Ex{x}{\ea{\tau}\g(X_{\tau})}\quad (x\in \estados).
\ed
where the supremum is taken over all stopping times with respect to the filtration $\{\F_t\}$ of the standard Markov process $X$. We consider $\g(X_{\tau})=0$ if $\tau=\infty$. With this assumption in mind we conclude that $\Va$ is non-negative. 
Another consequence of our assumption is that the optimal stopping with reward $g$ has the same solution as the problem with reward $g^+=\max\{g,0\}$, as it is never optimal to stop if $g(X_t)<0$. We might, without loss of generality, consider non-negative functions $g$, however, in some examples it is convenient to allow negative values. Observe that if $g$ is a non-positive function the OSP is trivially solved, being $\Va\equiv 0$ and $\tau^*=\infty$ a solution. From now on we assume that $g(x)>0$ for some $x$.

It can be seen that the $\desc$-discounted optimal stopping problem for the process $X$ is equivalent to the non-discounted problem associated with the $\desc$-killed process $Y$. Both problems have the same solution (value function and optimal stopping time).

\subsection{Dynkin's characterization for the OSP}
\label{sec:DynkinChar}

The Dynkin's characterization, proposed in \cite{dynkin:1963}, states that, given the reward function $g\colon \estados \to \R$ satisfying some mild regularity conditions \cite[see][Chapter III, Theorem 1]{shiryaev} the value function $V$, defined by
$$V(x):=\sup_\tau \Ex{x}{g(X_\tau)},$$
is an excessive function, satisfying $V(x)\geq g(x)$ for all $x\in \estados$, and such that if $W$ is another excessive function dominating $g$ then $V(x)\leq W(x)$ for all $x\in \estados$. In this sense $V$ is the \emph{minimal excessive majorant} of the reward.

Observe as well that function $f$ is excessive with respect to $Y$ (the $\desc$-killed process already defined) if and only if it is $\desc$-excessive with respect to $X$. 
Applying the Dynkin's characterization to the value function $V$ and to the process $Y$ and, taking the previous considerations into account, a discounted version of the Dynkin's characterization may be established: the value function 
$$\Va(x):=\sup_\tau \Ex{x}{\ea{\tau}g(X_\tau)}$$
of the $\desc$-discounted optimal stopping problem is the minimal $\desc$-excessive function %that is a majorant of 
that dominates
$g$. 
From this point on we also refer to this result as Dynkin's cha\-rac\-te\-ri\-za\-tion, as is usual done in the optimal stopping literature.

In our context, in order to Dynkin's characterization hold, it is sufficient to consider Borel-measurable reward functions $g$ such that 
\bd
\liminf_{x \to a} g(x)\geq g(a)
\ed
for all $a\in \estados$.

\subsection{Stopping and continuation region}
\label{sec:regions}
Once the value function $\Va$ is found as the minimal $\desc$-excessive majorant, or by any other method, the optimal stopping problem \eqref{eq:osp} is completely solved, as the optimal stopping time is the first time at which the process hits the stopping region. Therefore, $\tau^*$ is defined by
\bd
\tau^*:=\inf\{t\geq 0\colon X_t\in \SR \},
\ed
with $\SR$, the stopping region, defined by
\bd
\SR:=\{x \in \estados\colon \Va(x)=\g(x)\}.
\ed
From the stopping region we can define naturally \emph{the continuation region}, which is its complement, $\CR=\estados\setminus \SR$. With this criteria the state space is divided into stopping states and continuation states. However, it could be the case that there are some states of $\SR$ in which is as good to stop as to continue, so a smaller stopping region $\SR'$ could be defined and the stopping time $\tau'$, defined by
\bd
\tau':=\inf\{t\geq 0\colon X_t\in \SR' \},
\ed
would be optimal as well. In this case
\bd
\Va = \Ex{x}{\ea{\tau^*}g(X_{\tau^*})}= \Ex{x}{\ea{\tau'}g(X_{\tau'})}.
\ed
From now on, we ignore this ambiguity and denote by $\SR$, not only the set in which $\Va$ and $\g$ coincide, but also any set satisfying that to stop at the first time the process hits it, is actually optimal. In any of these cases, we call $\SR$ the \emph{stopping region}.

An important corollary of the Dynkin's formula has to do with the continuation region $\CR$ and the sign of $h$, when the reward function $g$ satisfy $g=\Ra h$. Assuming that $h$ is negative in $A$, a neighbourhood of $x$, it can be proved that $x\in \CR$: by Dynkin's formula \eqref{eq:dynFormMarkov}
\bd
g(x)=\Ex{x}{\int_0^{\hit{A^c}} \ea{t}h(X_t)dt}+\Ex{x}{\ea{\hit{A^c}}g(X_{\hit{A^c}})},
\ed
where $\hit{A^c}$ states for the first time at the process hits the set $S\setminus A$; from the made assumptions it follows that $\Ex{x}{\int_0^{\hit{A^c}}\ea{t}h(X_t)dt}<0$, so 
\bd
\Ex{x}{\ea{\hit{A^c}}g(X_{\hit{A^c}})}\geq g(x),
\ed
proving that is better to stop at $\hit{A^c}$ than to stop at $x$, this implying that the optimal action at $x$ is not to stop, therefore $x\in \CR$.
%suppose $f$ is the reward function in the optimal stopping problem and, assuming that $(\desc - A)g$ is continuous in $x$ and $(\desc - A)g(x)<0$, then there is a neighbourhood $N$ of $x$ such that $(\desc - A)g(y)<0$ for $y\in N$; considering $\tau$ the exit time from the set $N$, we get:
%$$\E_x \left(\int_0^\tau \ea{t}(\desc-A)f(X_t)dt\right)<0$$
%and, consequently $\E_x\left(\ea{\tau}f(X_\tau)\right)>f(x)$ and the optimal action --when the process is in the state $x$-- is not to stop.

\section{Preliminary results}

The following lemmas constitute relevant results in the approach considered in this work. We use the notation $\hit{S}$ for the hitting time of $S$,
\bd
\hit{S}:=\inf\{t\geq 0: X_t\in S\}
\ed

\begin{lem} \label{harmonic} Let $X$ be a standard Markov process, and consider $S \in \sigalg$. Then the Green kernel satisfies;
\bd
\Ga(x,H)=\Ex{x}{\ea{\hit{S}}\Ga(X_{\hit{S}},H )},
\ed
for all $x$ in $\estados$ and for all $H\in \sigalg$, $H\subseteq S$.

In other words, for every $x\in \estados$, both $\Ga(x,dy)$ and $\Ex{x}{\ea{\hit{S}}\Ga(X_{\hit{S}},dy)}$, are the same measure in $S$.
\end{lem}
\begin{proof}

	For $x\in S$ the assertion is clearly valid, since $\hit{S}\equiv 0$. Let us consider $x\in \estados\setminus S$. By the definition of $\Ga$ and some manipulation, we obtain
	\begin{align*}
		\Ga(x,H)&=\Ex{x}{\int_0^{\infty} \ea{t}\ind{H}(X_t)dt}\\
		&= \Ex{x}{\int_0^{\infty} \ea{t}\ind{H}(X_t)dt\ \ind{\{\hit{S}<\infty\}} } \\
		&=\Ex{x}{\int_0^{\hit{S}} \ea{t}\ind{H}(X_t)dt \ \ind{\{\hit{S}<\infty\}}} \\
		& \qquad +\Ex{x}{\int_{\hit{S}}^{\infty} \ea{t}\ind{H}(X_t)dt \ \ind{\{\hit{S}<\infty\}}},
	\end{align*}
	where the second equality holds, because if $\hit{S}$ is infinite, then $X_t$ does not hit $S$, therefore, $\ind{H}(X_t)=0$  for all $t$. In the third equality we simply split the integral in two parts. 
Note that the first term on the right-hand side of the previous equality vanishes, since $\ind{H}(X_t)$ is $0$ when $X_t$ is out of $S$ for the previously exposed argument. It remains to be proven that the second term is equal to $\Ex{x}{\ea{\hit{S}} \Ga(X_{\hit{S}},H )}$; the following chain of equalities completes the proof:
	\begin{align*}
		\Ex{x}{\int_{\hit{S}}^{\infty} \ea{t}\ind{H}(X_t)dt}
		&=\Ex{x}{e^{-\desc \hit{S}}\int_{0}^{\infty} \ea{t}\ind{H}(X_{t+\hit{S}})dt }\\ 
		&=\Ex{x}{\ea{\hit{S}} \Ex{x}{ \int_{0}^{\infty} \ea{t}\ind{H}(X_{t+\hit{S}})dt \big| \F_{\hit{S}}}}\\ 
		&=\Ex{x}{\ea{\hit{S}} \Ex{X_{\hit{S}}}{\int_{0}^{\infty} \ea{t}\ind{H}(X_{t})dt}}\\ 
		&= \Ex{x}{\ea{\hit{S}} \Ga(X_{\hit{S}},H)};
	\end{align*}
the first equality is a change of variable; in the second one, we take the conditional expectation into the expected value, and consider that $\hit{S}$ is measurable with respect to $\F_{\hit{S}}$; the third equality is a consequence of the strong Markov property; while in the last one, we use the definition of $\Ga$.

\end{proof}

As well as the previous lemma, the following one also considers standard Markov processes, even without continuous sample paths.

\begin{lem}
\label{lem:generalMin}
Let $X$ be a standard Markov process. Given $f\colon \estados \to \R$ and $S\in\sigalg$ such that
\begin{itemize}
\item $f$ is a $\sigalg$-measurable function, and
\item for all $x$ in $\estados$, $\int_{S} |f(y)| \Ga(x,dy)<\infty$;
\end{itemize}  
denote by $F_S \colon\estados \to \R$ the function 
\begin{equation*}
\label{eq:defF}
F_S(x):=\int_{S} f(y) \Ga(x,dy).
\end{equation*}
Then
\begin{equation*}
F_S(x)=\Ex{x}{\ea{\hit{S}}F_S(X_{\hit{S}})}.
\end{equation*}
\end{lem}
\begin{proof}
From \autoref{harmonic} we get that
$$F_S(x)=\int_{S} f(y) \Ex{x}{e^{-\desc \hit{S}} \Ga(X_{\hit{S}},dy)}.$$
Changing the integration sign with the expected value on the right-hand side of the equation, we complete the proof.

\end{proof}

%
%The following lemma is, in some way, equivalent to \autoref{lem:generalMin}. The difference is that now we consider processes with Green function and the definition of the function $F_B$ is slightly different.
%\begin{lem}
%\label{lem:Waconmedida}
%Consider $X$ a standard Markov process with Green function. Let $F_B \colon\estados \to \R$ be defined by
%\begin{equation*}
%F_B=\int_{B} \Ga(x,y)\mu(dy),
%\end{equation*}
%where $B$ is an open set and $\mu$ is a measure over $\sigalg$.
%Then
%\begin{equation*}
%F_B(x)=\Ex{x}{e^{-\desc \hit{B}}F_B(X_{\hit{B}})}
%\end{equation*}
%\end{lem}
%\begin{proof}
%\end{proof}

\section{One-dimensional diffusions}
\label{sec:one-dim-dif}
In this section we introduce the concept of \emph{one-dimensional (time-homogeneous) diffusion}, a very important sub-class of the strong Markov processes already presented. 
In a first approach, we may say that one-dimensional diffusions are strong Markov processes with continuous sample paths, whose state space is included in the real line. 
The aim of this section is to give some necessary results for chapters 2 and 3. 
Our main references in this topic are: \cite{borodin,itoMcKean}. 

To give a formal definition of one-dimensional diffusions we start by considering 
the set $\Omega$ of all functions $\omega\colon [0,\infty)\to \I$, 
where $\I$, to be the state space of the diffusion, is an interval of $\R$. 
We denote by $\ell$ and $r$ the infimum and supremum of $\I$ 
($\ell$ could be $-\infty$ and $r$ could be $\infty$).
For every non-negative $t$, consider the shift operator $\theta_t\colon \Omega \to \Omega$ such that
\bd
\theta_t\omega(s)=\omega(t+s).
\ed
Let $\F$ be the smallest $\sigma$-algebra over $\Omega$ such that the coordinate mappings $(\omega \mapsto \omega(t))$ are measurable for all $t\in [0,+\infty)$, and let $\{\F_t:t\geq 0\}$ be the filtration of $\F$ defined by
\bd
\F_t:=\sigma\left(\omega(s)\colon s\leq t\right).
\ed
Given a stopping time $\tau$ with respect to $\{\F_{t^+}\}_{t\geq 0}$, we denote by $\F_{\tau^+}$ the $\sigma$-algebra
\begin{equation*}
 \F_{\tau^+}:=\{A \in \F\colon \forall t> 0, A\cap\{\omega:\tau(\omega)< t\} \in \F_t \}.
\end{equation*}
Consider a family $\{\P_x\colon x\in\I\}$ of probability measures over $(\Omega,\F)$ such that for every $A\in\F$, the map $x\mapsto \P_x(A)$ is Borel-measurable, and  $\P_x(\omega(0)=x)=1$ for every $x$ in $\I$.

Over $\Omega$ consider the process $\{X_t\}_{t\geq 0}$ such that $X_t(\omega)=\omega(t)$. We say that $X=(X_t,\F_t,\P_x)$ is a \emph{one-dimensional diffusion} if:
\begin{itemize} 
\item for all $x$ in $\I$, the map $t\mapsto X_t(\omega)$ is continuous $\P_x$-$a.s$; and 
\item
for all $x$ in $\I$, for all $\{\F_{t^+}\}$-stopping time $\tau$, and for all $B\in \F$
\bd 
      \P_x(X_{\tau+s}\in B|\F_{\tau^+})=\P_{x_\tau}(X_s \in B) \quad (\P_x-a.s.).
\ed
\end{itemize}

A one-dimensional diffusion $X$ is said to be \emph{regular} \cite[see][vol 2, p. 121]{dynkin:books} if for all $x,y \in \I$:
\bd
\P_x(\hit{y}<\infty) > 0,
\ed
where $\hit{y}:=\inf\{t\geq 0\colon X_t=y\}$ is the hitting time of level $y$. All the one-dimensional diffusions considered in this work are assumed to be regular.

The most important example of one-dimensional diffusion is the well-known Wiener process --or Brownian motion-- and more generally the class of time-homogeneous Itô diffusions, which are solutions of stochastic differential equations of the form
\bd
dX_t=b(X_t)dt+\sigma(X_t)dB_t.
\ed
In fact, the term ``diffusion" in some literature is used to refer It\^o diffusions. Some basic references about It\^o diffusions are: \cite{oksendal, ks,ikeda1989stochastic}. Some applications to optimal stopping are also included in \cite{oksendal}.

\subsection{Differential and resolvent operators}\label{section:definitions}

%In the present section we introduce the main necessary concepts following mainly 
%\cite{karatzas91Brownianmotionstochasticcalculus}, were more details can be found. 

%\marginpar{agregar cita a Revuz y Yor, en particular aparecen resultados útiles, como que la speed es de radon, y que es positiva en todo abierto}

%\subsection{Infinitesimal generator and resolvent}
Let us denote by $\CC_b(\I)$ the family of all continuous and bounded functions $f\colon\I\to \R$.

The \emph{infinitesimal generator} of a diffusion $X$ is the operator $L$ defined by
\begin{equation}\label{eq:l}
Lf(x)=\lim_{h \to 0}\frac{ \Ex{x}{f(X_h)}-f(x)}{ h},
\end{equation}
applied to  the functions $f\in\CC_b(\I)$ for which the limit exists pointwise, is in $\CC_b(\I)$, and
\bd
 \sup_{h>0} \frac{\|\Ex{x}{f(X_h)-f(x)}\|}{ h}<\infty.
\ed
Let $\D_L$ denote this set of functions. The extended infinitesimal generator, already defined for more general Markov processes, is, in fact, an extension of $L$ \citep[see][chapter VII, section 1]{revuz}.

Consider the \emph{resolvent operator}, as was defined for Markov processes in \eqref{eq:resolvent}, with domain restricted to $\CC_b(\I)$. The range of the operator $\Ra$ is independent of  $\desc>0$ and coincides with the domain of the infinitesimal generator $\D_L$. Moreover, for any $f\in \D_L$, $\Ra \al f = f$ and for any $u\in \CC_b(\I)$, $\al \Ra u = u;$ in other words $\Ra$ and $\al$ are inverse operators.

Since $f=\Ra \al f$ we can apply the Dynkin's formula \eqref{eq:dynFormMarkov} to $f\in \D_L$; obtaining that for any stopping time $\tau$ and for any $f\in \D_L$:
\begin{equation}
\label{eq:dynkinFormula}
 f(x)=\Ex{x}{ \int_0^{\tau} \ea{t} \al f(X_t) dt}+ \Ex{x}{\ea{\tau}f(X_\tau)}.
\end{equation}

Denoting by $s$ and $m $ the scale function and the speed measure of the diffusion $X$ respectively, 
and by
\begin{equation*}
 \frac{\partial^+ f}{\partial s}(x)=\lim_{h \to 0^+} \frac{g(x+h)-g(x)}{s(x+h)-s(x)},
\quad 
\frac{\partial^- f}{\partial s}(x)=\lim_{h \to 0^+} \frac{g(x-h)-g(x)}{s(x-h)-s(x)},
\end{equation*}
the right and left derivatives of $f$ with respect to $s$; we have that for any $f \in \D_L$,
the lateral derivatives with respect to the scale function exist for every $x\in (\ell,r)$.
Furthermore, they satisfy 
\begin{equation}\label{eq:atom} 
\frac{\partial^+ f}{\partial s}(x)- \frac{\partial^- f}{\partial s}(x)=  m(\{x\}) Lf(x),
\end{equation}
and the following identity holds:
\begin{equation}
\label{eq:diferentialop}
 \frac{\partial^+ f}{\partial s}(z)-\frac{\partial^+ f}{\partial s}(y)=\int_{(y,z]} Lg(x) m(dx).
\end{equation}
%that, in the shorter notation, can be written as $Lf=\frac{\partial}{\partial m} \frac{\partial^+}{\partial s}f$.

This last formula allows us to define the \emph{differential operator} of the diffusion for a given function $u$ 
(not necessarily in $\D_L$) at $x\in\I$ by the equation
\begin{equation*}
Lf(x)=\frac{\partial}{\partial m} \frac{\partial^+}{\partial s}f(x),
\end{equation*}
when this expression exists, that extends the infinitesimal generator $L$.

%It is useful to extend the meaning of $L$ in \eqref{eq:l} even if a function $g$ is not in the domain $\D_L$. 
%In this case, when the 
%we give to $L$, if it make sense, the meaning given in equations \eqref{eq:diferentialop}; and we call it \emph{differential operator}. 

There exist two continuous functions $\phia \colon \I \to \R^+$ and $\psia\colon \I\to \R^+$, such that $\phia$ is decreasing, $\psia$ is increasing, both are solution of $\desc u = Lu$, and any other continuous function $u$ is solution of the differential equation if and only if there exist $a\in\R$ and $b\in \R$ such that $u=a\phia+b\psia$. Denoting by $\hit{z}=\inf \{t\colon X_t = z\}$, the hitting time of level $z\in\I$, we have that
	\begin{equation}
	\label{eq:hitting}
	\Ex{x}{ e^{-\desc \hit{z}}}=
	\begin{cases}
	\frac{\psia(x)}{\psia(z)},\quad x\leq z;\\
	\frac{\phia(x)}{\phia(z)},\quad x\geq z.
	\end{cases}
	\end{equation}
Functions  $\phia$ and $\psia$, though not necessarily in $\D_L$, also satisfy \eqref{eq:atom}
%it is valid %Like for $f\in \D_L$, for
%\begin{equation*} \frac{\partial^+ f}{\partial s}(x)- \frac{\partial^- f}{\partial s}(x)=  m(\{x\}) Lf(x),
%\end{equation*}
for all $x\in (\ell,r),$ which allows us to conclude that in case $m(\{x\})=0$ 
the derivative of both functions with respect to the scale at $x$ exists. 
It is easy to see that
\begin{equation*}
0<\frac{\partial^- \psia}{\partial s}(x) \leq \frac{\partial^+ \psia}{\partial s}(x)<\infty
\end{equation*}
and
\begin{equation*}
-\infty<\frac{\partial^- \phia}{\partial s}(x) \leq \frac{\partial^+ \phia}{\partial s}(x)<0.
\end{equation*}

The \emph{Green function} of the process $X$ with discount factor $\desc$ is defined by
\begin{equation*}
\Ga(x,y):=\int_0^\infty e^{-\desc t}p(t;x,y)dt,
\end{equation*}
where $p(t;x,y)$ is the transition density of the diffusion with respect to the speed measure $m(dx)$
(this density always exists, see \cite{borodin}). 
The Green function may be expressed in terms of $\phia$ and $\psia$ as follows:
\begin{equation}
\label{eq:Garepr}
G_\desc(x,y)=
\begin{cases}
w_\desc ^{-1} \psia(x) \phia (y),\quad  & x\leq y; \\
w_\desc ^{-1} \psia(y) \phia (x),  & x\geq y.
\end{cases}
\end{equation}
where $w_\desc$ --the \emph{Wronskian}-- is given by
\begin{equation*}
 w_\desc=\frac{\partial \psia^+}{\partial s}(x)\phia(x)-\psia(x)\frac{\partial \phia^+}{\partial s}(x)
\end{equation*}
and it is a positive constant independent of $x$. It should be observed that the relation between the Green kernel (already defined for Markov processes) and the Green function is given by
\bd
\Ga(x,dy)=\Ga(x,y)m(dy).
\ed
Then, considering  $f\colon\I\to \R$ under the condition $\int_{\I} \Ga(x,y) |f(y)| m(dy)<\infty$, by applying \eqref{eq:RafGa}, we obtain
% \begin{equation} \label{eqhipotesisfubini} \int_{\I}\Ga(x,y)|f(y)|m(dy)<\infty, \end{equation}
% or
% \begin{equation} \label{eqhipotesisfubini} \Ra |f(y)|m(dy)<\infty, \end{equation}
% for all $x\in\I$, then
\begin{equation}
\label{eq:Ra=Ga}
\Ra f(x)=  \int_{\I} \Ga(x,y) f(y) m(dy).
\end{equation}

%A  non-negative Borel function $u\colon\I\to \R$ is called \emph{$\desc$-excessive} for the process $X$ if
%\begin{itemize}
%\item[1.] $e^{-\desc t}\Ex{x}{u(X_t)}\leq u(x)$ for all $x\in \I$ and $t\geq 0$,
%\item[2.] $\lim_{t \to 0} \Ex{x}{u(X_t)}= u(x)$ for all $x\in \I$.
%\end{itemize}
%A 0-excessive function is said to be simply \emph{excessive}. 
%Considering $\desc$-killed process $Y$ already defined it is easy to see that the green function $G^Y$ of the process $Y$ coincides with $\Ga$; a Borel function $u\colon \I \to \R$ is excessive for $Y$ if and only if it is $\desc$-excessive for $X$.
%In fact, the non-discounted optimal stopping problem for the process $Y$ has the very same solution (value function and stopping time) as the $\desc$-discounted optimal stopping problem for $X$.

%Dynkin's characterization of the solution of the optimal stopping problem for Markov processes states that, under certain conditions,  $V$ is the value function of the non-discounted optimal stopping problem with reward $g$ if and only if $V$ is the least excessive function such that $V(x)\geq g(x)$ for all $x\in \I$. Applying this result for the process $Y$ and taking into account the relation between $X$ and $Y$ we obtain that $\Va$, the value function of the discounted problem with parameter $\desc$ could be characterized as the least $\desc$-excessive majorant of $g$.

We recall Riesz decomposition for excessive functions in our context
(see \cite{kw,kw1,dynkin:1969}). 
Every real valued excessive function $f$ has a unique decomposition in the form
\begin{equation}
\label{eq:excessiverepr}
 f(x)=\int_{(\ell,r)}G(x,y)\mu(dy) + h(x),
\end{equation}
where $h$ is a harmonic function. Morover, for every measure $\mu$ over $\I$ and for every harmonic function $h$, the function $f$ of \eqref{eq:excessiverepr} is excessive. 
%Furthermore, we have that the excessive function $u$ is $A$-harmonic if and only if measure $\mu$ does not charge the set $A$. 

Considering $\desc$-killed process $Y$ already defined, it is easy to see that the Green function $G^Y$ of the process $Y$ coincides with $\Ga$. 
Taking this into account and considering that $\desc$-excessive functions for $X$ are excessive functions for $Y$, %ed process $Y$ %and taking into account the analogy with the discounted problem for $X$ 
we get that any $\desc$-excessive function has a unique decomposition in the form
\begin{equation}
\label{eq:alphaexcessive}
 f(x)=\int_{(\ell,r)}\Ga(x,y)\mu(dy) + h(x),
\end{equation}
where $h$ is an $\desc$-harmonic function; and any function defined by \eqref{eq:alphaexcessive} is $\desc$-excessive. The measure $\mu$ is called the representing measure of $f$.

For general reference on diffusion we recommend: \cite{borodin,itoMcKean,revuz,dynkin:books,ks}.

\chapter{Optimal stopping for one-dimensional diffusions:\\ \textit{the one-sided case}}
\label{chap:diffonesided}
\section{Introduction}

Throughout this chapter, consider a non-terminating and regular one-dimen\-sional diffusion $X=\{X_t\colon t\geq 0\}$ as defined in \autoref{sec:one-dim-dif}. The state space of $X$ is denoted by $\I$ and it is an interval of the real line $\R$, with
left endpoint $\ell=\inf\I$ and right endpoint $r=\sup\I$, where $-\infty\leq\ell<r\leq\infty$.
Denote by $\P_x$ the probability measure associated with $X$ when starting from $x$, 
and by $\E_x$ the corresponding mathematical expectation. 
Denote by $\T$ the set of all stopping times with respect to $\FF$, the natural filtration generated by $X$.

Given a reward function $\g\colon\I\to \R$ and a discount factor $\desc >0$, 
consider the optimal stopping problem consisting in finding a function $\Va$ and a stopping time $\tau^*$ such that
\begin{equation}\label{eq:osp}
\Va(x)=\Ex{x}{\ea{\tau^*} \g(X_{\tau^*})} = \sup_{\tau}\Ex{x}{\ea{\tau} \g(X_{\tau})},
\end{equation}
where the supremum is taken over all stopping times.
The elements $\Va(x)$ and $\tau^*$, the solution to the problem, 
are called the \emph{value function} and the \emph{optimal stopping time} respectively. 

A large number of works considering the optimal stopping of one-dimensional diffusions were developed. See the comprehensive book by \cite{ps} and the references therein. 
%\cite{salminen85,christensen2011harmonic,dayanik}. The article \cite{dayanik} is based in Dayanik's PhD thesis \citep{dayanik2002thesis} and states a characterization of the solution of the OSP in terms of concave functions, departing from the Dynkin's characterization. The recent paper \cite{christensen2011harmonic} provides an approach based on harmonic functions. ?ver

Our approach in addressing this problem firmly adheres to Salminen's work \citep{salminen85} and, as pointed in the general introduction, is based on the Dynkin's characterization of the value function and on the Riesz representation of $\desc$-excessive functions.

To give a quick idea of the kind of results we prove in this chapter, suppose that
\begin{itemize}
\item the speed measure $m$ of the diffusion $X$ has not atoms, and
\item the reward function $g$ satisfies
\begin{equation}
\label{eq:ginversion}
g(x)=\int_{\I} \Ga(x,y) \al g(y) m(dy);
\end{equation}
\end{itemize} 
which is more than we really need in our theorems. We manage to prove the equivalence of the following three assertions.
\begin{enumerate}
\item The problem is right-sided with threshold $x^*$, or what is the same, the optimal stopping time $\tau^*$ is given by
\bd \tau^*=\inf\{t\geq 0\colon X_t\geq x^*\}\ed
\item $x^*$ satisfies
\begin{itemize}
\item[2.a.]\quad $g(x^*)=\int_{x^*}^r \Ga(x^*,y) \al g(y) m(dy)$,
\item[2.b.]\quad $\al g(x)\geq 0$ for all $x\geq x^*$,
\item[2.c.]\quad $\psia(x)\gpsi(x^*) \geq g(x)$ for all $x\leq x^*$.
\end{itemize}
\item The same as 2. but substituting 2.a. by $\gpsi(x^*)=\frac{\partial g}{\partial \psia}(x^*)$.

\end{enumerate}

The main byproduct of the previous result is the following recipe to solve the optimal stopping problem \eqref{eq:osp}, which is effective when the problem is right-sided:
\begin{itemize}
\item
Find the root $x^*$ of either the equation in 2.a. or the equation
\begin{equation*}
\gpsi(x)=\frac{\partial g}{\partial \psia}(x).
\end{equation*}
\item Verify conditions 2.b. and 2.c.
\end{itemize}
If these steps are fulfilled, from our results we conclude that the problem is right-sided with optimal threshold $x^*$. As a direct consequence of this conclusion we obtain that the value function fulfils
\begin{equation*}
 \Va(x)=
\begin{cases}
 \Ex{x}{\ea{\hit{x^*}} g(x^*)},  &x<x^*,\\
 g(x),& x\geq x^*,
\end{cases}
\end{equation*}
being greater or equal that the reward function. In virtue of equation \eqref{eq:hitting}, we obtain
\begin{equation}
\label{eq:VaCall}
 \Va(x)=
\begin{cases}
 \frac{\psia(x)}{\psia(x^*)}g(x^*),  &x<x^*,\\
 g(x),& x\geq x^*.
\end{cases}
\end{equation}

Some of the contents of this chapter are contained in \cite{crocceMordeckiDiff}.

\section{Main results}\label{section:main}

We formulate the hypothesis of the main result, weaker than the inversion formula 
\eqref{eq:ginversion}.

%\noindent \textbf{Condition RRC.}
\begin{cond} \label{cond:RRC}
We say that a function $g\colon\I \to \R$ satisfies the  \emph{right regularity condition} (RRC) for $x^*$ if there exists a function $\gi$ such that $\gi(x)=\g(x)$ for $x\geq x^*$ and 
 %the following inversion formula holds:
\bd
\gi(x) = \int_{\I} \Ga(x,y) \al \gi(y) m(dy)\quad (x\in\I),
\ed
which is \eqref{eq:ginversion} for $\gi$.
\end{cond}

Informally speaking, the RRC is fulfilled by functions $g$ that satisfy all the local conditions --regularity conditions-- to belong to $\D_L$ for $x\geq x^*$, and does not increase as quick as $\psia$ does when approaching $r$. Observe that if $\g$ satisfies the RRC for certain $x^*$ it also satisfies it for any greater threshold; and of course, if $\g$ itself satisfy \eqref{eq:ginversion} then it satisfies the RRC for all $x^*$ in $\I$. 
In \autoref{sec:aboutInv} we discuss conditions that ensures the validity of the inversion formula \eqref{eq:ginversion}.

During this chapter we use the notation $\Ig{a}$ to refer to the set $\I\cap \{x\colon x>a\}$. Symbols $\Ige{a}$, $\Il{a}$, $\Ile{a}$ are used in the same sense.

The equation
\begin{equation}
\label{eq:defx*teo}
g(x)= \wa^{-1} \psia(x)\int_{\Ig{x}} \phia(y)\al g(y) m(dy),
\end{equation}
plays an important role in our approach to solve the optimal stopping problem. Observe that it can be also written in terms of $\Ga(x,y)$ by 
\begin{equation*}
g(x)=\int_{\Ig{x}}\Ga(x,y)\al g(y) m(dy).
\end{equation*}

\begin{lem} \label{lem:Vequalsg1}
Assume $x^*$ is a root of equation \eqref{eq:defx*teo}, and the RRC (\autoref{cond:RRC}) is fulfilled for $x^*$. Then, for all $x \in \Ige{x^*}$, we have that
\begin{equation*}
g(x)=\int_{\Ig{x^*}}\Ga(x,y)\al g(y) m(dy).
\end{equation*}
\end{lem}
\begin{proof}
This lemma is a particular case of the slightly more general \autoref{lem:Vequalsg}.

\end{proof}

We are ready to present the first of our main results previously outlined.

\begin{teo}
\label{teo:verif1} Consider a diffusion $X$ and a reward function $g$. Let $x^*$ be a root of equation \eqref{eq:defx*teo} such that the RRC (\autoref{cond:RRC}) is fulfilled. Suppose
\begin{hip} \label{hyp:algpositive} for all $x>x^*$, $\al g(x)\geq 0$, and
\end{hip}
\begin{hip} \label{hyp:mayorizq} for all $x<x^*$, $\frac{g(x^*)}{\psia(x^*)}\psia(x)\geq g(x)$.
\end{hip}
\noindent Then, the optimal stopping problem \eqref{eq:osp} is right-sided and $x^*$ is an optimal threshold. Furthermore, the value function is $\Va\colon \I \to \R$ defined by
\begin{equation*}
 \Va(x):=\int_{\Ig{x^*}} \Ga(x,y) \al g(y) m(dy).
\end{equation*}
\end{teo}
\begin{proof}
Riesz decomposition, given in \eqref{eq:alphaexcessive}, and the hypothesis \autoref{hyp:algpositive} allow us to conclude that $\Va$ is $\desc$-excessive. Denoting by $\nu(dy)=\al g(y) m(dy)$ for $y\geq x^*$, by \eqref{eq:Garepr}, in the first equality and \eqref{eq:defx*teo} in the second, we have for $x\leq x^*$:
\begin{equation*}
 \Va(x)= \wa^{-1} \psia(x) \int_{\Ig{x^*}}  \phia(y)\nu(dy) = \frac{\psia(x)}{\psia(x^*)}g(x^*).
\end{equation*}
%From \autoref{hyp:mayorizq} we conclude $\Va(x) \geq g(x)$ for $x \leq x^*$. 
Hypothesis \autoref{hyp:mayorizq} renders $\Va(x) \geq g(x)$ for $x \leq x^*$.
Besides, from \autoref{lem:Vequalsg1}, it follows that $\Va(x)=g(x)$ for $x\in \Ige{x^*}$; thereby, $\Va$ is a majorant of $g$, and by Dynkin's characterization of the value function as the minimal $\desc$-excessive majorant of $g$, it follows that
\begin{equation*}
\Va(x) \geq \sup_{\tau} \Ex{x}{ e^{-\desc \tau}\g(X_\tau)}.
\end{equation*}
Observing that $\Va$ satisfies \eqref{eq:VaCall}, we conclude that $\Va$ is the expected discounted reward associated with $\tau^*=\inf\{t:X_t\geq x^*\}$ and 
\begin{equation*}
\Va(x) \leq \sup_{\tau} \Ex{x}{e^{-\desc \tau}\g(x_\tau)}.
\end{equation*}
From both previous inequalities we conclude that the equality holds and the optimal stopping problem is right-sided with threshold $x^*$. Thus completing the proof.

\end{proof}

\begin{remark}
Under the conditions of \autoref{teo:verif1}, if $x^{**}$ is another solution of \eqref{eq:defx*teo} greater that $x^*$, then
\begin{equation*}
 g(x^{**})=\int_{\Ig{x^{**}}}\Ga(x^{**},y) \al g(y)m(dy) = \int_{\Ig{x^*}}\Ga(x^{**},y) \al g(y)m(dy).
\end{equation*}
This means that 
\begin{equation*}
 \int_{(x^*,x^{**}]}\Ga(x,y)\al g(y)m(dy)=0,
\end{equation*}
and the measure $\al g(y)m(dy)$ does not charge $(x^*,x^{**}]$. Actually, every $x\in(x^*,x^{**}]$ is a root of \eqref{eq:defx*teo} and also an optimal threshold. This shows that the ``best candidate'' to be $x^*$ is the largest solution of \eqref{eq:defx*teo}.
\end{remark}

The existence of a solution of equation \eqref{eq:defx*teo} not only provides a solution to the problem, but also implies a certain type of \emph{smooth fit}, discussed in \autoref{section:sf}. 

Even if the optimal stopping problem is right-sided, it could be the case that \eqref{eq:defx*teo} has no solutions. The following result could be useful in these situations. 

\begin{teo}
\label{teo:verif2}
 Consider a diffusion $X$ and a reward function $g$. Consider $x^*$ such that: the RRC (\autoref{cond:RRC}) holds;  the inequality
\begin{equation}
\label{eq:defx*mayor}
 g(x^*) > \wa^{-1} \psia(x^*)\int_{\Ig{x^*}}\phia(y)\al g(y) m(dy),
\end{equation}
is fulfilled; and both hypothesis \autoref{hyp:algpositive} and \autoref{hyp:mayorizq} are valid. Then the optimal stopping problem is right-sided with optimal threshold $x^*$.
Furthermore, $x^*$ is the smallest value satisfying equation \eqref{eq:defx*mayor} and hypothesis \autoref{hyp:algpositive}.
\end{teo}

Before giving the proof we state the following useful lemma.

\begin{lem} \label{lem:Vequalsg}
Consider a diffusion $X$ and reward function $g$. Assume the RRC holds for a given $x^*$ and define $\Va\colon\I\to \R$ by
\begin{equation*}
 \Va(x)=\int_{\Ig{x^*}} \Ga(x,y) \al g(y) m(dy)+ k \Ga(x,x^*),
\end{equation*}
where $k$ is such that $\Va(x^*)=g(x^*)$, i.e.
\begin{equation*}
 k=\frac{g(x^*)-\int_{\Ig{x^*}} \Ga(x^*,y) \al g(y) m(dy)}{\Ga(x^*,x^*)}.
\end{equation*}
Then, for all $x \in \Ige{x^*}$ 
\begin{equation*}
g(x)=\Va(x).
\end{equation*}
\end{lem}
\begin{proof}
Let us compute $\Va(x)-g(x)$ for $x \geq x^*$. We can substitute $g$ by $\gi$, the extension given by the RRC, and use the inversion formula \eqref{eq:ginversion} for $\gi$; denoting by $\sigma(dy)=\al \gi(y) m(dy)$ we get
\begin{equation*}
 \Va(x)=\int_{\Ig{x^*}} \Ga(x,y) \sigma(dy)+ k \Ga(x,x^*),
\end{equation*}
and
\begin{equation*}
 \gi(x)=\int_{\I} \Ga(x,y) \sigma(dy).
\end{equation*}
Computing the difference for $x\in \Ige{x^*}$ we obtain that
\begin{align*}
\Va(x)-\g(x) &= \Va(x)-\gi(x)\\
&= \int_{\Ile{x^*}} \Ga(x,y) \sigma(dy) + k\Ga(x,x^*)\\
&= \wa^{-1} \phia(x) \left(\int_{\Ile{x^*}} \psia(y) \sigma(dy)+k\psia(x^*)\right) \\
&= \frac{\phia(x)}{\phia(x^*)}\left(\Va(x^*)-g(x^*)\right),
\end{align*}
that vanishes due to $\Va(x^*)=g(x^*)$.

\end{proof}

\begin{proof}[Proof of \autoref{teo:verif2}]
The idea is the same as in the proof of \autoref{teo:verif1} but considering $\Va$ as defined in \autoref{lem:Vequalsg}.
It is easy to reproduce that proof to observe that $\Va$ is $\desc$-excessive, is majorant of $g$, and  satisfies \eqref{eq:VaCall}. 

We move on to prove that $x^*$ is the smallest value satisfying equation \eqref{eq:defx*mayor} and hypothesis \autoref{hyp:algpositive}: Suppose that there exists $x^{**}$ such that $x_1<x^{**}<x^*$, satisfying \eqref{eq:defx*mayor} and \autoref{hyp:algpositive}. Let us compute $\Va(x^{**})-\g(x^{**})$ to see it is negative, in contradiction with the fact that $\Va$ is a majorant of $g$. 
Considering the extension $\gi$ of the RRC and calling $\sigma(dy)=\al \gi(y) m(dy)$ we find
\begin{equation*}
\Va(x^{**})-g(x^{**}) = -\int_{\Ile{x^*}}\Ga(x^{**},y)\sigma(dy)+k \Ga(x^{**},x^*).
\end{equation*}
Splitting the integral in $x^{**}$, the first term on the right-hand side of the previous equation is $s_1+s_2$
with
\begin{equation*}
s_1= -\int_{\Ile{x^{**}}}\Ga(x^{**},y)\sigma(dy),\quad \text{and}\quad s_2=-\wa^{-1} \psia(x^{**}) \int_{(x^{**},x^*]} \phia(y) \sigma(dy).
\end{equation*}
To compute the second term observe that
\begin{equation*}
 k=\frac{1}{\Ga(x^*,x^*)}\int_{\Ile{x^*}}\Ga(x^*,y)\sigma(dy)
\end{equation*}
and by \eqref{eq:Garepr} we get $\frac{\Ga(x^{**},x^*)}{\Ga(x^*,x^*)}=\frac{\psia(x^{**})}{\psia(x^*)}$; we obtain that $k \Ga(x^{**},x^*)=s_3+s_4$ with
\bd
s_3=\frac{\psia(x^{**})}{\psia(x^*)}\frac{\phia(x^*)}{\phia(x^{**})} \int_{\Ile{x^{**}}}\Ga(x^{**},y)\sigma(dy)
\ed
and 
\bd
s_4=\frac{\psia(x^{**})}{\psia(x^*)} \int_{(x^{**},x^*]}\Ga(x^*,y)\sigma(dy).
\end{equation*}
Finally observe that
\begin{equation*}
s_1+s_3= \left(\frac{\psia(x^{**})}{\psia(x^*)}\frac{\phia(x^*)}{\phia(x^{**})}-1\right) \int_{\Ile{x^{**}}}\Ga(x^{**},y)\sigma(dy)< 0 
\end{equation*}
because the first factor negative, while the second is positive by the assumption about $x^{**}$, and on the other hand
\begin{equation*} 
s_4+s_2=\frac{\wa^{-1} \psia(x^{**})}{\psia(x^*)} \int_{(x^{**},x^*]}\left(\psi(y)\phia(x^*)-\psia(x^*)\phia(y)\right)\sigma(dy)<0,
\end{equation*}
because the measure is positive by our assumption, and the integrand is non-positive (it is increasing and null in $y=x^*$). We have proved $\Va(x^{**})-g(x^{**})=s_1+s_2+s_3+s_4 <0$, thus completing the proof.

\end{proof}

Although the following result does not give a practical method to solve the optimal stopping problem, 
we find it relevant because it gives conditions under which the previous theorems can be used successfully to find the solution of the problem. Observe that, under the RRC, the following result is a converse result of \autoref{teo:verif1} and \autoref{teo:verif2} together.

\begin{teo} 
\label{teo:recip}
Consider a diffusion $X$ and a reward function $g$ such that the optimal stopping problem is right-sided and $x^*$ is an optimal threshold. Assume $X$ and $g$ satisfy the RRC for $x^*$. Then $x^*$ satisfies 
\be \label{eq:reciproco}
 g(x^*) \geq \wa^{-1} \psia(x^*)\int_{\Ig{x^*}}\phia(y)\al g(y) m(dy),
\ee
(i.e. $x^*$ satisfies either \eqref{eq:defx*teo} or \eqref{eq:defx*mayor}) and both \autoref{hyp:algpositive} and \autoref{hyp:mayorizq} are valid as well.
\end{teo}

\begin{proof}
Since the optimal stopping problem is right-sided we obtain that the value function $\Va$ satisfies \eqref{eq:VaCall}.  We also know that $\Va$ is a majorant of $g$, then \autoref{hyp:mayorizq} clearly holds. To conclude the validity of \autoref{hyp:algpositive}, one needs to observe that $\Ig{x^*}$ is included in the stopping region, and in general $\al g(y)\geq 0$ in that region. 
It only remains to verify \eqref{eq:reciproco}. 
Consider the excessive function $\Wa$ defined by
\begin{equation*}
\Wa(x):=\int_{\Ig{x^*}}\Ga(x,y)\al \g(y) m(dy). 
\end{equation*}
It is easy to see, by using \eqref{eq:Garepr}, that for $x\leq x^*$
\begin{equation*}
\Wa(x)=\psia(x)\frac{\Wa(x^*)}{\psia(x^*)},
\end{equation*}
then, by using \eqref{eq:VaCall}, we get
\be
\label{eq:dif1}
\Va(x)-\Wa(x)=\psia(x)\frac{g(x^*)-\Wa(x^*)}{\psia(x^*)},\quad \mbox{if }x\leq x^*.
\ee
Now consider the region $x\geq x^*$: the equalities $\Va=\g=\gi$ (where $\gi$ is the extension in the RRC) hold in that region, so by using the inversion formula for $\gi$, and also using the explicit formula for $\Ga$ we obtain that
\bd
\label{eq:dif2}
\Va(x)-\Wa(x)= \frac{g(x^*)-\Wa(x^*)}{\phia(x^*)} \phia(x),\quad \mbox{if }x\geq x^*;
\ed
multiplying and dividing by $\wa^{-1} \phia(x^*)$ in the previous equation, and by $\wa^{-1} \psia(x^*)$ in \eqref{eq:dif1}, we find, bearing \eqref{eq:Garepr} in mind,
\begin{equation*}
\Va(x)-\Wa(x)=\frac{\g(x^*)-\Wa(x^*)}{\Ga(x^*,x^*)} \Ga(x,x^*). 
\end{equation*}
Since $\Va$ is an excessive function, Riesz representation \eqref{eq:excessiverepr}  ensures the existence of a representing measure $\mu$ and a harmonic function $h$ such that
\begin{equation*}
 \Va(x)=\int_{(\ell,r)}\Ga(x,y) \mu(dy) + h(x).
\end{equation*}
Suppose that $g(x^*)<\Wa(x^*)$. We obtain that
\begin{align*}
\Wa(x)&=\Va(x)+\frac{\Wa(x^*)-\g(x^*)}{\Ga(x^*,x^*)} \Ga(x,x^*)\\
&=\int_{I}\Ga(x,y)\mu^*(dy) + h(x),
\end{align*}
where the measure $\mu^*=\mu+\frac{\Wa(x^*)-\g(x^*)}{\Ga(x^*,x^*)} \delta_{x^*}(dy)$, and this is absurd because the representing measure of $\Wa$ does not charge $\Ile{x^*}$ and $\mu^*(\{x^*\})>0$. We conclude that $g(x^*)\geq \Wa(x^*)$, that is \eqref{eq:reciproco}.

\end{proof}

The proof of the previous theorem gives a precise description of the Riesz representation of the value function $\Va$ as an $\desc$-excessive function given in equation \eqref{eq:alphaexcessive}.
\begin{prop}
\label{prop:repmeasure}
Assume the conditions of \autoref{teo:recip}.
Then, in the representation of $\Va$ as an $\desc$-excessive function \eqref{eq:alphaexcessive}, the harmonic part vanishes and the representing measure $\mu$ is given by
\bd
\mu(dx)= k \delta_{\{x^*\}}(dx)+\ind{\{x>x^*\}} \al g(x)m(dx),
\ed
with 
%$\nu(dy)$ such that $\nu(dy)=\al \g (y) m(dy),\ y>x^*,$
%$\nu$ does not charge $\Il{x^*}$, i.e. $\nu(\Il{x^*})=0,$ and eventually $\nu$ has an atom in the stopping threshold $x^*$, given by 
\begin{equation*}
\mu(\{x^*\})=k=\frac{\g(x^*)-\int_{\Ig{x^*}}\Ga(x^*,y)\al \g(y) m(dy)}{\Ga(x^*,x^*)}. 
\end{equation*}
Furthermore, if the RRC holds for some $x_1<x^*$ we have that 
\be \label{eq:nuleq}
\mu(\{x^*\})\leq \al g(x^*)m(\{x^*\}).
\ee

\end{prop}
\begin{proof}
 The first claim is already proved in the previous results. We only need to prove the validity of \eqref{eq:nuleq}. Assume it does not hold. Then 
\begin{equation*}
\al g(x^*)m(\{x^*\})<\mu(\{x^*\}) 
\end{equation*}
and 
\begin{equation*}
 \int_{\Ige{x^*}}\Ga(x^*,y) \al g(y) m(dy)<\int_{\Ige{x^*}}\Ga(x^*,y) \mu(dy)=g(x^*).
\end{equation*}
Since $\al g(x^*)>0$ and $\al g$ is a continuous function in $\Ige{x_1}$ we can find $x_2<x^*$ such that \eqref{eq:defx*mayor} and \autoref{hyp:algpositive} are fulfilled, this being in contradiction with \autoref{teo:verif2}.

\end{proof}

\autoref{teo:verif1} and \autoref{teo:verif2} give sufficient conditions for the problem \eqref{eq:osp} to be right-sided. \autoref{teo:recip} shows that these conditions are actually necessary. 
In order to know in a quick view if the problem is right-sided, the following result gives a simpler sufficient condition.

\begin{teo}
Consider a diffusion $X$ and a reward function $g$ such that the inversion formula, given in \eqref{eq:ginversion}, is fulfilled.
Suppose that there exists a root $c\in \I$ of the equation $\al g(x)=0$,
such that $\al g(x)<0$ if $x<c$ and $\al g(x)>0$ if $x>c$. Assume also that $g(x_1)>0$ for some $x_1\in \I$.
Then the optimal stopping problem \eqref{eq:osp} is right-sided and the optimal threshold is 
\begin{equation}
\label{eq:xstarsuficiente}
x^*=\inf\{x\colon b(x)\geq 0\} 
\end{equation}
where $b\colon \I \to \R$ is defined by
\begin{equation*}
b(x)=\int_{\Ile{x}} \psia(y) \al g(y) m(dy).
\end{equation*}
In fact, $b(x)< 0$, if $x<x^*$, and $b(x^*)> 0$ if $x>x^*$.
\end{teo}
\begin{proof}
The idea of the proof is to apply either \autoref{teo:verif1} or \autoref{teo:verif2}, to $x^*$ defined in \eqref{eq:xstarsuficiente}. We first must prove that $x^*$ is well defined.
We start by observing that the set $\{x\colon b(x)\geq 0\}$ is not empty, or what is the same 
\bd \int_{\I}\psia(y)\al g(y)m(dy)>0. \ed
To verify this, one can observe that, taking $k=\frac{\Ga(x_1,c)}{\psia(c)}>0$, we get $k\psia(y)\leq \Ga(x_1,y)$ for $y\leq c$, where $\al g(y)\leq 0$, and $k\psia(y)\geq \Ga(x_1,y)$ for $y\geq x_1$, where $\al g(y)\geq 0$. Then
\bd
 \int_{\I}k\psia(y)\al g(y)m(dy) \geq \int_{\I} \Ga(x_1,y) \al(y) m(dy) = g(x_1) >0.
\ed
By the assumptions on $\al g$ and the fact that $m$ is strictly positive in any open set, we conclude that $b$ is decreasing in $\Il{c}$ and increasing in $\Ig{c}$. Moreover, $b(x)<0$ if $x \leq c$. Since $b$ is right continuous and increasing in $\Ig{c}$ we have that as $\{x\colon b(x)\geq 0\}$ is not empty, it is $\Ig{x^*}$ with $x^*> c$. Observe that, by the inversion formula \eqref{eq:ginversion} and the explicit representation \eqref{eq:Garepr} of $\Ga$, we obtain that
\begin{equation*}
 g(x)=\wa^{-1}\phia(x)b(x)+\int_{\Ig{x}}\Ga(x,y)\al g(y)m(dy).
\end{equation*}
Since $b(x^*)\geq 0$ we get
\begin{equation*}
 g(x^*)\leq \int_{\Ig{x^*}}\Ga(x^*,y)\al g(y)m(dy),
\end{equation*}
and either equation \eqref{eq:defx*teo} or equation \eqref{eq:defx*mayor} is fulfilled. Since $x^*\geq c$, it follows that  $\al g(y)>0$ for $x>x^*$. 
It just remains to be proven that \autoref{hyp:mayorizq} holds. By the definition of $x^*$, there exists a signed measure $\sigma_{\ell}(dy)$ such that: 
\begin{itemize}
\item its support is contained in $\Ile{x^*}$; 
\item $\sigma_{\ell}(dy)=\al g(y) m(dy)$ for $y<x^*$; and 
\item $\int_{\Ile{x^*}}\psia(y)\sigma_{\ell}(dy)=0.$
\end{itemize}
Furthermore, $\sigma_r(dy)=\al g(y)m(dy)-\sigma_{\ell}(dy)$ is a positive measure supported in $\Ige{x^*}$. Using the inversion formula for $g$ and the explicit representation of $\Ga$, given in \eqref{eq:Garepr}, we get for $x<x^*$, that
\begin{equation*}
 g(x)-\frac{\psia(x)}{\psia(x^*)}g(x^*)=\int_{\Ile{x^*}} \Ga(x,y) \sigma_{\ell}(dy)\leq \frac{\Ga(x,c)}{\psia(c)} \int_{\Ile{x^*}}\psia(y)  \sigma_{\ell}(dy) =0,
\end{equation*}
where the inequality follows from the following facts:
if $y<c$ then $\sigma_{\ell}(dy)\leq 0$ and
\begin{equation*}
\psia(y) \frac{\Ga(x,c)}{\psia(c)} \leq \Ga(x,y),
\end{equation*}
while if $y>c$ then $\sigma_{\ell}(dy)\geq 0$
\begin{equation*}
\psia(y) \frac{\Ga(x,c)}{\psia(c)} \geq \Ga(x,y).
\end{equation*}
We can now apply either \autoref{teo:verif1} or \autoref{teo:verif2} to complete the proof.

\end{proof}

\subsection{An alternative equation to find the optimal threshold}
\label{sec:easyequation}
Equation \eqref{eq:defx*teo}, which is key in our approach to solve 
the optimal stopping problem, can be difficult to solve, as it involves the computation of a usually difficult integral. 
In this section we prove that, in order to find the optimal threshold $x^*$, one can solve the 
alternative, and typically much easier, equation
\be
\label{eq:alternativex*}
 \frac{\partial^+{g}}{\partial{\psia}}(x)=\frac{g(x)}{\psia(x)}.
\ee
which in the smooth case, if $g$ and $\psia$ have derivatives, becomes
\be
\label{eq:easyeq}
 \frac{g'(x)}{\psia'(x)}=\frac{g(x)}{\psia(x)}.
\ee

\begin{lem} \label{lem:equiveq}
Assume $X$ and $g$ satisfy the RRC (\autoref{cond:RRC}) for $x^*$. Then, $x^*$ is a root of \eqref{eq:defx*teo} if and only if it is a root of 
\eqref{eq:alternativex*}.
In other words, in the region where $g$ is regular the solutions of \eqref{eq:defx*teo}, coincide with the solutions of \eqref{eq:alternativex*}.
\end{lem}
\begin{proof}% ($\Rightarrow$)
Consider $\Va$ such that
\begin{equation*}
 \Va(x)=\int_{\Ig{x^*}} \Ga(x,y) \al g(y) m(dy)+ k \Ga(x,x^*),
\end{equation*}
where $k$ is such that $\Va(x^*)=g(x^*)$. From \autoref{lem:Vequalsg} it follows that $\Va(x)$ coincides with $g(x)$ for all $x\in \Ige{x^*}$, so we can compute the right derivative of $g$ --with respect to $\psia$-- at $x^*$ by
\begin{equation*}
\frac{\partial^+{g}}{\partial{\psia}}(x) = \lim_{x \to {x^*}^+} \frac{\Va(x)-\Va(x^*)}{\psia(x)-\psia(x^*)}.
\end{equation*}
For $x>x^*$, using $\nu(dy)$ as an abbreviation for $\al g(y) m(dy)$, we obtain that:
\begin{equation*}
\Va(x)=\phia(x) \wa^{-1}\int_{(x^*,x)}\psia(y)\nu(dy)+\psia(x) \wa^{-1}\int_{\Ige{x}}\phia(y)\nu(dy)+k\psia(x^*)\phia(x).
\end{equation*}
Computing the difference between $\Va(x)$ and $\Va(x^*)$ we obtain that
\begin{align*} \Va(x)-\Va(x^*)& =\wa^{-1} \left(\psia(x)-\psia(x^*)\right) \int_{\Ige{x}}\phia(y)\nu(dy)\\
&\qquad +\wa^{-1}\int_{(x^*,x)} 
\left(\phia(x) \psia(y)-\psia(x^*)\phia(y)\right) \nu(dy)\\
&\qquad + k \wa^{-1} \psia(x^*)\left(\phia(x)-\phia(x^*)\right).
\end{align*}
%Let us compute the right derivative of $\Va$ respect to $\psia$:
Dividing by the increment of $\psia$ and taking the limit we get
\begin{align*}
 \lim_{x \to {x^*}^+} \frac{\Va(x)-\Va(x^*)}{\psia(x)-\psia(x^*)} &= \wa^{-1} \int_{\Ig{x^*}}\phia(y)\nu(dy) \\
& \qquad + \wa^{-1} \lim_{x \to {x^*}^+} \frac{\int_{(x^*,x)}\phia(x)\psia(y)-\psia(x^*)\phia(y)\nu(dy)}{\psia(x)-\psia(x^*)}\\
& \qquad + k \wa^{-1} \psia(x^*)\frac{\partial^+\phia}{\partial\psia}(x^*).
\end{align*}
Assume for a moment that the middle term on the right-hand side vanishes; hence the limit 
exists, and
\begin{equation}
\label{eq:dgoverdpsi}
\frac{\partial g^+}{\partial \psia}(x^*)=\wa^{-1} \int_{\Ig{x^*}}\phia(y)\nu(dy)+k \wa^{-1} \psia(x^*)\frac{\partial^+\phia}{\partial\psia}(x^*).
\end{equation}
On the other hand, concerning $\frac{g(x^*)}{\psia(x^*)}$, we have:
\begin{equation}
\label{eq:goverpsi}
\frac{g(x^*)}{\psia(x^*)}=\frac{\Va(x^*)}{\psia(x^*)}=\wa^{-1} \int_{\Ig{x^*}}\phia(y)\nu(dy)+k \wa^{-1} \phia(x^*).
\end{equation}
Comparing the two previous equations, we obtain that $x^*$ satisfy \eqref{eq:alternativex*} if and only if 
\begin{equation*}
 k \wa^{-1} \psia(x^*)\frac{\partial^+\phia}{\partial\psia}(x^*)= k \wa^{-1} \phia(x^*).
\end{equation*}
Since $\phia(x^*)>0$ and $\frac{\partial^+\phia}{\partial\psia}(x^*)<0$ the previous equality holds if and only if $k=0$, and this is equivalent to \eqref{eq:defx*teo}.
This means that we only need to verify
\begin{equation}\label{eq:integral0}
\lim_{x \to {x^*}^+} \int_{(x^*,x)}\frac{\phia(x)\psia(y)-\psia(x^*)\phia(y)}{\psia(x)-\psia(x^*)}\nu(dy)=0,
\end{equation}
to finalize the proof.
Denoting by $f(y)$ the numerator of the integrand in \eqref{eq:integral0}, observe that
\begin{equation*}
 f(y)=\phia(x)(\psia(y)-\psia(x^*))+\psia(x^*)(\phia(x)-\phia(y)).
\end{equation*}
Concerning the first term, we have
\begin{equation*}
0\leq \phia(x)(\psia(y)-\psia(x^*))\leq \phia(x)(\psia(x)-\psia(x^*)),
\end{equation*}
while for the second we have
\begin{equation*}
0\geq \psia(x^*)(\phia(x)-\phia(y)) \geq \psia(x^*)(\phia(x)-\phia(x^*)).
\end{equation*}
These two previous inequalities render
\begin{equation*}
\psia(x^*)(\phia(x)-\phia(x^*))\leq f(y) \leq \phia(x)(\psia(x)-\psia(x^*)).
\end{equation*}
Dividing by $\psia(x)-\psia(x^*)$ in the previous formula, we conclude that the integrand in \eqref{eq:integral0} is bounded from below by
\begin{equation*}
b(x)=\psia(x^*)\frac{\phia(x)-\phia(x^*)}{\psia(x)-\psia(x^*)} = \psia(x^*)\frac{\phia(x)-\phia(x^*)}{s(x)-s(x^*)} \frac{s(x)-s(x^*)}{\psia(x)-\psia(x^*)},
\end{equation*}
and bounded from above by $\phia(x).$ We obtain that the integral in \eqref{eq:integral0} satisfies
\begin{equation*}
 b(x)\nu(x^*,x)\leq  \int_{(x^*,x)}\frac{\phia(x)\psia(y)-\psia(x^*)\phia(y)}{\psia(x)-\psia(x^*)}\nu(dy) \leq \phia(x)\nu(x^*,x).
\end{equation*}
Taking the limits as $x \to {x^*}^+$ we obtain that $\phia(x)\to \phia(x^*)$, $\nu(x,x^*)\to 0$, and
\begin{equation*}
 \lim_{x\to {x^*}^+}b(x)=\psia(x^*)\frac{\partial \phia^+}{\partial s}(x^*)\left/ \frac{\partial \psia^+}{\partial s}(x^*)\right.;
\end{equation*}
hence
\begin{equation*}
 0 \leq  \lim_{x\to {x^*}^+}\int_{(x^*,x)}\frac{\phia(x)\psia(y)-\psia(x^*)\phia(y)}{\psia(x)-\psia(x^*)}\nu(dy) \leq 0,
\end{equation*}
and \eqref{eq:integral0} holds.

\end{proof}

\begin{remark}
 A useful consequence of the previous proof is that, under the RRC for $x^*$, the right derivative of $g$ with respect to $\psia$ exists.
\end{remark}

\begin{lem} \label{lem:equiveq2}
Assume that $X$ and $g$ satisfy the RRC for $x^*$. Then, equation \eqref{eq:defx*mayor} is fulfilled if and only if
\begin{equation}
\label{eq:alternativex*mayor}
 \frac{\partial^+{g}}{\partial{\psia}}(x^*)<\frac{g(x^*)}{\psia(x^*)}.
\end{equation}
\end{lem}
\begin{proof}
Define $\Va$ as in the proof of the previous lemma. By equations \eqref{eq:dgoverdpsi} and \eqref{eq:goverpsi} we conclude that \eqref{eq:alternativex*mayor} holds if and only if
\begin{equation*}
 k \wa^{-1} \psia(x^*)\frac{\partial^+\phia}{\partial\psia}(x^*)< k \wa^{-1} \phia(x^*),
\end{equation*}
if and only if $k>0$, if and only if \eqref{eq:defx*mayor} holds.

\end{proof}

\begin{remark}
Applying \autoref{lem:equiveq} and \autoref{lem:equiveq2}, we obtain results analogous to \autoref{teo:verif1}, \autoref{teo:verif2} and \autoref{teo:recip}, by substituting equations \eqref{eq:defx*teo} and \eqref{eq:defx*mayor} by \eqref{eq:alternativex*} and \eqref{eq:alternativex*mayor} respectively.
\end{remark}

%We use the notation $\ga$ to refer to $\frac{g}{\psia}$

\begin{teo}
 Let $x^*$ be the argument of the maximum of $\gpsi$. Assume that the RRC is fulfilled for $x^*$, and $\al g(x)\geq 0$ for all $x>x^*$. Then, the optimal stopping problem \eqref{eq:osp} is right-sided and $x^*$ is an optimal threshold.
\end{teo}
\begin{proof}
 The idea is to prove that either \autoref{teo:verif1} or \autoref{teo:verif2} are applicable. 
 The validity of \autoref{hyp:mayorizq} is a direct consequence of the definition of $x^*$. 
 We need to verify that $x^*$ satisfies \eqref{eq:defx*teo} or \eqref{eq:defx*mayor}, or, what is equivalent, that
\begin{equation*}
 g(x^*)\geq \wa^{-1}\psia(x^*)\int_{\Ig{x^*}}\phia(y)\al g(y)m(dy).
\end{equation*}
By \autoref{lem:equiveq} and \autoref{lem:equiveq2} we get that this is equivalent to the condition
\begin{equation*}
 \frac{\partial^+g}{\partial \psia}(x^*)\leq \gpsi(x^*).
\end{equation*}
Suppose that the previous condition does not hold, i.e.
\begin{equation*}
 \lim_{x\to x^{*+}}\frac{g(x)-g(x^*)}{\psia(x)-\psia(x^*)}>\frac{g(x^*)}{\psia(x^*)},
\end{equation*}
and consider $x_2>x^*$ such that 
\begin{equation*}
 \frac{g(x_2)-g(x^*)}{\psia(x_2)-\psia(x^*)}>\frac{g(x^*)}{\psia(x^*)};
\end{equation*}
doing computations we conclude that $\gpsi(x_2)>\gpsi(x^*)$, which contradicts the hypothesis on $x^*$. We conclude, by the application of \autoref{teo:verif1} if \eqref{eq:alternativex*}, and by the application of \autoref{teo:verif2} if \eqref{eq:alternativex*mayor}, that the problem is right-sided with threshold $x^*$.

\end{proof}

%In the following theorem we give a converse result of the previous one.
The following is a converse of the previous result:
\begin{teo}
Consider a diffusion $X$ and a reward function $g$, such that the optimal stopping problem \eqref{eq:osp} is right-sided with optimal threshold $x^*$. Then, $\gpsi$ takes its maximum value at $x^*$. Furthermore, if the RRC is fulfilled, then $\al g(x) \geq 0$ for all $x > x^*$.
\end{teo}
\begin{proof}
We have already seen that $\al g\geq 0$ in the stopping region. Let us see that $\gpsi(x^*)$ is maximum. For $x<x^*$ we have, by equation \eqref{eq:VaCall}, that
\begin{equation*}
 \Va(x)=\psia(x)\gpsi(x^*)\geq g(x),
\end{equation*}
therefore, $\gpsi(x)\leq \gpsi(x^*)$. Suppose there exists $x_2>x^*$ such that $\gpsi(x_2)>\gpsi(x^*)$, we would conclude, by equation \eqref{eq:hitting}, that
\begin{equation*}
 \E_{x^*}\left(e^{-\desc \hit{x_2}}g(X_{\hit{x_2}})\right)=\frac{\psia(x^*)}{\psia(x_2)}g(x_2)>g(x^*),
\end{equation*}
what would be in contradiction with the fact of $x^*$ being the optimal threshold.

\end{proof}

\begin{remark}
 Consider the optimal stopping problem \eqref{eq:osp} but taking the supremum over the more restrictive class of hitting times of sets of the form $\Ige{z}$. Departing from $\ell$, for example, we have to solve
\begin{equation*}
 \sup_{z\in \I}\Ex{\ell}{ \ea{\hit{z}}g(X_{\hit{z}})}=\sup_{z\in \I} \frac{\psia(\ell)}{\psia(z)}g(z).
\end{equation*}
In conclusion we need to find the supremum of $\gpsi(z)$. Previous results give conditions under which the solution of the problem \eqref{eq:osp} is the solution of the problem in the restricted family of stopping times.
\end{remark}

\subsection{About the inversion formula}
\label{sec:aboutInv}
In order to apply the previous results we need to know whether a function $g$ satisfies equation \eqref{eq:ginversion}. We remember the equation
\begin{equation} \tag{\ref{eq:ginversion}} \label{eq:ginversionReciente}
g(x)=\int_{\I} \Ga(x,y) \al g(y) m(dy).
\end{equation}

As we have seen in the preliminaries, if $g\in \D_L$ we have that $\Ra \al g = g$ and, if equation \eqref{eq:Ra=Ga} is valid  for $\al g$, we then get \eqref{eq:ginversionReciente}. 
This is the content of the following result:
\begin{lem}
 If $g\in \D_L$ and for all $x\in \I$ we have that 
\begin{equation*}
 \int_{\I} \Ga(x,y) |\al g(y)| m(dy) < \infty;
\end{equation*}
then, for all $x \in \I$, equation \eqref{eq:ginversionReciente} holds.
%\begin{equation*}
%g(x)=\int_{\I} \Ga(x,y) \al g(y) m(dy);
%\end{equation*}
\end{lem}
The condition of the previous lemma is very restrictive in solving concrete problems, 
because typically, we have that $\lim_{x\to r^-}g(x)=\infty$ (in this case $r\notin\I$).
The following result extends the previous one to unbounded reward functions. 
It states, essentially, that if a function $g$ satisfies all the local conditions to belong to $\D_L$, and its only ``problem'' is its behaviour 
%that it is not well-behaved in 
at $r^-$ (typically $r^-=+\infty$), then it is sufficient to verify the condition $\gpsi(x)\to 0$  ($x\to r$), to ensure the validity of \eqref{eq:ginversionReciente}.

\begin{prop} \label{propInversion2}
 Suppose that $r\notin\I$ and $g\colon\I\to \R$ is such that: the differential operator is defined for all $x\in\I$;
 %($g$ not necessarily in $\D_L$) 
 and
\be \label{eq:alIntegrable}
 \int_{\I}\Ga(x,y) |\al g(y)| m(dy)<\infty.
\ee
Assume that, for each natural number $n$ that satisfies $r-\frac1n\in \I$, there exists a function $g_n \in \D_L$ such that $g_n(x)=g(x)$ for all $x\leq r-\frac1n$. If 
\be \label{eq:gOverPsi}
 \lim_{z\to r^-} \frac{g(z)}{\psia(z)}=0,
\ee
then \eqref{eq:ginversionReciente} holds.
\end{prop}

\begin{proof}
By \eqref{eq:Ra=Ga} we get that $\int_{\I} \Ga(x,y) \al g(y) m(dy)=\Ra \al g(x)$. 
Consider the strictly increasing sequence $r_n:=r-\frac1{n-1}$.
%$(r_n)\subset \I$ such that $r_n \to r$ when $n \to \infty$; 
By the continuity of the paths we conclude that $\hit{r_n}\to \infty,\ (n\to\infty)$. 
Applying formula \eqref{eq:dynkinFormula} to $g_n$ and $\hit{r_n}$ we obtain, for $x<r_n$, that
\bd
g_n(x)=\Ex{x}{\int_0^{\hit{r_n}} \ea{t} \al g_n(X_t) dt}+ \Ex{x}{\ea{\hit{r_n}}g_n(r_n)}.
\ed
By using $g_n(x)=g(x)$ and $\al g(x)=\al g_n(x)$ for $x<r_{n+1}$ we conclude that
\be \label{eq:paraconvdominada}
g(x)=\Ex{x}{\int_0^{\hit{r_n}} \ea{t} \al g(X_t) dt}+ \Ex{x}{\ea{\hit{r_n}}g(r_n)}.
\ee
Taking the limit as $n\to \infty$, by  \eqref{eq:hitting} and \eqref{eq:gOverPsi} we obtain that:
\bd
 \Ex{x}{\ea{\hit{r_n}}}g(r_n) = \frac{\psia(x)}{\psia(r_n)}g(r_n)\to 0.
\ed
Let us verify we are can apply the Lebesgue dominated convergence theorem to compute the limit of the first term on the right-hand side of \eqref{eq:paraconvdominada}:
\begin{align*}
 \left|\int_0^{\hit{r_n}} \ea{t} \al g(X_t) dt\right|& \leq\int_0^{\hit{r_n}} \ea{t} |\al g(X_t)| dt \\
& \leq \int_0^\infty \ea{t} |\al g(X_t)| dt
\end{align*}
and by Fubini's Theorem, and hypothesis \eqref{eq:alIntegrable},
\begin{align*}
\Ex{x}{\int_0^\infty \ea{t} |\al g(X_t)|} &= \int_0^\infty \ea{t} \Ex{x}{|\al g(X_t)|} dt \\
&= \int_{\I} \Ga(x,y) |\al g(y)| m(dy)<\infty.
\end{align*}
Taking the limit into the expected it follows that
\bd
\Ex{x}{\int_0^{\hit{r_n}} \ea{t} \al g(X_t) dt} \to \int_0^\infty \Ex{x}{ \ea{t} \al g(X_t)} dt\quad (n\to \infty).
 % \int_0^{\hit{r_n}} \ea{t} |\al g(X_t)| dt \to \int_0^\infty \Ex{x}{ \ea{t} \al g(X_t)} dt\quad (n\to \infty). esto decía pero debe estar mal
\ed
We have obtained:
\begin{align*}
 g(x)&=\int_0^\infty \Ex{x}{ \ea{t} \al g(X_t)} dt
\\&= \int_{\I} \Ga(x,y) \al g(y) m(dy),
\end{align*}
thus completing the proof.

\end{proof}

The following result is analogous to \autoref{propInversion2}, considering the case in which the function $g$ does not belong to $\D_L$ due to its behaviour on the left endpoint of $\I$.

\begin{prop} \label{propInversion3}
 Suppose that $\ell \notin\I$ and $g\colon\I\to \R$ is such that: the differential operator is defined for all $x\in\I$;
 %($g$ not necessarily in $\D_L$) 
 and
\bd
 \int_{\I}\Ga(x,y) |\al g(y)| m(dy)<\infty.
\ed
Assume that, for each natural number $n$ that satisfies $\ell+\frac1n\in \I$, there exists a function $g_n \in \D_L$ such that $g_n(x)=g(x)$ for all $x\geq \ell+\frac1n$. If 
\bd
 \lim_{z\to \ell^+} \frac{g(z)}{\phia(z)}=0,
\ed
then \eqref{eq:ginversionReciente} holds.
\end{prop}

\section{On the Smooth fit principle}\label{section:sf}

The well-known \emph{smooth fit principle} \sfp \ states that under some conditions at the critical threshold $x^*$  the condition $V'(x^*)=g'(x^*)$ is satisfied. 
This principle was used for the first time by \cite{mikhalevish1958Ukranian} and constitutes a widely used method to find the limit between the stopping and the continuation region.
Many works studying the validity of this principle in different situation were developed. Some of them are: \cite{grigelionis1966stefan,brekkehigh,chernoff1968optimal,alili-kyp,dayanik2002thesis,dayanik,villeneuve,christensen2009note}.  
In \cite{salminen85}, the author proposes an alternative version of this principle, considering derivatives with respect to the scale function. 
We say that there is \emph{scale smooth fit} \ssfp \ when the value function has derivative at $x^*$ with respect to the scale function; note that if $g$ also has derivative with respect to the scale function, they coincide, as $g=\Va$ in $\Ige{x^*}$. 
%$$\frac{\partial V}{\partial s}(x^*)=\frac{\partial g}{\partial s}(x^*),$$
%where $\partial s$ stands for the derivative with respect to the scale function $s$.
This alternative smooth fit principle happens to be more general than the \sfp\ in the sense that it needs less regularity conditions in order to hold.

In this work we consider another smooth fit principle considering derivatives with respect to $\psia$. 
We say that there is \emph{smooth fit with respect to $\psia$}, and we denote \psfp, if the value function has derivative with respect to $\psia$, i.e. if the following limit exists,
\begin{equation}
\frac{\partial \Va}{\partial \psia}(x^*)=\lim_{h \to 0}\frac{\Va(x^*+h)-\Va(x^*)}{\psia(x^*+h)-\psia(x^*)}.
\end{equation}

As the scale function is an increasing solution to the equation $\al f=0$ in case $\desc=0$, 
the \ssfp\ can be considered as an $0$-SF, obtaining then a generalization of Salminen's proposal
(although in \cite{salminen85} the \ssfp\ is considered also in case $\desc>0$). 
In what follows we discuss in detail conditions in order to the different smooth fit principles hold.

In the previous section we give a solution to the optimal stopping problem \eqref{eq:osp} when it is right-sided and the RRC is satisfied. We have seen that the Riesz representation of the value function is
\begin{equation*}
\Va(x)=\int_{\Ige{x^*}}\Ga(x,y) \nu(y)
\end{equation*}
where $\nu(dy)=\al g(y) m(dy)$ in $\Ig{x^*}$. In the particular case in which $\nu(\{x^*\})=0$ 
have the following result:
\begin{teo} \label{teosf}
Given a diffusion $X$, and a reward function $g$, if the optimal stopping problem is right-sided and the value function satisfies 
\begin{equation*}
 \Va(x)=\int_{\Ig{x^*}}\Ga(x,y) \al g(y) m(dy)
\end{equation*}
then $\Va$ is differentiable in $x^*$ with respect to $\psia$, i.e.
\begin{equation*}
 \nu(\{x^*\})=0\quad\Rightarrow\quad \text{\rm \psfp}.
\end{equation*}
%\marginpar{capaz que alcanza con $\nu(x^*)=0$, de paso, habría que ver que $\al g m(dy)$ es de radon}
\end{teo}
 \begin{proof}
 For $x\leq x^*$
 \begin{equation*}
 \Va(x)=\wa^{-1} \psia(x) \int_{\Ig{x^*}}\phia(y) \nu(dy), 
 \end{equation*}
 and the left derivative of $\Va$ with respect to $\psia$ in $x^*$ is
 \begin{equation*}
 \frac{\partial \Va^-}{\partial \psia}(x^*)=\wa^{-1} \int_{\Ig{x^*}}\phia(y) \nu(dy). 
 \end{equation*}

 It still remains to verify that
\begin{equation*}
 \lim_{x\to {x^*}^+}\frac{V(x)-V(x^*)}{\psia(x)-\psia(x^*)}=\wa^{-1} \int_{\Ig{x^*}}\phia(y) \nu(dy);
\end{equation*}
what can be done by the same computation as in the proof of \autoref{lem:equiveq}, for $k=0$.

\end{proof}

As we have seen in the previous section, if the speed measure does not charge $x^*$ neither does the representing measure, i.e.
\begin{equation*}
 m(\{x^*\})=0 \quad \Rightarrow \quad \nu(\{x^*\})=0;
\end{equation*}
then the previous theorem is applicable to this case to conclude that there is \psfp. Next theorem states that in this case there is \ssfp.
\begin{cor}
\label{scalesf}
Under the same conditions as in the previous theorem, if the speed measure does not charge $x^*$, then there is scale smooth fit, i.e.
\begin{equation*}
 m(\{x^*\})=0\quad\Rightarrow\quad \text{\rm \ssfp}.
\end{equation*} 
\begin{proof}
From the previous theorem  it follows that there is \psfp. Hypothesis $m(\{x^*\})=0$ implies $\psia$ has derivative with respect to the scale function. We have obtained that
\begin{equation*}
 \frac{\partial \Va}{\partial s}(x^*)=\frac{\partial \Va}{\partial \psia}(x^*) \left/\frac{\partial \psia}{\partial s}(x^*)\right.;
\end{equation*}
thus competing the proof.

\end{proof}
\end{cor}
Next corollary, which is a direct consequence from the previous results, gives sufficient conditions in order to the \sfp\ hold.
\begin{cor}
 \label{clasicsf}
Suppose either, \begin{itemize}
\item the hypotheses of \autoref{teosf} are fulfilled and $\psia$ is differentiable at $x^*$, or 
\item the hypotheses of \autoref{scalesf} are fulfilled and $s$ is differentiable at $x^*$.
\end{itemize}
Then, if the reward function $g$ is differentiable at $x^*$ the classic smooth fit principle holds.
\end{cor}

\section{Examples}\label{section:examples}
We follow by showing how to solve some optimal stopping problems using the previous results. We start with some well-known problems, whose solutions are also known, and we include as well some new examples in which the smooth fit principle is not useful to find the solution.

\subsection{Standard Brownian motion}
\label{ex:Brownian}
Through this subsection we consider $X$ such that $X_t := W_t$, the standard Wiener process \citep[see][p.119]{borodin}. The state space is $\I=\R$, the scale function is $s(x)=x$ and the speed measure is $m(dx)=2dx$. The differential operator $L$ is given $Lf=f''/2$, being its domain  
\begin{equation*}
\D_L=\{f: f,\ Lf \in \CC_b(\R)\}. 
\end{equation*} 
The fundamental solutions of $\al u=0$ are $\phia(x)=e^{-\sqrt{2 \desc}x}$ and $\psia(x)=e^{\sqrt{2 \desc}x}.$ The Wronskian is $\wa=2\sqrt{2\desc}$ and, according to \eqref{eq:Garepr},
the Green function is given by
\begin{equation*}
\Ga(x,y)=
\begin{cases}
\frac{1}{2\sqrt{2 \desc}}e^{-\sqrt{2\desc}\,x}e^{\sqrt{2\desc}\,y}, & x\geq y; \\[.5em]
\frac{1}{2\sqrt{2 \desc}}e^{-\sqrt{2\desc}\,y}e^{\sqrt{2\desc}\,x}, & x< y.
\end{cases} 
\end{equation*}

\begin{example}
Consider the reward function $\g(x):=x^+$, where $x^+$ states for the maximum between $x$ and $0$. 
The OSP for $X$ and $g$ was solved for the first time by \cite{taylor}. 
It is easy to verify that the RRC is fulfilled for $x_1=0$: it is enough to consider $\gi \in \CC^2(\R)$ such that $\gi(x)=0$ for $x<-1$ and $\gi(x)=x$ for $x>0$ and, by the application of \autoref{propInversion2}, one can prove that $\gi$ satisfies the inversion formula \eqref{eq:ginversion}.
Now we check that we can apply \autoref{teo:verif1} to find the solution of the optimal stopping problem: 
Observe that $\al g$ is given by
\begin{equation*}
\al g(x)=\desc x, \quad \mbox{for $x>0$.}
\end{equation*}
To find $x^*$ we solve equation \eqref{eq:defx*teo}, which is
\begin{align*}
x^* &=\frac{1}{2\sqrt{2 \desc}} e^{\sqrt{2\desc}\,x^*} \int_{x^*}^\infty e^{-\sqrt{2\desc}\,y}\desc y\,2dy\\ 
&=\frac{1}{2 \sqrt{2\desc}}(x^* \sqrt{2 \desc} +1 ),
\end{align*}
obtaining that $x^*=\frac{1}{\sqrt{2\desc}}$. The conditions \autoref{hyp:algpositive} and \autoref{hyp:mayorizq} are easy to verify. We conclude that the problem is right-sided and $x^*$ is the optimal threshold. Observe that the hypotheses of  \autoref{teosf},  \autoref{scalesf} and \autoref{clasicsf} are fulfilled; then, all variants of smooth fit principle hold in this example.
 
Considering the results in \autoref{sec:easyequation}, we can see that the simpler equation \eqref{eq:easyeq} could be used to find $x^*$. For $x>0$ we have $g'(x)=1$ and $\psia'(x)=\sqrt{2\desc}\, e^{\sqrt{2\desc}}$, so in this particular case, equation \eqref{eq:easyeq} is
\[
\frac{1}{\sqrt{2\desc}\, e^{\sqrt{2\desc}}}=\frac{x}{e^{\sqrt{2\desc}}}
\]
obtaining the same solution.

According to equation \eqref{eq:VaCall} the value function is 
\bd
\Va(x)=\begin{cases}
\frac{e^{\sqrt{2\desc}x}}{e\sqrt{2\desc}},  & x < \frac{1}{\sqrt{2\desc}},\\
x, & x\geq \frac{1}{\sqrt{2\desc}}.
\end{cases}
\ed
\begin{figure}
\centering
\includegraphics[scale=.8]{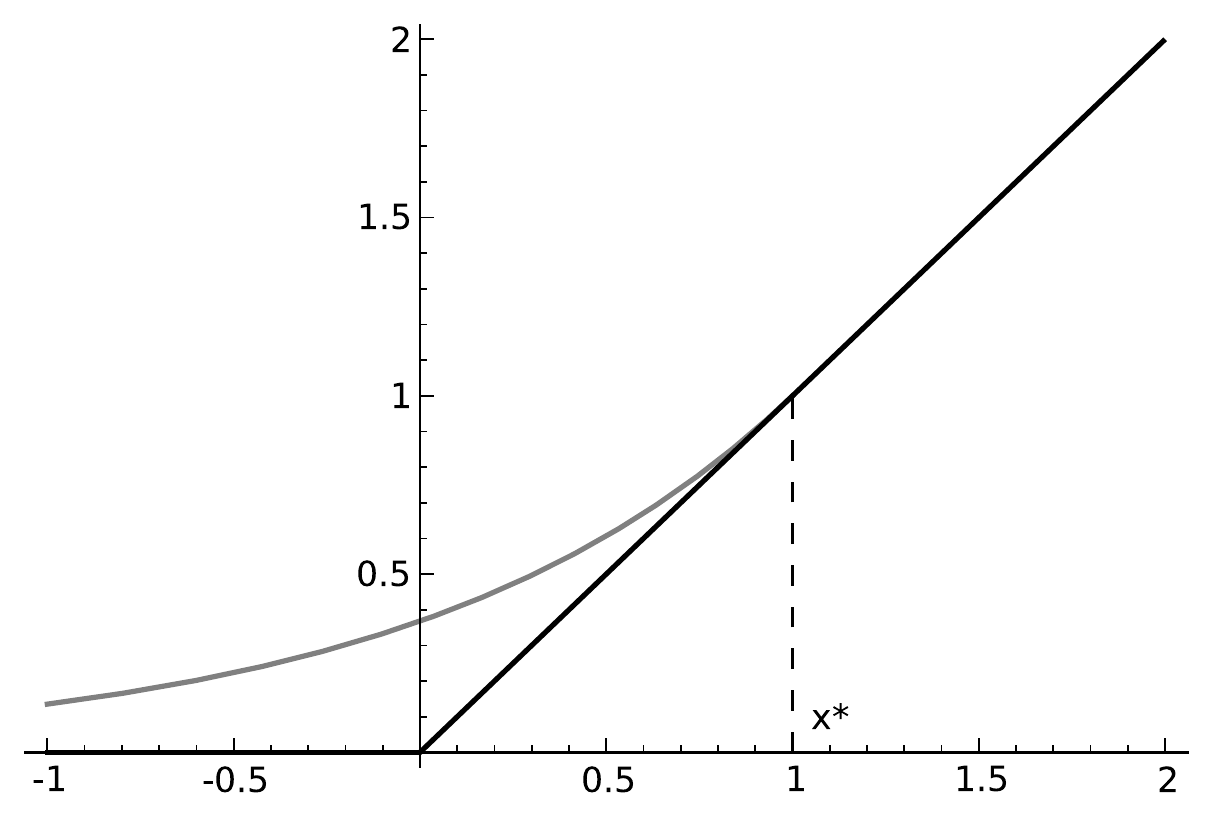}
\caption{\label{fig:taylor} OSP for the standard BM: $g$ (black), $\Va$ (gray, when different from $g$). Parameter $\desc=\frac12$.}
\end{figure}
\autoref{fig:taylor} shows the value function $\Va$ for $\desc=\frac12$.
\end{example}

\begin{example}[Discontinuous reward]
Now consider the reward function 
\be \label{eq:gnonsmooth}
g(x)=\begin{cases}
0,& x<0, \\
1, & x\geq 0.
\end{cases}
\ee
Obviously, the optimal stopping problem \eqref{eq:osp} with positive discount has optimal threshold $x^*=0$; we show this fact as a consequence of \autoref{teo:verif2}. 
First we observe that the RRC is fulfilled for $x^*=0$, for example $\gi(x)=1$ satisfy the inversion formula and is equal to $g(x)$ for $x\geq 0$. 
The inequality \eqref{eq:defx*mayor} is, in this particular case and considering $x^*=0$,
\bd
1 \geq \frac1{2\sqrt{2\desc}}\int_0^\infty e^{-\sqrt{2\desc}y}\desc 2 dy,
\ed
which clearly holds, since the right-hand side is $\frac12$. 
For $x\geq 0$ we have that $\al g=\desc>0$, therefore \autoref{hyp:algpositive} is valid. Finally we observe that 
$\psia(x)>0$ for $x<0$ and then \autoref{hyp:mayorizq} is also valid. We conclude, by the application of \autoref{teo:verif2}, that the problem is right-sided with $0$ as optimal threshold. The value function is, by \eqref{eq:VaCall}
\bd
\Va(x)=\begin{cases}
e^{\sqrt{2\desc}x},& x < 0,\\
1, & x\geq 0.
\end{cases}
\ed
Clearly the smooth fit principle does not hold in this example; from \autoref{teosf} we get that the representing measure $\nu$ of $\Va$ has to have a mass at $0$, in fact, from \autoref{prop:repmeasure}, we obtain that $\nu(\{0\})=\frac12$. We remark that the hypotheses that ensures \eqref{eq:nuleq} in \autoref{prop:repmeasure} are not fulfilled in this case; in fact, the absence of smooth fitting, in this particular example, comes from the reward function. Further on we present examples in which, although the reward function is differentiable, the smooth fitting fail due to the lack of regularity of the scale function or due to atoms in the speed measure.
\end{example}
\begin{example}[Non-smooth reward]
Following with the same Wiener process, consider the continuous reward function $g$, depending on a positive parameter $a$, defined by
\bd
g(x)=\begin{cases}
0, & x\leq -\frac1a,\\
ax+1, & -\frac1a<x<0,\\
1, & x\geq 0
\end{cases}
\ed
The question we try to answer is: for which values of the discount parameter $\desc$, the optimal threshold associated to the optimal stopping problem \eqref{eq:osp} is $x^*=0$. As in the previous example, and with the same arguments, we observe that the RRC (\autoref{cond:RRC}) is fulfilled for $x^*=0$ and also the inequality \eqref{eq:defx*mayor} and the condition \autoref{hyp:algpositive} in \autoref{teo:verif2} hold. 
We only need to verify that \autoref{hyp:mayorizq} is valid, i.e. for all $x<0$,
\bd
\psia(x)\frac{g(0)}{\psia(0)}\geq g(x),
\ed
or what is the same $e^{\sqrt{2\desc}x}\geq ax+1$ for $-\frac1a<x<0$. It is easy to see that previous equation holds if $a\geq \sqrt{2\desc}$. \autoref{fig:nonsmoothreward} shows the value function for $\desc=\frac12$ and $\desc=\frac18$ when $a=1$

\begin{figure}
\centering
\includegraphics[scale=.8]{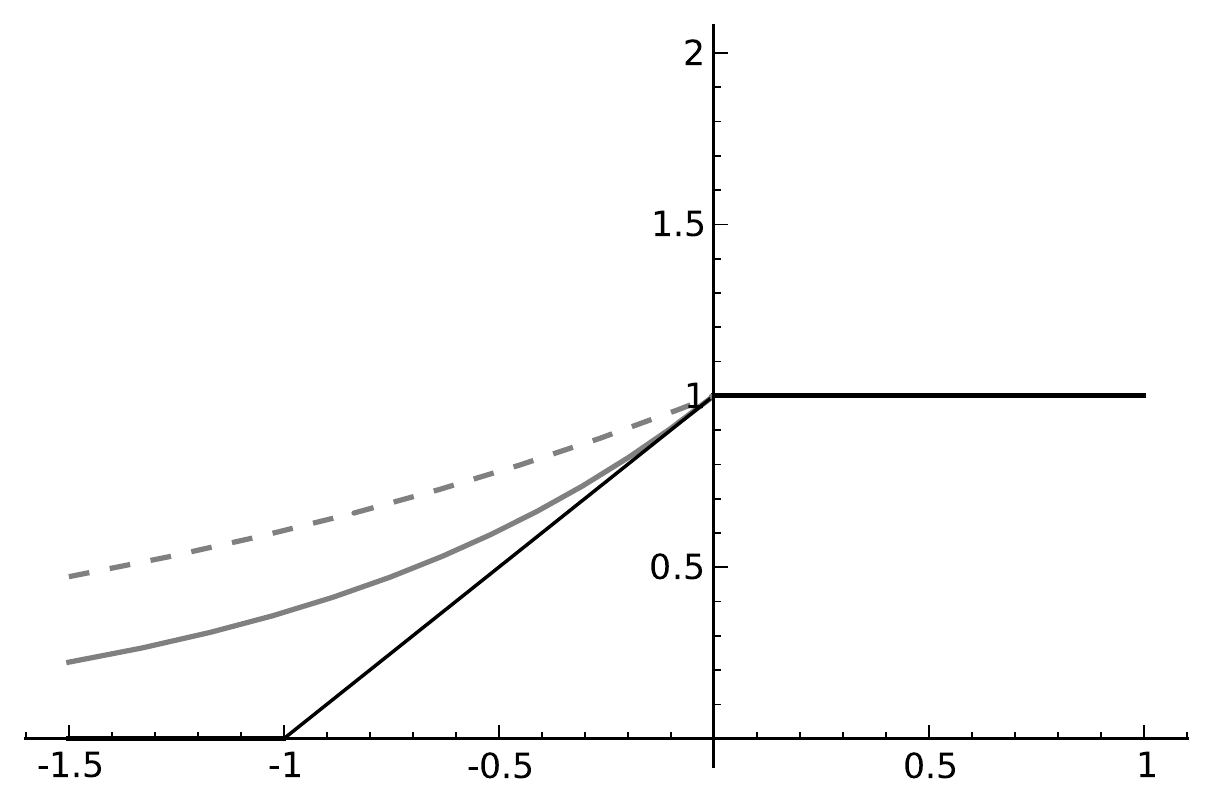}
\caption{\label{fig:nonsmoothreward} OSP for the standard BM, non-smooth reward: $g$ (black), $V_{0.5}$ (gray, when is different from $g$), $V_{0.125}$ (gray dashed, when is different from $g$). Parameter $a=1$}
\end{figure}
\end{example}

\subsection{American call options}

We consider the classical optimal stopping problem associated with the pricing of an American call option with dividends, solved by \cite{merton}.
\begin{example}
Consider  a geometric Brownian motion $\{X_t\}$ \citep[see][p.132]{borodin}, which is a solution of the stochastic differential equation
\begin{equation*}
dX_t=\sigma X_t dW_t + \mu X_t dt;
\end{equation*}
where $\mu\in \R$ and $\sigma^2>0$. The state space is $\I=(0,\infty)$. Set $\nu=\frac{\mu}{\sigma^2}-\frac{1}{2}$, the scale function is
\begin{equation*}
 s(x)=
\begin{cases}
-\frac{x^{-2\nu}}{2\nu},& \nu\neq 0,\\
\ln(x),&\quad \nu=0;\\
\end{cases}
\end{equation*}
and the speed measure $m(dx)=\frac{2}{\sigma^2}x^{\frac{2\mu}{\sigma^2}-2}dx.$
 The differential operator is  $Lf=\frac{1}{2}\sigma^2x^2 f''+ \mu x f',$ and the domain of the infinitesimal generator is
\begin{equation*}
\D_L=\left\{f: f \in \CC_b(\I),\ Lf \in \CC_b(\I)\right\}. 
\end{equation*} 
Functions $\psia$ and $\phia$ are given by 
\begin{equation*}
\psia(x)=x^{\gamma_1}, \text{ with }\gamma_1=\frac{1}{2}-\frac{\mu}{\sigma^2}{+\sqrt{\left(\frac{1}{2}-\frac{\mu}{\sigma^2}\right)^2+ \frac{2\desc}{\sigma^2}}}
\end{equation*}
and
\begin{equation*}
\phia(x)=x^{\gamma_2}, \text{ with $\gamma_2=\frac{1}{2}-\frac{\mu}{\sigma^2}{-\sqrt{\left(\frac{1}{2}-\frac{\mu}{\sigma^2}\right)^2+ \frac{2\desc}{\sigma^2}}}$},
\end{equation*}
the Wronskian is $\wa=2\sqrt{\nu^2+2\desc/\sigma^2}$.
The Green function, by \eqref{eq:Garepr},  is given by
\begin{equation*}
\Ga(x,y)=
\begin{cases}
\wa^{-1}x^{\gamma_2}y^{\gamma_1}, & x\geq y>0, \\[.5em]
\wa^{-1}y^{\gamma_2}x^{\gamma_1}, & y\geq x>0.
\end{cases}
\end{equation*}
Consider the reward function $g\colon \R^+ \to \R,\quad g(x):=(x-K)^+$
where $K$ is a positive constant, and a positive discount factor $\desc$ satisfying $\desc>\mu$. 
The reward function $g$ satisfies the RRC for $x_1=K$, it is enough to consider $\gi$ as a $\CC^2$ function, bounded in $(0,k)$ and such that $\gi(x)=x-k$ for $x\geq k$. 
Function $\gi$ defined satisfies the following inversion formula:
\begin{equation*}
 \gi(x)=\int_{\I} \Ga(x,y)  \al \gi(y) m(dy),
\end{equation*}
as a consequence of \autoref{propInversion2}. Observe that equation \eqref{eq:alIntegrable} holds. Equation \eqref{eq:gOverPsi} is in this case
\begin{equation*}
 \lim_{z\to \infty} \frac{\gi(z)}{\psia(z)} = \lim_{z\to \infty} z^{1-\gamma_1};
\end{equation*}
the last limit vanishes if $1-\gamma_1<0$ which is equivalent to $\mu < \desc$.
To find $x^*$ we solve equation \eqref{eq:defx*teo}. We will need a primitive of $\phia(x) \al \gi(x) m(x)$, where $m(x)$ is the density of the speed measure with respect to Lebesgue measure,
\begin{align*}
F(x)&=\int \phia(x) \al \gi(x) m(x) dx\\ 
&= \frac{2(\desc-\mu)}{\sigma^2}\int x^{\gamma_2+\frac{2 \mu}{\sigma^2}-1}-\frac{2 \desc K}{\sigma^2}\int x^{\gamma_2+\frac{2 \mu}{\sigma^2}-2} \\
&= \frac{2(\desc-\mu)}{\sigma^2(\gamma_2+\frac{2 \mu}{\sigma^2})} x^{\gamma_2+\frac{2 \mu}{ \sigma^2}}-\frac{2 \desc K}{\sigma^2(\gamma_2+\frac{2 \mu}{\sigma^2}-1)} x^{\gamma_2+\frac{2 \mu}{\sigma^2}-1};
\end{align*}
observe that 
\begin{equation*}
 \gamma_2+\frac{2 \mu}{\sigma^2}-1 < \gamma_2+\frac{2 \mu}{\sigma^2} = 1-\gamma_1 <0 
\end{equation*}
and $ \lim_{x\to \infty} F(x)=0.$
We solve equation \eqref{eq:defx*teo},
\begin{align*}
x^*-K &= \wa^{-1} \psia(x^*) \int_{x^*}^\infty \phia(y)\al g(y)m(y) dy\\
&= \wa^{-1} \psia(x^*) (-F(x^*))\\
&= \wa^{-1} \left(\frac{2 \desc K}{\sigma^2(\gamma_2+\frac{2 \mu}{\sigma^2}-1)}{x^*}^{\gamma_1+\gamma_2+\frac{2 \mu}{\sigma^2}-1}- \frac{2(\desc-\mu)}{\sigma^2(\gamma_2+\frac{2 \mu}{\sigma^2})}{x^*}^{\gamma_1+\gamma_2+\frac{2 \mu}{\sigma^2}}\right) \\ 
&= \wa^{-1} \left(\frac{2 \desc K}{\sigma^2(\gamma_2+\frac{2 \mu}{\sigma^2}-1)}- \frac{2(\desc-\mu)}{\sigma^2(\gamma_2+\frac{2 \mu}{\sigma^2})}{x^*}\right)
\end{align*}
then, using $\gamma_2=1-\frac{2 \mu}{\sigma^2}-\gamma_1$ we obtain that
\begin{equation*}
x^*\left(1+ \wa^{-1} \frac{2(\desc-\mu)}{\sigma^2(1-\gamma_1)}\right) = K \left(1+\wa^{-1} \frac{2 \desc }{\sigma^2(-\gamma_1)}\right),
\end{equation*}
concluding that
\begin{equation*}
 x^*=K \left(\frac{\gamma_1-1}{\gamma_1}\right) \left( \frac{\wa \gamma_1 - 2\desc/{\sigma^2}}{\wa(\gamma_1-1)-2\desc/{\sigma^2}+2\mu/{\sigma^2}}\right).
\end{equation*}
Calling $s=\sqrt{\frac{1}{4}-\frac{\mu}{\sigma^2}+\frac{\mu^2}{\sigma^4}+ \frac{2\desc}{\sigma^2}}$, we have $\wa=2s$, $\gamma_1 = \frac{1}{2}-\frac{\mu}{\sigma^2}+s$, and
\begin{align*}
\wa \gamma_1 - \frac{2\desc}{\sigma^2}&=s-\frac{2\mu s}{\sigma^2}+2s^2-\frac{2\desc}{\sigma^2}\\
&=s-\frac{2\mu s}{\sigma^2}+s^2+\underbrace{\frac{1}{4}-\frac{\mu}{\sigma^2}+\frac{\mu^2}{\sigma^4}+ \frac{2\desc}{\sigma^2}}_{s^2}-\frac{2\desc}{\sigma^2}=\gamma_1^2.
\end{align*}
In the same way,
\begin{equation*}
\wa(\gamma_1-1)-\frac{2\desc}{\sigma^2}+\frac{2\mu}{\sigma^2} 
= \gamma_1^2-\wa+\frac{2\mu}{\sigma^2}
= \gamma_1^2-2\gamma_1+1
=(\gamma_1-1)^2.
\end{equation*}
Finally, we arrive to the conclusion: 
\begin{equation*}
x^*=K \left(\frac{\gamma_1}{\gamma_1-1}\right).
\end{equation*}
Observing that $x^*>x_1=K$, we only need to verify \autoref{hyp:algpositive} and \autoref{hyp:mayorizq} is fulfilled in order to apply  \autoref{teo:verif1}. The condition \autoref{hyp:algpositive} is, in this example,
\begin{equation*}
(\desc-\mu) x-\desc K\geq 0,\quad \mbox{ if }\quad x>K \left(\frac{\gamma_1}{\gamma_1-1}\right);
\end{equation*}
it is enough to prove $(\desc-\mu) x^*-\desc K\geq 0$, which is equivalent to $\mu \gamma_1\leq \desc.$ To justify the validity of the previous equation remember that $\psia$ satisfies $\al \psia \equiv 0$ and in this particular cases
\begin{equation*}
\al \psia(x) = \left(\desc - \frac{1}{2}\sigma^2 \gamma_1 (\gamma_1 -1)-\mu \gamma_1\right)x^{\gamma1},
\end{equation*}
so we have $\desc -\mu \gamma_1 = \frac{1}{2}\sigma^2 \gamma_1 (\gamma_1 -1) \geq 0,$ concluding what we need. In order to verify \autoref{hyp:mayorizq} we need to observe $\left(\frac{1}{\gamma_1-1}\right)\left(\frac{\gamma_1 -1}{K \gamma_1} x\right)^{\gamma_1}<(x-K)^+$ for all $x<x^*.$

We conclude, by the application of \autoref{teo:verif1}, that the problem is right-sided with optimal threshold $x^*$. Observe that, as in the previous example, hypothesis of \autoref{teosf}, \autoref{scalesf} and \autoref{clasicsf} are fulfilled and all variants of smooth fit principle hold.

In virtue of the results in \autoref{sec:easyequation}, the threshold $x^*$ can be obtained also by solving \eqref{eq:easyeq}, which in this case is
\[
\frac{1}{\gamma_1\, x^{\gamma_1-1}}=\frac{x-K}{x^{\gamma_1}},
\]
and we obtain very easily the already given threshold $x^*$.
\end{example}

\subsection{Russian Options}

Optimal stopping problems regarding the maximum of a process constitutes other topic of high fertility in optimal stopping. In particular, it has applications in the search of maximal inequalities. Some optimal stopping problems for the maximum of a process can be reduced to a regular optimal stopping problem --as is the case in this example--. Some works regarding the maximum process are \cite{dubinsSheppSh,peskir1998optimal, pedersen2000discounted,skewOS}. For further reference see \cite{ps}. In the article by \cite{kramkov1994integral}, the authors use the same kind of reduction used in this example, but in this case to reduce an optimal stopping problem for the integral of the process.

The Russian Option was introduced by \cite{shepp93}. If $X_t$ is a standard geometric Brownian motion and $S_t=\max\{X_s\colon 0\leq s \leq t\}$, the Russian option gives the holder the right --but not the obligation-- to receive an amount $S_\tau$ at a moment he can choose. 
First the authors found the value of the option, reducing the problem to an optimal stopping problem of a two-dimensional Markov process.
 Later they found the way to solve the same problem based on the solution of a one-dimensional optimal stopping problem \citep{shepp94}. More recently \cite{salminen00RussianOptions}, making use of a generalization of Lévy's theorem for a Brownian motion with drift shortened the derivation of the valuation formula in \cite{shepp94} and solved the related optimal stopping problem.  Now we show how to use our results to solve this one-dimensional optimal stopping problem.
 \begin{example}
Consider $\desc>0$, $r>0$ and $\sigma>0$. Let $X$ be a Brownian motion on $\I=[0,\infty)$,  with drift $-\delta < 0$, where $\delta=\frac{r+\sigma^2/2}{\sigma} $ and reflected at $0$ (see \cite{borodin}, p. 129).
The scale function is
\begin{equation*}
 s(x)=\frac{1}{-2\delta}(1-e^{2 \delta x});
\end{equation*}
the speed measure is $m(dx)=2 e^{-2 \delta\,x}.$ 
The differential operator is given by $Lf(x)=f''(x)/2-\delta f'(x)$ for $x>0$, and $Lf(0)=\lim_{x\to 0^+} Lf(x)$, being its domain 
\begin{equation*}
\D_L=\left\{f: f\in \CC_b(\I),\ Lf \in \CC_b(\I), \lim_{x\to 0^+}f'(x)=0\right\}. 
\end{equation*} 
Functions $\phia$ and $\psia$ are given by 
\begin{equation*}
\phia(x)=e^{-(\gamma - \delta) \,x}
\end{equation*}
and 
\begin{equation*}
\psia(x)=\frac{\gamma - \delta}{2 \gamma} e^{(\gamma + \delta) x} +\frac{\gamma + \delta}{2 \gamma} e^{-(\gamma - \delta)x},
\end{equation*}
where $\gamma=\sqrt{2 \desc + \delta^2}$, the Wronskian is given by $\wa=\gamma - \delta$. Consider the reward function $g(x):=e^{\sigma x}$, which does satisfy the RRC for every $x_1>0$. We have,
\begin{equation*}
\al g(x)=(\desc - \sigma^ 2/2 + \delta \sigma)e^{\sigma x} = (\desc + r)e^{\sigma x}>0. 
\end{equation*}
In order to apply \autoref{teo:verif1} we solve equation \eqref{eq:defx*teo}, which in this case is
\[
e^{\sigma x}=\frac{1}{\gamma-\delta}\left(\frac{\gamma - \delta}{2 \gamma} e^{(\gamma+\delta) x}+\frac{\gamma + \delta}{2 \gamma} e^{-(\gamma-\delta) x} \right)\int_{x}^\infty 2 (\desc+r) e^{(-\gamma-\delta+\sigma) y}dy,
\]
or we can solve the easier equivalent equation \eqref{eq:easyeq}, which in this case is
\[
\frac{\sigma e^{\sigma x}}{\frac{\gamma^2 - \delta^2}{2 \gamma} \left(e^{(\gamma + \delta) \, x} - e^{-(\gamma - \delta) \, x}\right)}=\frac{e^{\sigma x}}{\frac{\gamma - \delta}{2 \gamma} e^{(\gamma + \delta) \, x} +\frac{\gamma + \delta}{2 \gamma} e^{-(\gamma - \delta) \, x}}
\]
obtaining that% (observe that  $-\gamma-\delta+\sigma<0$)
\bd
x^*=\frac{1}{2\gamma}\ln{\left(\left(\frac{\gamma+\delta}{\gamma-\delta}\right)\left(\frac{\gamma-\delta+\sigma}{\gamma+\delta-\sigma}\right)\right)}. 
\ed
Assertions \autoref{hyp:algpositive} and \autoref{hyp:mayorizq} remains to be verified to obtain that the optimal stopping rule is to stop when $X_t\geq x^*$. This result agree with the ones obtained in the previous works.

We solve a particular case with $\desc=0.7$, $r=0.5$ and $\sigma=1$. \autoref{fig:russian} shows the value function $\Va$ and the reward $g$ for this example. The threshold is $x^*\simeq 0.495$.
 \end{example}

\begin{figure}[htb]
\begin{center}
\includegraphics[scale=.8]{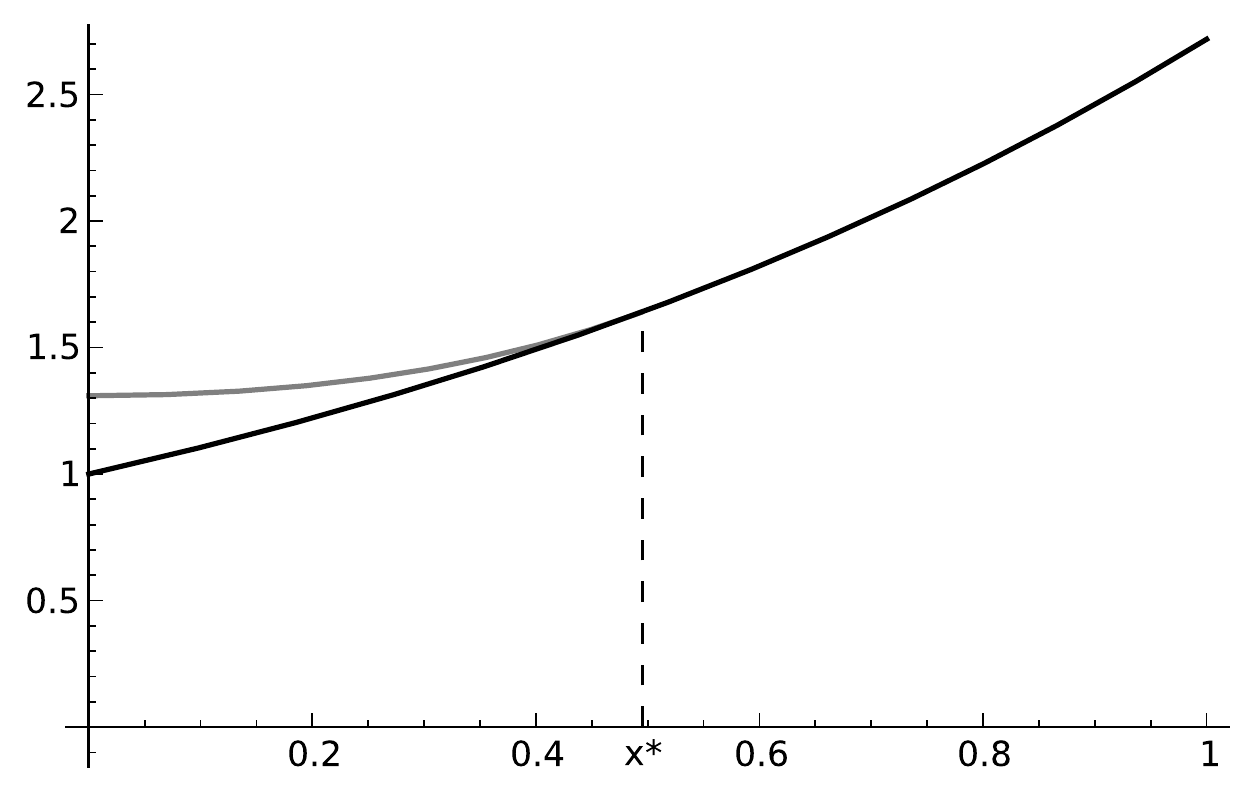}
\end{center}
\caption{\label{fig:russian} Russian options: $g$ (black), $\Va$ (gray, when different from $g$). Parameters: $\desc=0.7$, $r=0.5$ and $\sigma=1$.}
\end{figure}

\subsection{Skew Brownian motion}
We consider a Brownian motion skew at zero \citep[see][p.126]{borodin}. For further results about Skew Brownian we refer to \cite{lejay} and the references therein. This process is a standard Brownian motion when off the origin, but it has an asymmetric behaviour from the origin. It has the following property
\begin{equation*}
\P_0(X_t\geq 0)= 1-\P_0(X_t\leq 0)=\beta.
\end{equation*}
The parameter $\beta\in (0,1)$ is known as the ``skewness parameter of the process''.
The space state of this process is $\I=\R$; the scale function is
\begin{equation*}
 s(x)=\begin{cases}
       \frac{x}{\beta},&x\geq 0,\\
       \frac{x}{1-\beta},&x<0;
      \end{cases}
\end{equation*}
and the speed measure is
\begin{equation*}
m(dx)= \begin{cases}
	  2(1-\beta)\ dx, &x<0;\\
	  2 \beta\ dx, & x>0.
       \end{cases}
\end{equation*}
The differential operator is $Lf(x)=f''(x)/2$ if $x\neq 0$ and $Lf(0)=\lim_{x\to 0}Lf(x)$. The domain of the infinitesimal generator is
\begin{equation*} 
\D_L=\{f: f,\ Lf \in \CC_b(\I),\ \beta f'(0^+)=(1-\beta)f'(0^-)\}. 
\end{equation*}
Functions $\phia$ and $\psia$ are given by 
\begin{equation*}
\phia(x)=
\begin{cases}
\frac{1-2\beta}{1-\beta}\sinh(x \sqrt{2\desc})+e^{-\sqrt{2\desc}\,x},& x\leq 0, \\
e^{-\sqrt{2\desc}\,x}, & x\geq 0,
\end{cases}
\end{equation*}
and 
\begin{equation*}
\psia(x)=
\begin{cases}
e^{\sqrt{2\desc}\,x}, & x\leq 0, \\
\frac{1-2\beta}{\beta}\sinh(x \sqrt{2\desc})+e^{\sqrt{2\desc}\,x}, & x\geq 0,
\end{cases}
\end{equation*}
and the Wronskian is $\wa=\sqrt{2\desc}$.
 \begin{example}[Skew BM, $g(x)=x^+$]
Consider the reward function $g(x)=x^+$. It satisfies the RRC for $x_1=0$. We have $\al g(x)= \desc x, x\geq0.$
Equation \eqref{eq:defx*teo} is in this case
\begin{equation*}
x^*=\frac{1}{\sqrt{2\desc}}\left(\frac{1-2\beta}{\beta}\sinh(\sqrt{2\desc}\,x^*)+e^{\sqrt{2\desc}\,x^*}\right)\int_{x^*}^\infty e^{-\sqrt{2\desc}\,t}\desc t\,2\beta dt 
\end{equation*}
or equivalently
\begin{equation} \label{eq:xxskew} x^*=\frac{1}{2\sqrt{2\desc}}\left((2\beta-1)e^{-\sqrt{2\desc}\,x^*}(\sqrt{2\desc}\,x^*+1)+\sqrt{2\desc}\,x^*+1\right).
\end{equation}
In general it is not possible to solve analytically equation \eqref{eq:xxskew}. If we consider the particular case $\beta=\frac{1}{2}$, in which the process is the ordinary Brownian motion, we obtain that 
$x^*=\frac{1}{\sqrt{2\desc}};$
according with results obtained in \autoref{ex:Brownian}. Consider a particular case, in which $\desc=1$ and $\beta=0.9$. Solving numerically equation \eqref{eq:xxskew} we obtain that
\begin{equation*}
x^*\simeq  0.82575. 
\end{equation*}
 \autoref{fig:skew} shows the optimal expected reward function $\Va$. Observe that if $x>x^*$, $\Va(x)=x$, and $\Va$ has derivative in $x^*$.
\begin{figure}[htb]
\begin{center}
\includegraphics[scale=.8]{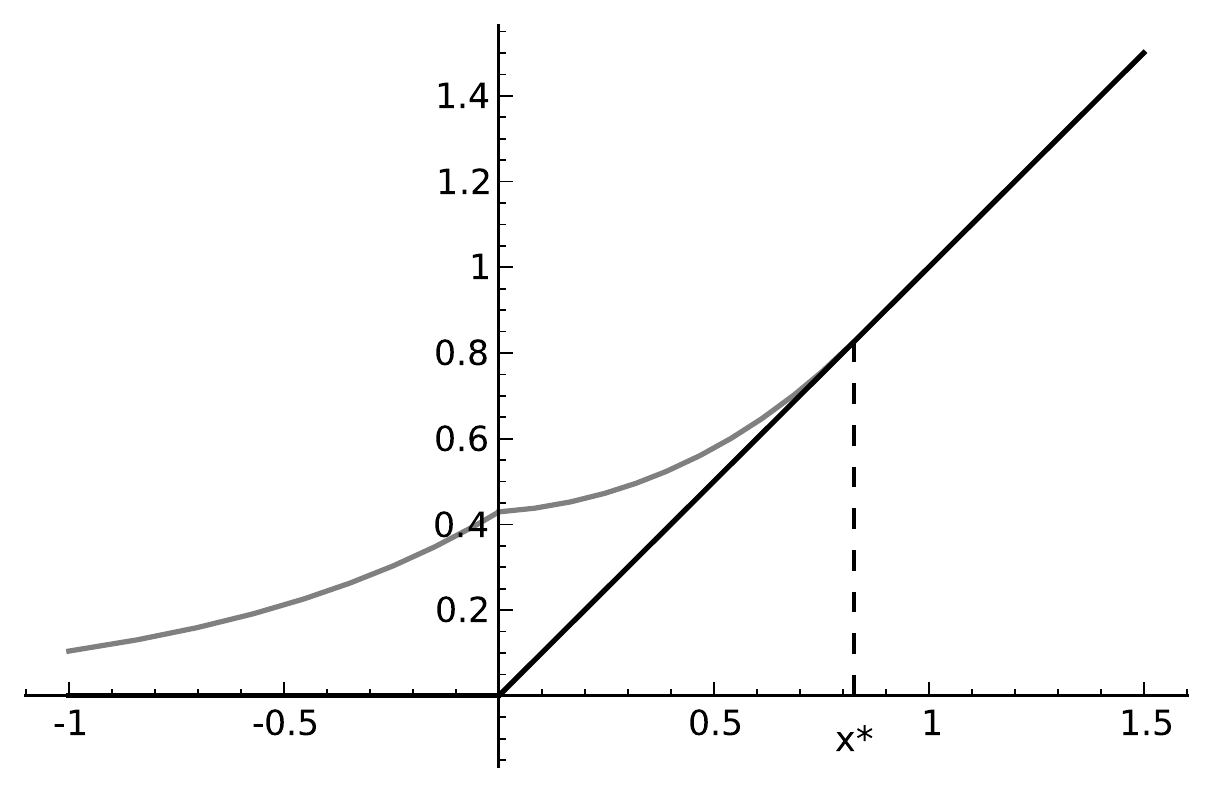}
\end{center}
\caption{\label{fig:skew} OSP for the Skew BM: $g$ (black), $\Va$ (gray, when different from $g$). Parameters: $\desc=1$ and $\beta=0.9$.}
\end{figure}
The article by \cite{skewOS} considers the optimal stopping problem for the maximum of this process.
\end{example}

\begin{example}[Skew BM: an example without smooth fitting]
Consider the Skew Brownian motion, process presented in the previous example, with parameters value $\beta=1/3$ and $\desc=1/8$. Let $g(x)=(x+1)^+$ be the reward function. Functions $\phia$, $\psia$ were already presented.
We have $\al g(x)=\desc (x+1),\ x\geq0.$
Observe that $x^*=0$ is solution of \eqref{eq:defx*teo}. It is easy to see that the hypotheses of \autoref{teo:verif1} are fulfilled. 
We conclude that the problem is right-sided with threshold $x^*=0$. 
Moreover, the value function satisfies:
\begin{equation*}
\Va(x)=
\begin{cases}
x+1, & x\geq 0, \\
\psi_\desc(x), & x\leq 0. 
\end{cases}
\end{equation*}
Unlike the previous examples $\Va$ is not derivable at $x^*$. As can be seen in  \autoref{fig:skewNoFit} the graphic of $\Va$ shows an angle in $x=0$.

\begin{figure}[htb]
\begin{center}
\includegraphics[scale=.8]{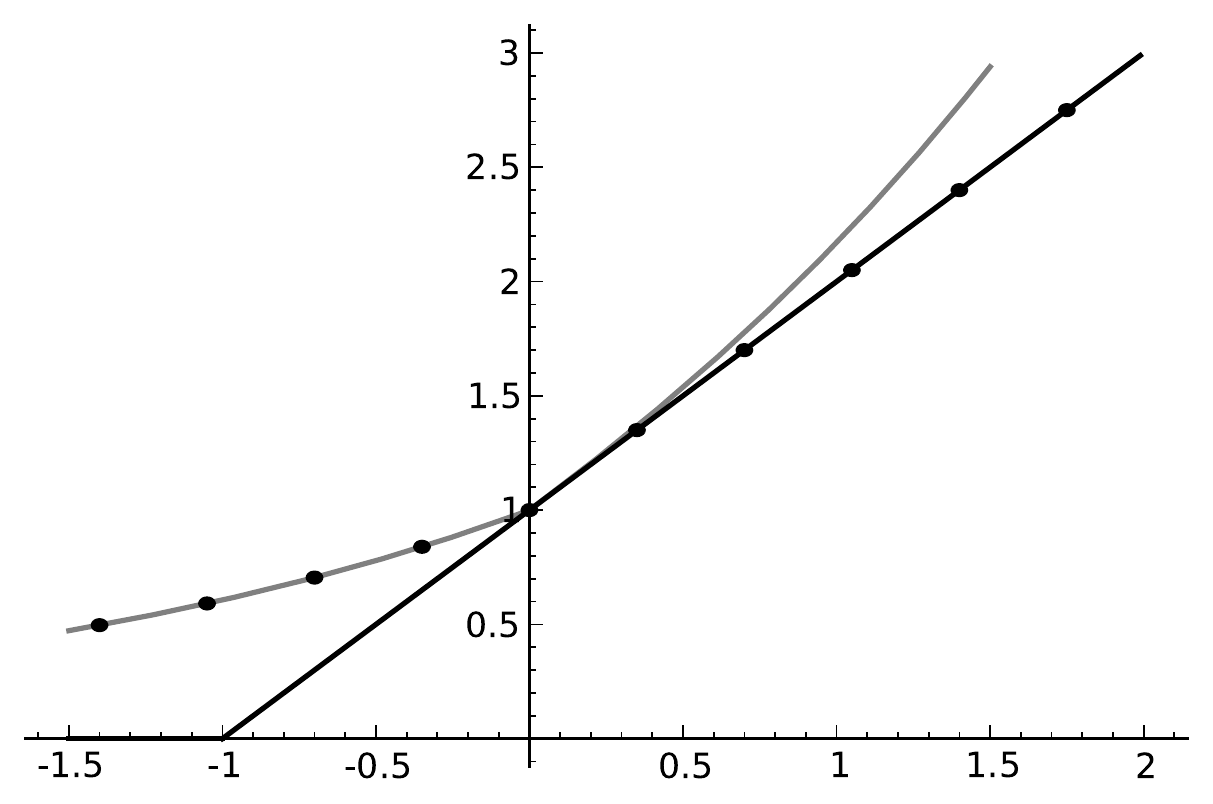}
\end{center}
\caption{\label{fig:skewNoFit} OSP for the Skew BM, example without smooth fitting: $g$ (black), $\psia$ (gray), $\Va$ (highlighted with dots). Parameters: $\desc=1/8$ and $\beta=1/3$.}
\end{figure}

As we have mentioned, the smooth fit principle states that $x^*$, the critical value between the continuation and the stopping region, satisfies the equation $\Va'(x^*)=g'(x^*)$. 
This principle is valid for a wide class of optimal stopping problems, 
and it is commonly used to find $x^*$. 
In the article \citep{peskir07},
the author gives an example of an optimal stopping problem of a regular diffusion with a differentiable reward function in which the smooth fit principle does not hold. 
Despite the fact that in the previous example the reward function is not differentiable (at $0$), it is very easy to see that a differentiable reward function $\tilde{g}$ for which the solution of the OSP is exactly the same can be considered;  \autoref{fig:skewNoFitDif} shows a possible differentiable $\tilde{g}$.

\begin{figure}[hbt]
\begin{center}
\includegraphics[scale=.8]{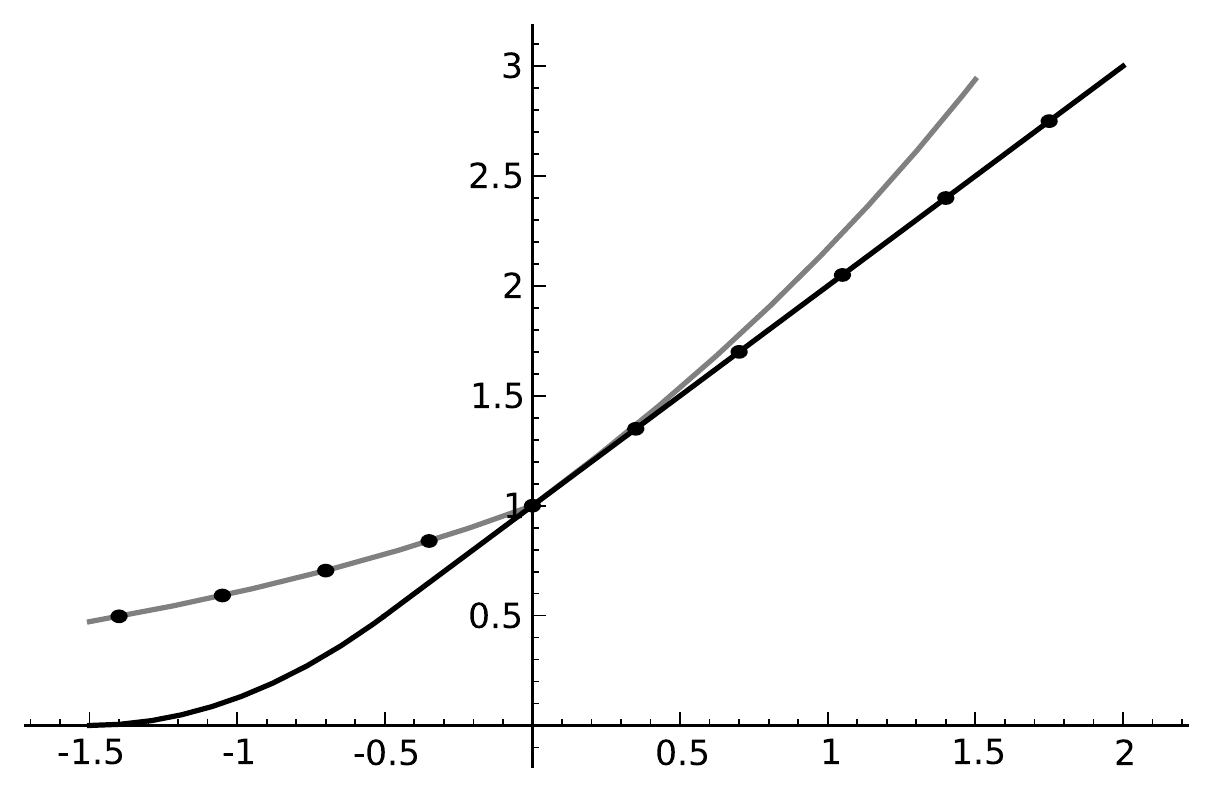}
\end{center}
\caption{\label{fig:skewNoFitDif} OSP for the Skew BM, example without smooth fitting with differentiable reward: $\tilde{g}(x)$ (black), $\psia$ (gray), $\Va$ (highlighted with dots). Parameters: $\desc=1/8$ and $\beta=1/3$.}
\end{figure}

\end{example}

\subsection{Sticky Brownian Motion} 

Consider a Brownian motion, sticky in 0  \citep[see][p. 123]{borodin}. It is a Brownian motion out of 0, but in 0 expends a positive time, which depends on a positive parameter that we assume to be 1.

The space state of this process is $\I=\R$. The scale function is $s(x)=x$ and the speed measure is $m(dx)=2dx+2\delta_{\{0\}}(dx)$. The differential operator is $Lf(x)=\frac{f''(x)}{2}$ when $x\neq 0$, and $Lf(0)=\lim_{x\to 0}Lf(x)$; being its domain
\begin{equation*} 
\D_L=\left\{f: f,\ Lf \in \CC_b(\I),\ f''(0^+)=f'(0^+)-f'(0^-)\right\}. 
\end{equation*}
Functions $\phia$ and $\psia$ are given by 
\begin{equation*}
\phia(x)=
\begin{cases}
e^{-\sqrt{2\desc}\,x} - \sqrt{2\desc}\,\sinh(x \sqrt{2\desc}),& x\leq 0, \\
e^{-\sqrt{2\desc}\,x}, & x\geq 0;
\end{cases}
\end{equation*}
and 
\begin{equation*}
\psia(x)=
\begin{cases}
e^{\sqrt{2\desc}\,x}, & x\leq 0, \\
e^{\sqrt{2\desc}\,x} +\sqrt{2\desc}\,\sinh(x \sqrt{2\desc}), & x\geq 0;
\end{cases} 
\end{equation*}
the Wronskian is $\wa=2\sqrt{2\desc}+2\desc.$

We remark that in the literature there are different definitions of sticky Brownian motion, \cite[see][note on page 223]{amir1991sticky}. 
According to the definition we follow the process is defined with $\R$ as space state. Another used definition has space state $[0, \infty)$ and it is given by the stochastic differential equation
\bd
dX_t=\theta \ind{\{X_t=0\}} dt + \ind{\{X_t>0\}}dB_t.
\ed

\begin{example} \label{ex:sticky}
Consider the reward function $g(x)=(x+1)^+$. It is easy to see that the RRC is fulfilled for $x_1=-1$. 
We discuss the solution of the optimal stopping problem depending on the discount factor; particularly, we are interested in finding the values of $\desc$ for which the optimal threshold is the sticky point. We are going to use \eqref{eq:defx*teo} in a different way: we fix $x=0$ and solve the equation with $\desc$. We obtain that
\begin{equation*}
 1=\wa^{-1} \int_{(0,\infty)} e^{-\sqrt{2\desc}\,y}\desc(y+1) 2dy
\end{equation*}
and we find $\desc_1=\frac{(-1+\sqrt{5})^2}{8}\simeq 0.19$ is the solution. It can be seen, by the application of  \autoref{teo:verif1}, that with $\desc=\desc_1$ the problem is right-sided with threshold 0. Another option to have threshold 0 could be, by the application of  \autoref{teo:verif2}, that $x^*=0$ is the minimum satisfying \eqref{eq:defx*mayor}. In order to find these values of $\desc$ it is useful to solve equation
\begin{equation}
\label{eq:alpha2}
g(x)=\wa^{-1}\psia(x)\int_{[x,\infty)}\phia(y)\al g(y) m(dy)
\end{equation}
with $x=0$. Since the measure m(dx) has an atom at $x=0$, the solution of the previous equation is different from $\desc_1$. 
Solving this equation we find the root $\desc_2=1/2$. It is easy to see that, for $\desc \in (\desc_1,\desc_2]$, the minimal $x$ satisfying \eqref{eq:defx*mayor} is 0. Then \autoref{teo:verif2} can be applied to conclude that 0 is the optimal threshold.
In this case we cannot apply the theorems of \autoref{section:sf} to conclude that any of the smooth fit principles hold. In fact for $\desc\in (\desc_1,\desc_2)$ any of the principles is fulfilled. 
With $\desc=\desc_2$ there is \sfp\ and \ssfp. This is not a consequence of \eqref{eq:alpha2}, but it follows from the particular choose of the reward function. 
This example shows that theorems on smooth fit only give sufficient conditions.
\autoref{table} summarizes the information about the solution of the OSP in this example.
\begin{table}
\caption{\label{table} Sticky BM: solution of the OSP depending on $\desc$}
\centering
{\begin{tabular}{@{}|lcccccc|}
\hline
$\desc$ & $x^*$ & Theo. & \sfp & \ssfp & \psfp & Fig.\\
\hline
$\desc\in(0,\desc_1)$ & $x^*>0$ & \ref{teo:verif1} & yes & yes & yes &\ref{fig:menor}\\
$\desc=\desc_1$ & $x^*=0$ & \ref{teo:verif1} & no & no & yes&\ref{fig:a1}\\
$\desc\in(\desc_1,\desc_2)$& $x^*=0$ &\ref{teo:verif2} & no & no & no&\ref{fig:medio}\\
$\desc=\desc_2$ & $x^*=0$ &\ref{teo:verif2} & yes & yes & no&\ref{fig:a2}\\
$\desc\in (\desc_2,+\infty)$ & $x^*<0$ & \ref{teo:verif1} & yes & yes & yes&\ref{fig:mayor}\\
\hline
\end{tabular}}
\end{table}

\begin{figure}
\begin{center}
\includegraphics[scale=.8]{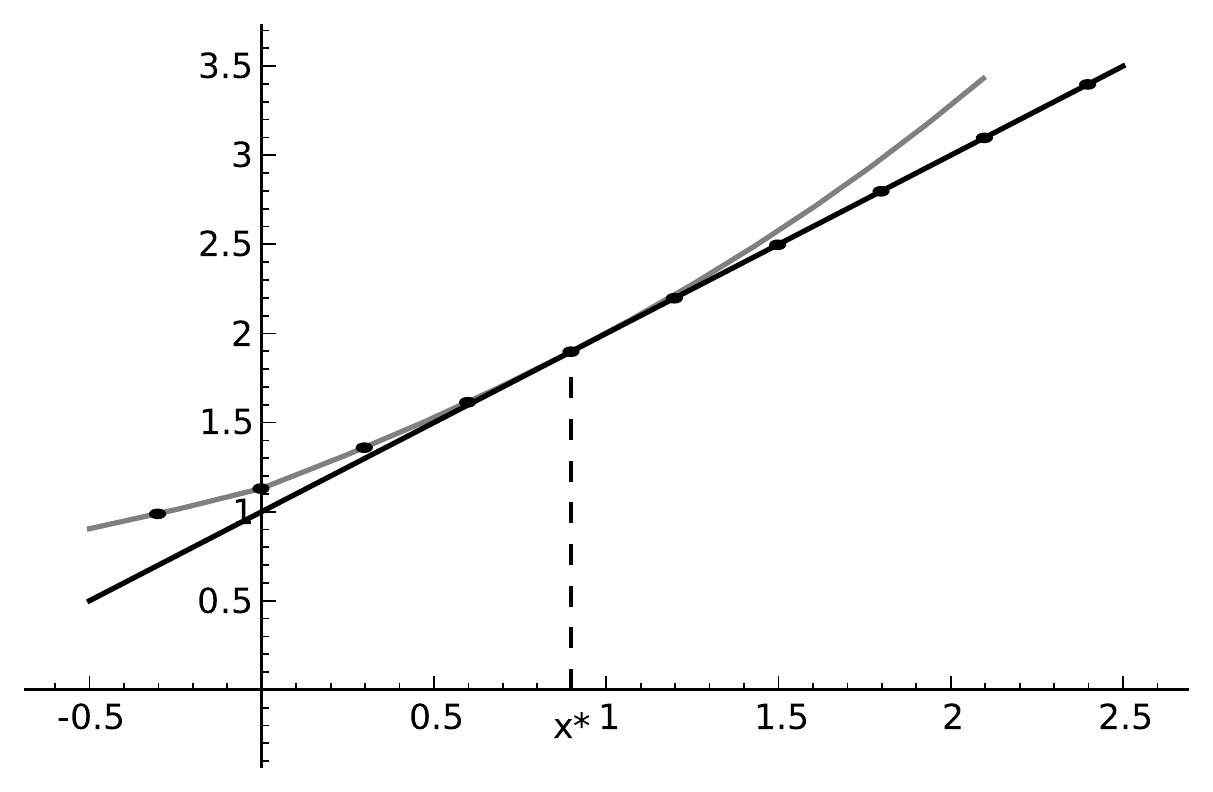}
\caption{\label{fig:menor}OSP for the Sticky BM: $g$ (black), $k \psia$ (gray), $\Va$ (highlighted with dots). Parameter $\desc=0.1$.}
\end{center}
\end{figure}

\begin{figure}
\begin{center} 
\includegraphics[scale=.8]{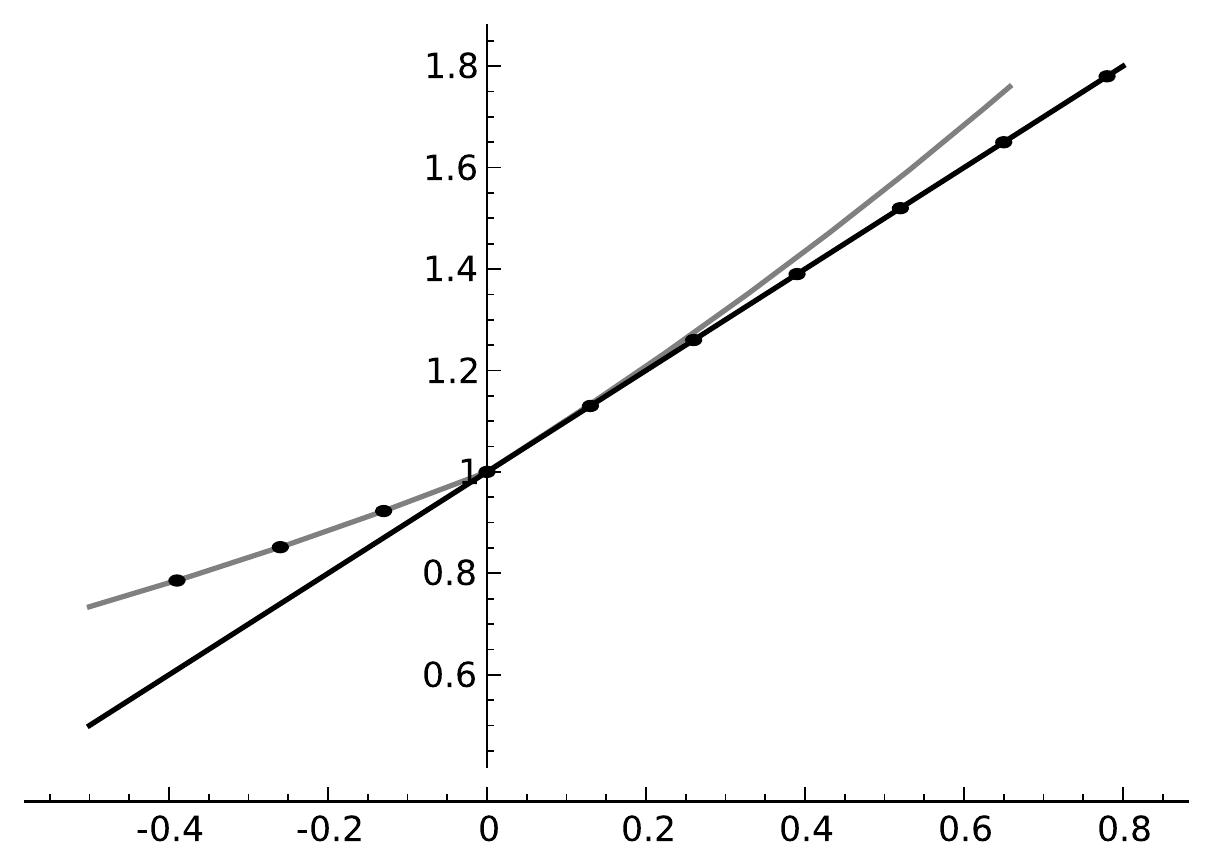}
\caption{\label{fig:a1} OSP for the Sticky BM: $g$ (black), $k \psia$ (gray), $\Va$ (highlighted with dots). Parameter $\desc=\desc_1\simeq 0.19$.}
\end{center}
\end{figure}

\begin{figure}
\begin{center}
\includegraphics[scale=.8]{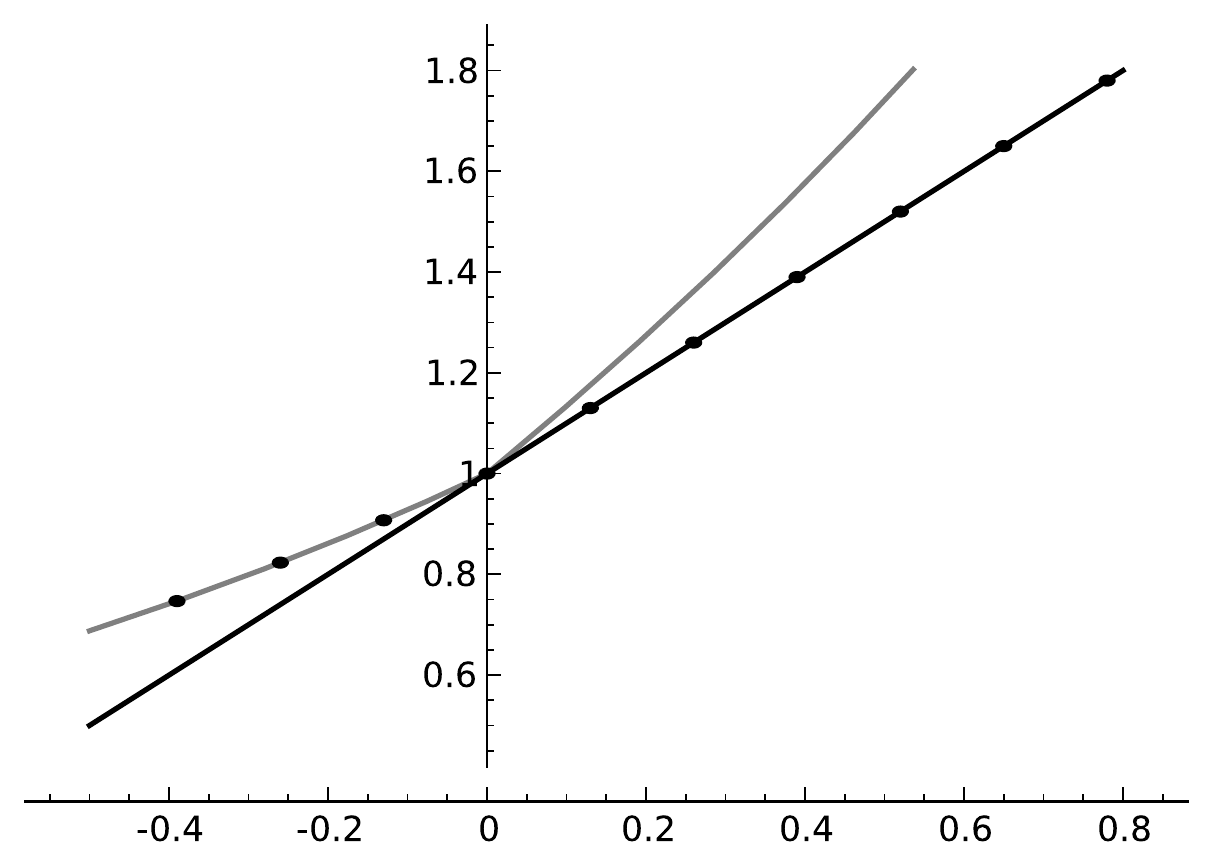}
\caption{\label{fig:medio}OSP for the Sticky BM: $g$ (black), $k \psia$ (gray), $\Va$ (highlighted with dots). Parameter $\desc=0.28$.}
\end{center}
\end{figure}

\begin{figure}
\begin{center}
\includegraphics[scale=.8]{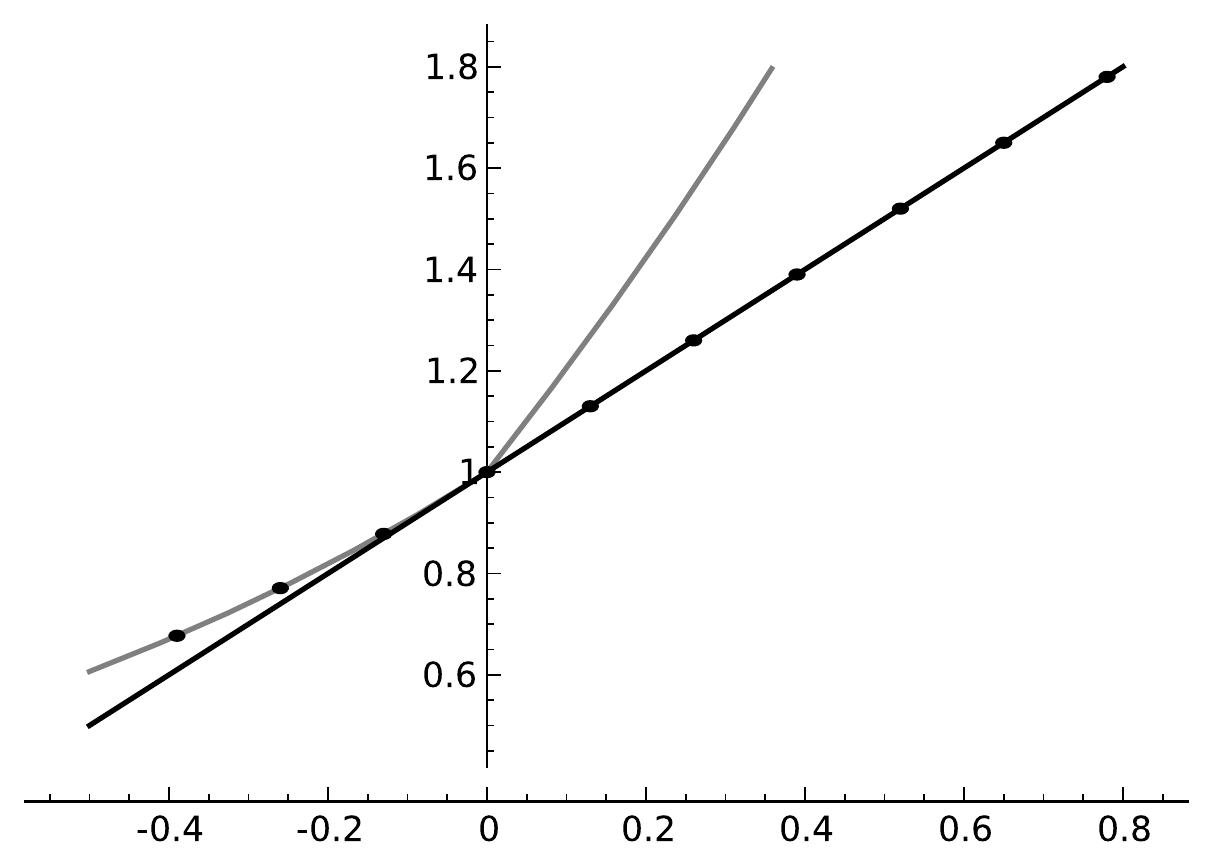}
\caption{\label{fig:a2}OSP for the Sticky BM: $g$ (black), $k \psia$ (gray), $\Va$ (highlighted with dots). Parameter $\desc=\desc_2=0.5$.}
\end{center}
\end{figure}

\begin{figure}
\begin{center}
\includegraphics[scale=.8]{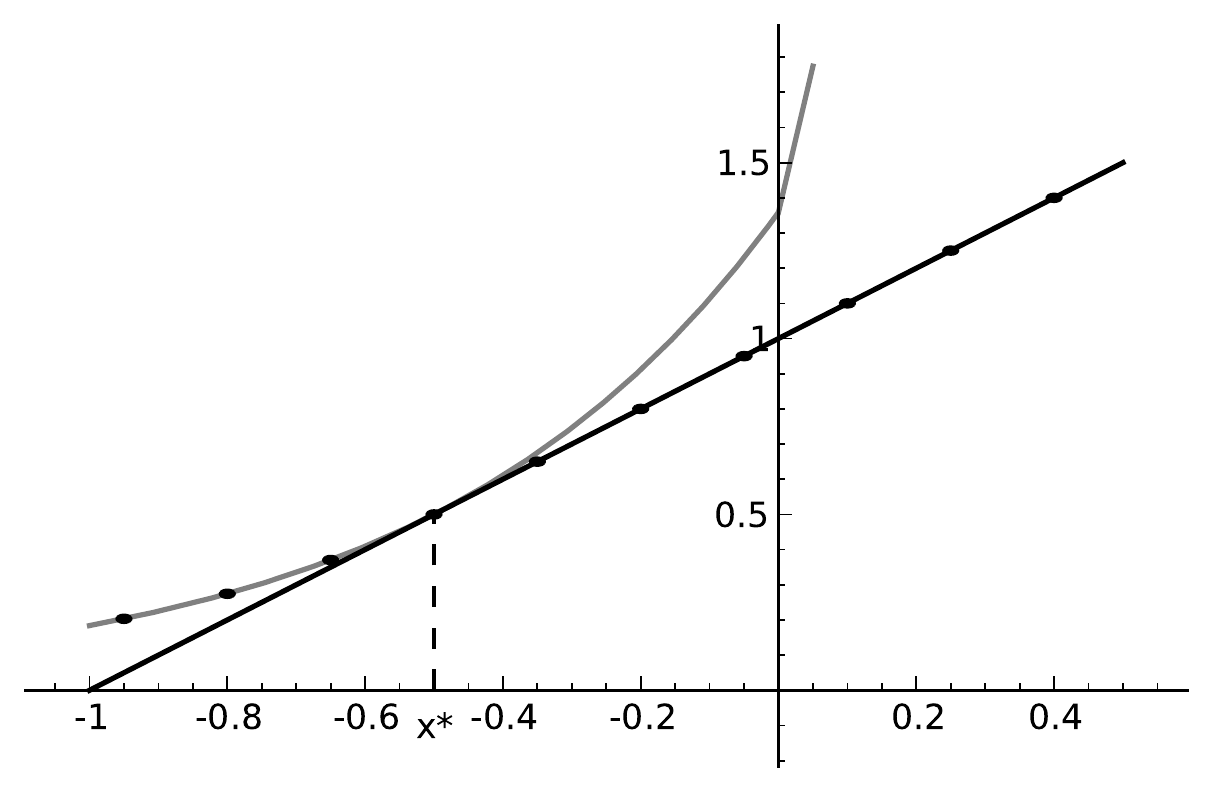}
\caption{\label{fig:mayor}OSP for the Sticky BM: $g$ (black), $k \psia$ (gray), $\Va$ (highlighted with dots). Parameter $\desc=2$.}
\end{center}
\end{figure}

\end{example}

%\newpage

\subsection{Bessel process} \label{ex:bessel}
Consider a 3-dimensional Bessel process $X$. It is a process with state space $\I=[0,\infty)$, which is given by $X_t=\|B_t\|$ with $B$ a 3-dimensional Brownian motion. \cite[see][p. 134 for details]{borodin}.

We have for $x\neq 0$:
\begin{equation*}
\psia(x)=\frac{2 \sinh(\sqrt{2\desc}\, x)}{x}
\end{equation*}
and 
\begin{equation*}
\psia'(x)=\frac{ 2 x \sqrt{2\desc} \cosh(\sqrt{2\desc}\, x)-2 \sinh(\sqrt{2\desc}\, x)}{x^2}
\end{equation*}

%The scale function is $s(x)=1/x$ and the speed measure is $m(dx)=2x^2dx$. The differential operator is $Lf(x)=\frac{f''(x)}{2}+\frac{1}{x}f'(x)$ when $x > 0$.
%The domain of the infinitesimal generator is
%\begin{equation*} 
%\D_L=\{f: f,\ Lf \in \CC_b(\I),\ \frac{df^+}{ds}(0^+)=0\}. 
%\end{equation*}
%Functions $\phia$ and $\psia$ are given by 
%\begin{equation*}
%\phia(x)=\frac{1}{x}e^{-\sqrt{2\desc}\, x}\qquad (x\neq 0)
%\end{equation*}
%and 
%\begin{equation*}
%\psia(x)=\frac{2}{x}\sinh(\sqrt{2\desc}\, x)\qquad (x\neq 0)
%\end{equation*}
%the Wronskian is $\wa=2\sqrt{2\desc}.$

\begin{example}
Consider the reward function $g(x)=x^2$. Assuming that the OSP \eqref{eq:osp} is right-sided we find the threshold by solving equation \eqref{eq:easyeq}, which in this case is
\[
\frac{2x^3}{ 2 x \sqrt{2\desc} \cosh(\sqrt{2\desc}\, x)-2 \sinh(\sqrt{2\desc}\,x)}=\frac{x^3}{2 \sinh(\sqrt{2\desc\,x})}.
\]
After computations, we conclude that $x^*=z/\sqrt{2\desc}$, with $z$ the positive solution of
\[
\arctan(z)=\frac{z}{3}.
\]
To conclude that the problem is indeed right-sided it remains to observe that $\al g(x)>0$ for $x>x^*$ and $g(x)<\psia(x)\frac{g(x^*)}{\psia(x^*)}$ for $x<x^*$; both conditions are easy to prove.
\end{example}
%
%The scale function is $s(x)=1/x$ and the speed measure is $m(dx)=2x^2dx$. The differential operator is $Lf(x)=\frac{f''(x)}{2}+\frac{1}{x}f'(x)$ when $x > 0$.
%The domain of the infinitesimal generator is
%\begin{equation*} 
%\D_L=\{f: f,\ Lf \in \CC_b(\I),\ \frac{df^+}{ds}(0^+)=0\}. 
%\end{equation*}
%Functions $\phia$ and $\psia$ are given by 
%\begin{equation*}
%\phia(x)=\frac{1}{x}e^{-\sqrt{2\desc}\, x}\qquad (x\neq 0)
%\end{equation*}
%and 
%\begin{equation*}
%\psia(x)=\frac{2}{x}\sinh(\sqrt{2\desc}\, x)\qquad (x\neq 0)
%\end{equation*}
%the Wronskian is $\wa=2\sqrt{2\desc}.$

 %esto era el artículo
\chapter{Optimal stopping for one-dimensional diffusions:\\ \textit{the general case}}
\label{chap:diffGeneral}
\section{Introduction}
The optimal stopping problem which we deal with in this chapter, as in  \autoref{chap:diffonesided}, consist on finding a stopping time $\tau^*$ and a value function $\Va$ such that
\begin{equation*}\label{eq:ospChap3}
\Va(x)=\Ex{x}{\ea{\tau^*} \g(X_{\tau^*})} = \sup_{\tau}\Ex{x}{\ea{\tau} \g(X_{\tau})},
\end{equation*}
where $X=\{X_t\}$ is a regular one-dimensional diffusion and the supremum is taken over all stopping times. In \autoref{chap:diffonesided} we gave several results concerning one-sided problems. In the present chapter we consider results for ``two-sided''  problems, and also more general situations.

The most important result of this chapter considers reward functions $g$ that satisfy the (already stated) inversion formula
\begin{equation} \tag{\ref{eq:ginversion}} \label{eq:ginversionCap3}
g(x)=\int_{\I} \Ga(x,y) \al g(y) m(dy).
\end{equation}
It states that the optimal continuation region $\CR$, is a disjoint union of intervals $J_i=(a_i,b_i)$ that satisfy
\bd 
\int_{J_i} \psia(y)\al \g(y) m(dy)=0
\ed
if $b_i\neq r$, and also
\bd
\int_{J_i} \phia(y)\al \g(y) m(dy)=0.
\ed
if $a_i\neq \ell$. Its proof consists on  an algorithm that determines this continuation region. A remarkable characteristic of the method is that it always arrives to the solution, thus no verification is needed. Furthermore, we find the following simple expression to the value function:
\bd
\Va(x)=
\begin{cases}
g(x),&\text{if $x \notin \CR$},\\
k_1^i\phia(x)+k_2^i\psia(x) &\text{if $x \in J_i\colon i=1\ldots n$}
\end{cases}
\ed
with
 \bd k_1^i=\begin{cases}
 0,& a_i=\ell , \\
 \frac{g(a_i)}{\phia(a_i)},& b_i=r, \\
 \frac{g(b_i)\psia(a_i)-g(a_i)\psia(b_i)}{\psia(a_i)\phia(b_i)-\psia(b_i)\phia(a_i)}&\text{else};
 \end{cases}
 \ed
 and
 \bd k_2^i=\begin{cases}
 \frac{g(b_i)}{\psia(b_i)},& a_i=\ell ,\\
 0,& b_i=r \\
 \frac{g(a_i)\phia(b_i)-g(b_i)\phia(a_i)}{\psia(a_i)\phia(b_i)-\psia(b_i)\phia(a_i)},& \text{else}.
 \end{cases} 
 \ed
The previous expression for the value function is alternative to the Riesz representation
\bd
\Va = \int_{\SR}\Ga(x,y)\al g(y) m(dy),
\ed
which also holds in this case, with $\SR=\I\setminus \CR$.

The described result also holds for one-sided problems. In fact, some results of the previous chapter follow from this general results in case the general assumptions are fulfilled.

\section{Preliminary results}
We start by presenting a few preliminary results about one-dimensional diffusions.

\begin{lem} \label{psiharmonic} \label{phiharmonic}
Consider a one-dimensional diffusion $X$. Consider $a,x,b \in \I$ such that $a\leq x \leq b$. Denote by $\hit{ab}$ the hitting time of the set $\{a,b\}$, i.e.
\bd
\hit{ab}:=\inf\{t\colon X_t\in \{a,b\}\}.
\ed
 Then,
if $a>\ell$
\bd
\phia(x)=\Ex{x}{\ea{\hit{ab}}\phia(X_{\hit{ab}})},
\ed
and if $b<r$ 
\bd
\psia(x)=\Ex{x}{\ea{\hit{ab}}\psia(X_{\hit{ab}})}.
\ed
\end{lem}

\begin{proof}
Let us prove the first statement, which is a direct consequence of the discounted Dynkin's formula \eqref{eq:dynkinFormula} for functions that belong to $\D_L$. As $\phia\notin \D_L$, we consider a function $h\in \CC_b(\I)$ such that $h(x)=0$ for $x\geq a$ and $h(x)>0$ for $x<a$.
Then $f$ defined by $f(x):=\left(\Ra h \right)(x)$ belongs to $\D_L$ and there exist a constant $k>0$ such that for  $x\geq a$, $f(x)=k \phia(x)$ \citep[see][section 4.6]{itoMcKean}. 
The discounted Dynkin's formula holds for $f$, so, for $x\geq a$,
\bd
f(x)-\Ex{x}{\ea{\hit{ab}}f(X_{\hit{ab}})}=\Ex{x}{\int_0^{\hit{ab}}\al f(X_t) dt}.
\ed
From the continuity of the paths, for $t\in[0,\hit{ab}]$, $X_t\geq a$ and $\al f(X_t)=h(X_t)=0$, so the right-hand side of the previous equation vanishes. Finally taking into account the relation between $f$ and $\phia$ the conclusion follows.

The second statement is proved in an analogous way.

\end{proof}

\begin{lem}
\label{lem:Wa}
Let $X$ be a one-dimensional diffusion. Consider the function $\Wa\colon\I \to \R$ such that
\bd
\Wa(x)=\int_S \Ga(x,y)\sigma(dy),
\ed
where $\sigma$ is a positive measure and the set $S$ is
\bd
S=\I \setminus \cup_{i=1}^N J_i,
\ed 
where $N$ could be infinite, and $J_i$ are disjoint intervals included in $\I$.

Then $\Wa$ satisfies
\bd
\Wa(x)=\E_x \left(\ea{\hit{S}} \Wa(X_{\hit{S}})\right).
\ed
\end{lem}
\begin{proof}
If $x\in S$ the result is trivial, because $\hit{S}\equiv 0$. Let us consider the case $x\notin S$. In this case $x \in J_i$ for some $i$; we move on to prove that
\bd
\Ga(x,y) = \Ex{x}{\ea{\hit{S}}\Ga(X_{\hit{S}},y)}
\ed 
for all $y$ in $S$. To see this, let us denote by $a=\inf J_i$ and $b=\sup J_i$, and observe that $\hit{S}=\hit{ab}$. If $b < r$ and $y \geq b$ we have
$\Ga(x,y)=\wa^{-1}\psia(x)\phia(y)$ and by \autoref{psiharmonic} we get
\begin{align*}
\Ga(x,y)&= \wa^{-1} \Ex{x}{\ea{\hit{ab}}\psia(X_{\hit{ab}})}\phia(y)\\
&= \Ex{x}{\ea{\hit{ab}}\Ga(X_{\hit{ab}},y)},
\end{align*}
where in the second equality we have used again \eqref{eq:Garepr} and the fact that $\hit{ab}\leq y$. In the case $y\leq a$ we have to do the analogous computation.

Now we can write
\begin{align*}
\Wa(x)&=\int_S \Ga(x,y)\sigma(dy)\\
&=\int_S \Ex{x}{\ea{\hit{S}}\Ga(X_{\hit{S}},y)} \sigma(dy)\\
&=\Ex{x}{\ea{\hit{S}}\int_S \Ga(X_{\hit{S}},y) \sigma(dy)}\\
&=\Ex{x}{\ea{\hit{S}}\Wa(X_{\hit{S}},y)}
\end{align*}
and the result follows.

\end{proof}

\begin{lem}
\label{lem:waeqg}
Let $X$ be a one-dimensional diffusion and consider the function $g\colon \I \to \R$ defined by
\bd
g(x):=\int_{\I} \Ga(x,y) \sigma(dy),
\ed
where $\sigma$ is a signed measure in $\sigalg$. Consider the
 function $\Wa\colon \estados \to \R$ defined by 
\be
\label{eq:defWa2}
\Wa(x):=\int_{S} \Ga(x,y) \sigma(dy),
\ee
where the set $S$ is
\bd
S:=\I \setminus \cup_{i=1}^N J_i,
\ed 
where $N$ could be infinite, and $J_i\subset \I$ are intervals such that $J_i\cap J_j=\emptyset$ if $j\neq i$ and
\begin{itemize}
  \item $\int_{J_i}\phia(y)\sigma(dy)=0$ if there is some $x\in \I$ such that $x<y$ for all $y\in J_i$,
  \item $\int_{J_i}\psia(y)\sigma(dy)=0$ if there is some $x\in \I$ such that $x>y$ for all $y\in J_i$,
\end{itemize}

Then $g(x)=\Wa(x)$ for all $x\in S$. 
\end{lem}

\begin{proof}
From the definitions of $g$ and $\Wa$ we get
\begin{align*}
%\label{eq:gVa}
g(x)&=\int_{\I}\Ga(x,y)\sigma(dy)\\
&=\Wa(x) + \sum_{i=1}^N \int_{J_i}\Ga(x,y)\sigma(dy). \notag
\end{align*}
To prove the result it is enough to verify that if $x\in S$, then 
\bd
\int_{J_i}\Ga(x,y)\sigma(dy)=0,\quad \mbox{for all $i$.}
\ed
Consider $x\in S$, then for any $i=1\ldots N$, we have that $x\notin J_i$. Since $J_i$ is an interval either $x<y$ for all $y$ in $J_i$ or $x>y$ for all $y$ in $J_i$. Suppose the first case, from \eqref{eq:Garepr} we obtain
\begin{align*}
\int_{J_i}\Ga(x,y)\sigma(dy)=\wa^{-1}\psia(x)\int_{J_i}\phia(y)\sigma(dy)=0,
\end{align*}
where the second equality follows from hypothesis. The other case is analogous.

\end{proof}
\subsection{More about the inversion formula}
In \autoref{sec:aboutInv} we already considered the problem of determining when the inversion formula \eqref{eq:ginversionCap3} is valid. In this chapter, we need to have an easy rule to decide whether the inversion formula holds, for instance, for functions $g$ such that both $\lim_{x\to r}g(x)$ and $\lim_{x\to \ell}g(x)$ are infinite. These cases are not considered in \autoref{sec:aboutInv}. Nevertheless, an analogous result of \autoref{propInversion2} and \autoref{propInversion3} can be stated as well in this case.

\begin{prop} \label{propInversion4}
Suppose that $\I=(\ell,r)$ and that $g\colon\I\to \R$ is such that the differential operator is defined for all $x\in\I$
 %($g$ not necessarily in $\D_L$) 
 and
\be \label{eq:alIntegrable4}
 \int_{\I}\Ga(x,y) |\al g(y)| m(dy)<\infty.
\ee
Assume that for each natural number $n$, satisfying $\ell+\frac1n<r-\frac1n\in \I$ there exists a function $g_n \in \D_L$ such that $g_n(x)=g(x)$ for all $x\colon \ell+\frac1n \leq x \leq r-\frac1n$. If 
\be \label{eq:gOverPsi4}
 \lim_{z\to r^-} \frac{g(z)}{\psia(z)}=\lim_{z\to \ell^+} \frac{g(z)}{\phia(z)}=0,
\ee
then \eqref{eq:ginversionCap3} holds.
\end{prop}
\begin{proof}
The outline of the proof is the same as in \autoref{propInversion2} with minor differences:
By \eqref{eq:Ra=Ga} we get that 
\bd \int_{\I} \Ga(x,y) \al g(y) m(dy)=\Ra \al g(x). \ed
Consider the strictly increasing sequence $r_n:=r-\frac1{n-1}$ and the strictly decreasing sequence $\ell_n:=\ell+\frac1{n-1}$.
%$(r_n)\subset \I$ such that $r_n \to r$ when $n \to \infty$; 
Let $\tau_n$ be the hitting time of the set $\I\setminus (\ell_n,r_n)$, defined by
\bd
\tau_n:=\inf\{t\geq 0\colon X_t \notin (\ell_n,r_n)\}.
\ed
Observe that $\tau_n=\inf\{\hit{r_n},\hit{\ell_n}\}$.
By the continuity of the paths it can be concluded that $\tau_n \to \infty,\ (n\to\infty)$. 
Applying formula \eqref{eq:dynkinFormula} to $g_n$ and $\tau_n$ we obtain, for $x\in(\ell_n,r_n)$,
\begin{equation*}
g_n(x)=\Ex{x}{\int_0^{\tau_n} \ea{t} \al g_n(X_t) dt}+ \Ex{x}{\ea{\tau_n}g_n(X_{\tau_n})},
\end{equation*}
taking into account that $g_n(x)=g(x)$ and $\al g(x)=\al g_n(x)$ for $\ell_{n+1} <x <r_{n+1}$, from the previous equality follows that
\be \label{eq:paraconvdominada4}
g(x)=\Ex{x}{\int_0^{\tau_n} \ea{t} \al g(X_t) dt}+ \Ex{x}{\ea{\tau_n}g(X_{\tau_n})}.
\ee
About the second term on the right-hand side of the previous equation we have
\begin{align*}
\Ex{x}{\ea{\tau_n}g(X_{\tau_n})}&=\Ex{x}{\ea{\hit{r_n}} g(X_{\hit{r_n}})\ind{\{\hit{r_n}<\hit{\ell_n}\}}}+\Ex{x}{\ea{\hit{\ell_n}} g(X_{\hit{\ell_n}})\ind{\{\hit{\ell_n}<\hit{r_n}\}}}\\
&\leq \Ex{x}{\ea{\hit{r_n}} g(X_{\hit{r_n}})}+\Ex{x}{\ea{\hit{\ell_n}} g(X_{\hit{\ell_n}})}\\
&=\psia(x) \frac{g(r_n)}{\psia(r_n)}+\phia(x)\frac{g(\ell_n)}{\phia(r_n)},
\end{align*}
which taking the limit as $n\to \infty$ vanishes, by the hypotheses.
Finally, in the same way we did in \autoref{propInversion2}, we can apply Fubini's theorem, and dominated convergence theorem to conclude that the limit as $n\to \infty$ of the first term on the right-hand side of \eqref{eq:paraconvdominada4} is
\bd
\int_{\I} \Ga(x,y) \al g(y) m(dy),
\ed
thus completing the proof.

\end{proof}

\section{The two-sided case}

The following theorem is a two-sided version of \autoref{teo:verif1}.
\begin{teo} \label{teo:twosided}
Consider a one-dimensional diffusion $X$  and a reward function $g\colon \I \to \R$. Suppose $\al g(x)$ is defined and non-negative for  $x<x_{\ell}$ and $x>x_r$, where $x_\ell,x_r\colon  \ell<x_{\ell}<x_r<r$ are solution of the system of equations
 \begin{equation}
 \label{eq:deftresholds}
 \begin{cases}
  \phia(x_\ell) k_\ell(x_\ell) +\psia(x_\ell)k_r(x_r) = g(x_\ell),\\
  \phia(x_r) k_\ell(x_\ell) +\psia(x_r)k_r(x_r) = g(x_r),
 \end{cases}
\end{equation}
with
\bd
k_\ell(x_\ell)=\wa^{-1}\int_{(\ell,x_\ell)} \psia(y) \sigma(dy)
\ed
and
\bd
k_r(x_r)=\wa^{-1}\int_{(x_r,r)} \phia(y) \sigma(dy),
\ed
where $\sigma(dy)$ states for $\al g(y) m(dy)$.
Assume either that $\g$ satisfy the inversion formula \eqref{eq:ginversionCap3} or there exists $\gi$ such that $\gi(x)=g(x)$ for $x\leq x_\ell$ and for $x\geq x_r$ and $\gi$ satisfies the inversion formula \eqref{eq:ginversionCap3}. Define $\Va$ by
 \begin{equation*}
   \Va(x):=\int_{\I\setminus [x_\ell,x_r]}\Ga(x,y)\sigma(dy).
 \end{equation*}
 If $\Va(x)\geq g(x)$ for $x\in (x_\ell,x_r)$, then $\SR=\I\setminus [x_\ell,x_r]$ is the stopping region and $\Va$ is the value function.
\end{teo}

\begin{remark}
By \eqref{eq:Garepr} the system of equation \eqref{eq:deftresholds} is 
\bd
\begin{cases}
 \int_{\I\setminus [x_\ell,x_r]}\Ga(x_\ell,y) \sigma(dy) = g(x_\ell)\\
 \int_{\I\setminus [x_\ell,x_r]}\Ga(x_r,y) \sigma(dy) = g(x_\ell)
\end{cases}
\ed
and the value function $\Va$ also can be represented by
\bd
\Va(x)=\begin{cases}
g(x), \qquad & x\in \SR,\\
k_\ell \phia(x)+k_r \psia(x), &x\in [x_\ell,x_r].
\end{cases}
\ed
\end{remark}

\begin{proof}
Clearly the defined $\Va$ is an $\desc$-excessive function. Consider $x\in \SR$; by hypothesis, there exists $\gi(x)$ that satisfies \eqref{eq:ginversionCap3} and it is equal to $g$ for $x\leq x_\ell$ and for $x\geq x_r$ (in some cases $\gi$ could be $\g$ itself). Therefore
\begin{align}\label{eq:hinv}
\gi(x)&=\int_{\I}\Ga(x,y)\al \gi(y)m(dy) \notag \\
&=\int_{\I\setminus [x_\ell,x_r]}\Ga(x,y)\sigma(dy)+\int_{[x_\ell,x_r]}\Ga(x,y)\al \gi(y)m(dy) \notag \\
&=\Va(x)+\int_{[x_\ell,x_r]}\Ga(x,y)\al \gi(y)m(dy).
\end{align}
Observe that \eqref{eq:deftresholds} is $g(x_\ell)=\Va(x_\ell)$ and $g(x_r)=\Va(x_r)$. Considering $g(x_\ell)=\gi(x_\ell)$ and $g(x_r)=\gi(x_r)$, we obtain, by \eqref{eq:hinv}, that 
\bd
\int_{[x_\ell,x_r]}\Ga(x_\ell,y)\al \gi(y)m(dy)=\int_{[x_\ell,x_r]}\Ga(x_r,y)\al \gi(y)m(dy)=0
\ed
and this implies, using the explicit formula for $\Ga$, that
\bd
\int_{[x_\ell,x_r]}\phia(y)\al \gi(y)m(dy)=\int_{[x_\ell,x_r]}\psia(y)\al \gi(y)m(dy)=0.
\ed
Considering the previous equation we are in conditions to apply \autoref{lem:waeqg} to $\gi$, with 
$\sigma(dy)=\al \gi(y)m(dy)$ and $\Wa=\Va$; we conclude that $\Va(x)=\gi(x)$ for $x\in \SR$, region in which  $\gi(x)=\g(x)$. We assumed as hypothesis $\Va(x)\geq g(x)$ for $x\in (x_\ell,x_r)$, therefore we have proved that $\Va$ is a majorant of $g$.
Since, by Dynkin's characterization, the value function is the minimal $\desc$-excessive majorant of $g$, and $\Va$ is $\desc$-excessive and majorant we conclude that 
$$\Va(x) \geq \sup_{\tau}\E_x \left(\ea{\tau}g(X_\tau)\right).$$
By \autoref{lem:Wa} we get that $\Va(x)=\E_x\left(\ea{\hit{\SR}}\Va(X_{\hit{\SR}})\right)$ and considering that $\Va(x)=g(x)$ for $x\in \SR$ we conclude that
$$\Va(x)=\E_x\left(\ea{\hit{\SR}}g(X_{\hit{\SR}})\right),$$
which proves the other inequality.
\end{proof}

\subsection{Brownian motion with drift and $g(x)=|x|$}
\label{ex:absx1}
Consider $X$ a Brownian motion with drift $\mu$ \cite[see][p. 127]{borodin}. The scale function is $s(x)=(1-e^{2\mu x})/2\mu$, the speed measure is $m(dx)=2e^{2\mu x}dx$. The differential operator is $Lf(x)=f''(x)/2+\mu f'(x)$. Denoting by $\gamma=\sqrt{2\desc+\mu^2}$ we have 
\bd
\phia(x)=e^{-(\gamma+\mu)x} \quad \mbox{and}\quad \psia(x)=e^{(\gamma-\mu)x}.
\ed
The Wronskian is $\wa=2\gamma$.

\begin{example}
We consider the OSP with reward function $g(x)=|x|$. Since it is not differentiable at $0$, we can not apply \autoref{teo:diffusionGeneral} as in the previous examples; nevertheless we do can apply \autoref{teo:twosided} (observe that the reward function suggests that the problem is two-sided). 
We have to solve \eqref{eq:deftresholds}. By the nature of the problem we assume that $x_\ell < 0 < x_r$. We have 
\bd
\sigma(dx)= \begin{cases} 
(\desc x-\mu) 2e^{2\mu x}dx\qquad x>0,\\
(-\desc x+\mu) 2e^{2\mu x}dx\qquad x<0;
\end{cases}
\ed
then, if $|\mu|< \gamma$
\bd 
\int_{x_r}^{\infty} \phia(y) \sigma(dy) = e^{-(\gamma-\mu)x_r}(x_r(\gamma+\mu)+1)
\ed
and 
\bd 
\int^{x_\ell}_{-\infty} \psia(y) \sigma(dy) = e^{(\gamma+\mu)x_\ell}(-x_\ell(\gamma-\mu)+1).
\ed
The system of equations \eqref{eq:deftresholds} becomes 
\bd
\begin{cases}
e^{(\gamma-\mu)(x_{\ell}-x_r)}(x_r(\gamma+\mu)+1)=(-\gamma-\mu)x_{\ell} -1\\
e^{(\gamma+\mu)(x_{\ell}-x_r)}(-x_\ell(\gamma-\mu)+1)=(\gamma-\mu)x_r -1,
\end{cases}
\ed
which has a unique solution $(x_{\ell},x_r)$; by the application of \autoref{teo:twosided} 
we can conclude that the interval $(x_\ell,x_r)$ is the continuation region associated to the optimal stopping problem. 
We remark that the system of equations
that defines $x_\ell$ and $x_r$ is equivalent to the system obtained in \citep{salminen85}. In \autoref{ex:absx2} we consider this problem again and include some graphics with the solution.
\end{example}

\section{General case with regular reward}

During this section we consider a one-dimensional diffusion  $X$ whose speed measure has no atoms, and a reward function $g$ such that the inversion formula \eqref{eq:ginversionCap3} is fulfilled (the differential operator $L$ must be defined for all $x\in \I$). We also assume that the set 
\bd \left\{x\colon \al g(x)<0 \right\} \ed
has a finite number of connected components. We denote by $\sigma(dx)$ the measure $\al g(x)m(dx)$. 

We characterize the solution to the OSP for this kind of reward functions, providing an algorithm to find it.

\begin{teo} \label{teo:diffusionGeneral}
Under the assumptions of this section, the value function associated with the OSP is
% Consider a diffusion $X$ whose speed measure has no atoms. Assume the reward function $g$ satisfies the inversion formula \eqref{eq:ginversionCap3} and that the set $\al g(y)<0$ has a finite number of connected components. Using $\sigma(dy)$ to denote the measure $\al g(y)m(dy)$

 \begin{equation}
 \label{eq:VaDiffGeneral}
  \Va(x)=\int_{\SR} \Ga(x,y) \sigma(dy),
 \end{equation}
 where $\SR$, the stopping region, is given by $\SR=\I\setminus \CR$ and $\CR$, the continuation region, can be represented as a disjoint union of intervals $(\ell_i,r_i)$, satisfying
 \begin{itemize}
  \item if $\ell_i\neq \ell$, then $\int_{(\ell_i,r_i)}\phia(y)\sigma(dy)=0$,
  \item if $r_i\neq r$, then $\int_{(\ell_i,r_i)}\psia(y)\sigma(dy)=0$, and
  \item the set $\{x\colon \al g(x)<0\}$ is included in $\CR$
 \end{itemize}
 Furthermore, $\CR$ can be found by the algorithm \ref{algor}, to be presented further on.
\end{teo}

\begin{remark}
The condition $m(\{x\})=0$ (absence of atoms of the speed measure) is required only for simplicity of exposition. A corresponding result for the general case can also be obtained. Observe that in the one-sided case we have not considered this restriction (see, for instance, \autoref{ex:sticky} of the sticky BM).
\end{remark}

\begin{remark} \label{remark:formaVgeneral}
Once the continuation region is found, we have the integral formula for $\Va$, given in \eqref{eq:VaDiffGeneral}. Consider one of the connected components of the continuation region $J_i=(a_i,b_i)$; for $x\in J_i$ we have
\begin{align*}
\Va(x)&=\int_{\I\setminus \CR}\Ga(x,y)\sigma(dy)\\
&=\int_{(\I\setminus \CR)\cap \{x<a_i\}}\wa^{-1}\psia(y)\phia(x)\sigma(dy)+\int_{(\I\setminus \CR)\cap \{x>b_i\}}\wa^{-1}\psia(x)\phia(y)\sigma(dy)\\
&=k_1^i \phia(x)+k_2^i \psia(x).
\end{align*}
Other alternative way to find $k_1^i$ and $k_2^i$ is to take into account the fact that $\Va(a_i)=g(a_i)$ and $\Va(b_i)=g(b_i)$ and solve the system of equations
\begin{equation*}
\begin{cases}
k_1^i \phia(a_i) + k_2^i\psia(a_i)=g(a_i)\\
k_1^i \phia(b_i) + k_2^i\psia(b_i)=g(b_i)
\end{cases}
\end{equation*}
obtaining
 $$k_1^i=\frac{g(b_i)\psia(a_i)-g(a_i)\psia(b_i)}{\psia(a_i)\phia(b_i)-\psia(b_i)\phia(a_i)}$$
 and
 $$k_2^i=\frac{g(a_i)\phia(b_i)-g(b_i)\phia(a_i)}{\psia(a_i)\phia(b_i)-\psia(b_i)\phia(a_i)}.$$
In the particular case in which $a_i=\ell$ we have $k_1^i=0$ and $k_2^i=g(b_i)/\psia(b_i)$, and if $b_i=r$ then $k_1=g(a_i)/\phia(a_i)$ and $k_2=0$. We have the following alternative formula for $\Va$:
\begin{equation}
\label{eq:VaMasHumana}
\Va(x)=
\begin{cases}
g(x)&\text{for $x \notin \CR$}\\
k_1^i\phia(x)+k_2^i\psia(x) &\text{for $x \in J_i\colon i=1\ldots n$}
\end{cases}
\end{equation}
\end{remark}

The proof of this theorem is essentially the algorithm to find the intervals that constitute the continuation region (\autoref{algor}) and it requires some previous results. Anyway we start by giving a brief idea of the algorithm as a motivation:
\begin{enumerate}
\item split the set $\al g(x)<0$ in $J_1=(a_1,b_1), \ldots, J_n=(a_n,b_n)$ disjoint intervals with $a_1<b_1<a_2<b_2<\ldots <a_n<b_n$;
\item for each $J_i$ consider a bigger interval ${\bar{J}_i}$ contained in the continuation region (see \autoref{cond:continuation});
\item if ${\bar{J}_i}$ are disjoint intervals then $\CR=\bigcup_i {\bar{J}_i}$;
\item else, consider, for each connected component $A$ of $\bigcup_i {\bar{J}_i}$, a unique interval $(a',b')$, where $a'=\inf\{a_i: a_i\in A\}$ and $b'=\sup\{b_i: b_i\in A\}$ and return to step 2.
\end{enumerate}
%Clearly we have to give a precise description of the algorithm and to prove that it indeed gives the solution.

Before giving the algorithm and the proof of \autoref{teo:diffusionGeneral} we need some preliminary results.

Given an interval $J\subseteq \I$ we define the signed measure $\sigma_J$ by
 \begin{equation*}
 \sigma_J(dx)=
 \begin{cases}
   \sigma(dx) &\mbox{if $x\in J$ or $\al g(x)>0$,}\\
   0 &\mbox{else}.
 \end{cases}
 \end{equation*}
Observe that $\sigma_J$ is a positive measure out of $J$, and it is equal to $\sigma$ into $J$.

\begin{cond} \label{cond:continuation}
We say that the pair of intervals $(J,\bar{J}) \colon J\subseteq\bar{J}\subseteq \I$ satisfies the \autoref{cond:continuation} if the following assertions hold:
\begin{enumerate}[(i)]
\item \label{i} both, $\int_J \phia(x) \sigma(dx)\leq 0$ and $\int_J \psia(x) \sigma(dx)\leq 0$;
\item \label{ii} if $\inf\{\bar{J}\}\neq \ell$ then $\int_{\bar{J}}\phia(x)\sigma_J(dx)=0$;
\item \label{iii}if $\sup\{\bar{J}\}\neq r$ then $\int_{\bar{J}}\psia(x)\sigma_J(dx)=0$; and
\item \label{iv}for every $x\in \bar{J}$, $\int_{\bar{J}}\Ga(x,y)\sigma_J(dy)\leq 0.$
\end{enumerate}
\end{cond}
We denote by $\sigma^+(dx)$ the measure 
\bd \sigma^+(dx):= \sigma(dx)\ind{\{\al g(x)>0\}}.\ed

\begin{lem}
\label{lem:pasobase}
 Under the assumptions of this section, consider an open interval $J\subseteq \I$, such that $\sigma(dx)<0$ for $x\in J$.
 Then, there exists an interval $\bar{J}$ such that $(J,\bar{J})$ satisfies \autoref{cond:continuation}
\end{lem}
\begin{proof}
 Consider $J$ to be $(a,b)$. Assertion (\ref{i}) in \autoref{cond:continuation} is clearly fulfilled. Without loss of generality (denoting by $\phia$ the result of multiplying $\phia$ by the necessary positive constant) we may assume 
 \bd \int_J \psia(x)\sigma(dx)=\int_J \phia(x)\sigma(dx)<0. \ed
 Under this assumption, $\phia(a)<\psia(a)$ and $\phia(b)>\psia(b)$. Consider 
 \begin{equation*}
 x_1:=\inf\left\{x\in [\ell,a]\colon \int_{(x_1,b)}\phia(x)\sigma_J(dx)<0\right\}.
 \end{equation*}
 Since $\phia(x)>\psia(x)$ for $x\leq a$ and $\sigma_J(dx)$ is non-negative in the same region we conclude that $\int_{(x_1,b)}\psia(x)\sigma_J(dx)\leq 0$. Consider $y_1>b$ defined by
 \begin{equation*}
 y_1:=\sup\left\{x\in [b,r]\colon \int_{(x_1,y_1)}\psia(x)\sigma_J(dx)<0\right\}. 
 \end{equation*}
 Now we consider $x_2\geq x_1$ as 
  \begin{equation*}
 x_2:=\inf\left\{x\in [\ell,a]\colon \int_{(x_2,y_1)}\phia(x)\sigma_J(dx)<0\right\}
  \end{equation*}
 and $y_2\geq y_1$ as
  \begin{equation*}
 y_2:=\sup\left\{x\in [b,r]\colon \int_{(x_2,y_2)}\psia(x)\sigma_J(dx)<0\right\}. 
  \end{equation*}
 Following in the same way we obtain two non-decreasing sequences $\ell\leq \{x_n\}\leq a$ and $b\leq \{y_n\}\leq r$. By construction, the interval $\bar{J}=(\lim x_n, \lim y_n)$ satisfies (\ref{ii}) and (\ref{iii}) in \autoref{cond:continuation}. To prove (\ref{iv}), first we find $k_1(x)$ and $k_2(x)$ such that
 $$
 \begin{cases}
 k_1(x) \psia(a)+k_2(x) \phia(a)=\Ga(x,a) \\
  k_1(x) \psia(b)+k_2(x) \phia(b)=\Ga(x,b). 
 \end{cases}
 $$
 Solving the system we obtain 
 $$k_1(x)=\frac{\Ga(x,b)\phia(a)-\Ga(x,a)\phia(b)}{\psia(b)\phia(a)-\psia(a)\phia(b)}$$
 and
 $$k_2(x)=\frac{\Ga(x,a)\psia(b)-\Ga(x,b)\psia(a)}{\psia(b)\phia(a)-\psia(a)\phia(b)}.$$
Let us see that $k_1(x),k_2(x)\geq 0$ for any $x\in \bar{J}$: using the explicit formula for $\Ga$ it follows that
\begin{equation*}
k_1(x) =
\begin{cases}
 0 &\mbox{for $x\leq a$,} \\
 \wa^{-1}\phia(b)\frac{\psia(x)\phia(a)-\psia(a)\phia(x)}{\psia(b)\phia(a)-\psia(a)\phia(b)} &\mbox{for $x\in (a,b)$,}\\
 \wa^{-1}\psia(x) &\mbox{for $x\geq b$}.
\end{cases}
\end{equation*}
 The only non-trivial case is when $x\in (a,b)$. The numerator and denominator are non-negative because $\phia$ is decreasing and $\psia$ increasing. The case of $k_2$ is completely analogous.
 
 Considering $h(x,y)=k_1(x)\psia(y)+k_2(x)\phia(y)$, it can be seen (discussing for the different positions of $x$ and $y$ with respect to $a$ and $b$) that for all $x\in \bar{J}$, $h(x,y)\leq \Ga(x,y)$ for $y \in (a,b)$ and $h(x,y)\geq \Ga(x,y)$ for $y\notin (a,b)$. From these inequalities we conclude that
 \begin{equation*}
 \int_{\bar{J}}\Ga(x,y)\sigma_J(dy) \leq \int_{\bar{J}}h(x,y)\sigma_J(dy) \leq 0;
 \end{equation*}
 where the first inequality is consequence of $\sigma_J(dy)\geq 0$ in $\I\setminus J$ and $\sigma_J(dy)\leq 0$ in $J$; and the second one is obtained fixing $x$ and observing that $h(x,y)$ is a linear combination of $\psia$ and $\phia$ with non-negative coefficients.
 
\end{proof}

\begin{lem} \label{lem:casoIntervalosPropios}
Under the assumptions of this section, consider $J_1=(a_1,b_1)$, $J_2=(a_2,b_2)$ such that $b_1<a_2$ and $\al g(x)\geq 0$ for $x$ in $(b_1,a_2)$. Let ${\bar{J}_1}=(\bar{a}_1,{\bar{b}_1})$ and ${\bar{J}_2}=({\bar{a}_2},{\bar{b}_2})$ be intervals such that ${\bar{a}_1}>\ell$, ${\bar{b}_1}<r$, ${\bar{a}_2}>\ell$, ${\bar{b}_2}<r$. Suppose that the two pairs of intervals $(J_1,{\bar{J}_1})$, $(J_2,{\bar{J}_2})$ satisfy \autoref{cond:continuation}. 

If ${\bar{J}_1}\cap {\bar{J}_2}\neq \emptyset$ then, considering $J=(a_1,b_2)$, there exists an interval $\bar{J}$ such that $(J,\bar{J})$ satisfies \autoref{cond:continuation}.
\end{lem}
\begin{proof}
By hypothesis
\bd
\int_{{\bar{J}_i}}\phia(x)\sigma_{J_i}(dx)=\int_{{\bar{J}_i}}\psia(x)\sigma_{J_i}(dx)=0.
\ed
Then
\bd
\int_{{\bar{J}_1}\cup {\bar{J}_2}}\phia(x)\sigma(dx)=-\int_{{\bar{J}_1}\cap {\bar{J}_2}}\phia(x)\sigma^+(dx)
\ed
and
$$\int_{{\bar{J}_1}\cup {\bar{J}_2}}\psia(x)\sigma(dx)=-\int_{{\bar{J}_1}\cap {\bar{J}_2}}\psia(x)\sigma^+(dx).$$
We assume, without loss of generality, that 
$$\int_{{\bar{J}_1}\cap {\bar{J}_2}}\phia(x)\sigma^+(dx)=\int_{{\bar{J}_1}\cap {\bar{J}_2}}\psia(x)\sigma^+(dx)>0$$
and therefore, denoting by $(a',b')$ the interval ${\bar{J}_1}\cup{\bar{J}_2}$, we get:
\bd
\int_{(a',b')}\phia(x)\sigma(dx)=\int_{(a',b')}\psia(x)\sigma(dx)<0;
\ed
$\psia(a')\leq\phia(a')$; and $\psia(b')\geq \phia(b')$. The same procedure in the proof of \autoref{lem:pasobase}, allow us to construct an interval $\bar{J}$ such that $(J,\bar{J})$ satisfy (\ref{i}), (\ref{ii}) and (\ref{iii}) in \autoref{cond:continuation}. Let us prove (\ref{iv}):
If $x<a_1$ we have $\Ga(x,y)=\wa^{-1}\psia(x)\phia(y)$ for $y\geq a_1$ and $\Ga(x,y)\leq \wa^{-1}\psia(x)\phia(y)$ for $y\leq a_1$; since $\sigma_{J}(dy)$ is non-negative in $y\leq a_1$ we find
\bd
\int_{\bar{J}} \Ga(x,y)\sigma_J(dy)\leq \wa^{-1}\psia(x) \int_{\bar{J}} \phia(y)\sigma_J(dy)\leq 0.
\ed
An analogous argument prove the assertion in the case $x>b_2$. Now consider $x\in J$, suppose $x<\min\{a_2,{\bar{b}_1}\}$ (in case $x>\max\{b_1,{\bar{a}_2}\}$ an analogous argument is valid), we get 
\begin{align*}
\int_{\bar{J}}\Ga(x,y)\sigma_J(dy)&=\int_{{\bar{J}_1}}\Ga(x,y)\sigma_{J_1}(dy)+\int_{{\bar{J}_1}}\Ga(x,y)(\sigma_{J}-\sigma_{J_1})(dy)\\
&\quad+\int_{\bar{J}\setminus {\bar{J}_1}}\Ga(x,y)\sigma_{J}(dy),
\end{align*}
where $\int_{{\bar{J}_1}}\Ga(x,y)\sigma_{J_1}(dy)\leq 0$ by hypothesis. 
We move on to prove that the sum of the second and the third terms on the right-hand side of the previous equation are non-positive, thus completing the proof: Observe that 
$$\Ga(x,y)\leq \wa^{-1}\psia(x)\phia(y)$$
and 
$$\Ga(x,y)=\wa^{-1}\psia(x)\phia(y)\quad  (y\geq \min\{a_2,{\bar{b}_1}\})$$ 
The measure $(\sigma_J-\sigma_{J_1})$ has support in $J_2$, where the previous equality holds. The measure $\sigma_J(dy)$ is positive for $y<a_1$ where we do not have the equality, then
\begin{align*}
&\int_{{\bar{J}_1}}\Ga(x,y)(\sigma_{J}-\sigma_{J_1})(dy) +\int_{\bar{J}\setminus {\bar{J}_1}}\Ga(x,y)\sigma_{J}(dy)\\
&\leq \wa^{-1}\psia(x)\left(\int_{{\bar{J}_1}}\phia(y)(\sigma_{J}-\sigma_{J_1})(dy) +\int_{\bar{J}\setminus {\bar{J}_1}}\phia(y)\sigma_{J}(dy)\right)\leq 0,
\end{align*}
where the last inequality is a consequence of
\begin{align*}
\int_{\bar{J}}\phia(y)\sigma_J(dy)&=\int_{{\bar{J}_1}}\phia(y)\sigma_{J_1}(dy)+\int_{{\bar{J}_1}}\phia(y)(\sigma_{J}-\sigma_{J_1})(dy)\\
&\quad+\int_{\bar{J}\setminus {\bar{J}_1}}\phia(y)\sigma_{J}(dy)\leq 0,
\end{align*}
and 
$$\int_{{\bar{J}_1}}\phia(y)\sigma_{J_1}(dy)=0.$$
This completes the proof.

\end{proof}

\begin{lem} \label{lem:caso1}
Under the assumptions of this section, consider the interval $J=(a,b)$ and $\bar{J}=(\ell,\bar{b})$ (with $\bar{b}<r$) such that $(J,\bar{J})$ satisfies \autoref{cond:continuation}. Then, there exists $b'\geq \bar{b}$ such that $(J'=(\ell,b),\bar{J'}=(\ell,b'))$ satisfies \autoref{cond:continuation}.
\end{lem}
\begin{proof}
By hypothesis we know
\bd
\int_{\bar{J}}\psia(y)\sigma_{J}(dy)=0.
\ed
It follow that
\bd
\int_{\bar{J}}\psia(y)\sigma_{J'}(dy)\leq 0.
\ed
Consider $b'=\sup\{x \in [\bar{b},r) \colon \int_{(\ell,x)}\psia(y)\sigma_{J'}(dy)\leq 0\}$. It is clear that
$$\int_{\bar{J'}}\psia(y)\sigma_{J'}(dy)\leq 0,$$
with equality if $b'=r$. This proves (\ref{iii}) in \autoref{cond:continuation}. Now we prove (\ref{iv}). Consider
\begin{align} \label{eq:splitIntegral1}
\int_{\bar{J'}}\Ga(x,y)\sigma_{J'}(dy)&=\int_{\bar{J}}\Ga(x,y)\sigma_{J}(dy)+\int_{\bar{J}}\Ga(x,y)(\sigma_{J'}-\sigma_{J})(dy) \notag \\
&\qquad +\int_{\bar{J'}\setminus \bar{J}}\Ga(x,y)\sigma_{J'}(dy).
\end{align}
The first term on the right-hand side is non-positive by hypothesis. 
Let us analyse the sum of the remainder terms. 
Considering the previous decomposition with $\psia(y)$ instead of $\Ga(x,y)$, and taking $\int_{\bar{J}}\psia(y)\sigma_{J}(dy)= 0$ into account, we obtain
\begin{equation}
\label{eq:intdiferpsi}
\int_{\bar{J}}\psia(y)(\sigma_{J'}-\sigma_{J})(dy)+\int_{\bar{J'}\setminus \bar{J}}\psia(y)\sigma_{J'}(dy)\leq 0.
\end{equation}
Consider $k(x)$ such that $k(x)\psia(\bar{b})=\Ga(x,\bar{b})$; we have $k(x)\psia(y)\leq\Ga(x,y)$ if $y\leq \bar{b}$ and $k(x) \psia(y)\geq\Ga(x,y)$ if $y\geq \bar{b}$. 
Also note that $(\sigma_{J'}-\sigma_J)(dy)$ is non-positive in $\bar{J}$ and $\sigma_{J'}$ is non-negative in $\bar{J'}\setminus \bar{J}$. We get
\begin{align*}
& \int_{\bar{J}}\Ga(x,y)(\sigma_{J'}-\sigma_{J})(dy) +\int_{\bar{J'}\setminus \bar{J}}\Ga(x,y)\sigma_{J'}(dy) \\ 
&\qquad \leq \int_{\bar{J}}\psia(y)(\sigma_{J'}-\sigma_{J})(dy)+\int_{\bar{J'}\setminus \bar{J}}\psia(y)\sigma_{J'}(dy)\leq 0.
\end{align*}
This completes the proof of (\ref{iv}). To prove (\ref{iii}), i.e.
$$\int_{\bar{J'}}\phia(y)\sigma_{J'}(dy)\leq 0,$$
consider $k>0$ such that  $\phia(\bar{b})=k \psia(\bar{b})$. It is easy to see that all the steps considered in proving (\ref{iv}) also work in this case. 

\end{proof}

\begin{lem}\label{lem:caso2}
Under the assumptions of this section, consider the interval $J=(a,b)$ and $\bar{J}=(\bar{a},r)$ (with $\bar{a}>\ell$), such that $(J,\bar{J})$ satisfies \autoref{cond:continuation}. Then, there exists $a'\leq \bar{a}$ such that $(J'=(a,r),\bar{J'}=(a',r))$ satisfies \autoref{cond:continuation}.
\end{lem}
\begin{proof}
Analogous to the proof of the previous lemma.

\end{proof}

\begin{lem}\label{lem:casoizq}
Under the assumptions of this section, consider $J_1=(\ell,b_1)$, $J_2=(a_2,b_2)$ such that: $b_1<a_2$; and $\al g(x)\geq 0$ for $x$ in $(b_1,a_2)$. Let ${\bar{J}_1}=(\ell,{\bar{b}_1})$ and ${\bar{J}_2}=({\bar{a}_2},{\bar{b}_2})$ be intervals such that: ${\bar{b}_1}<r$; ${\bar{a}_2}>\ell$; and ${\bar{b}_2}<r$. Suppose that the two pairs of intervals $(J_1,{\bar{J}_1})$, $(J_2,{\bar{J}_2})$ satisfy \autoref{cond:continuation}. 
If ${\bar{J}_1}\cap {\bar{J}_2}\neq \emptyset$ then, considering $J=(\ell,b_2)$, there exists $\bar{b}$ such that $(J,\bar{J}=(\ell,\bar{b}))$ satisfies \autoref{cond:continuation}.
\end{lem}
\begin{proof}
Define $\bar{b}=\sup\{x \in [{\bar{b}_2},r) \colon \int_{(\ell,x)}\psia(y)\sigma_{J}(dy)\leq 0\}$ (note that $\bar{b}_2$ belongs to the set). We have
\begin{equation}
\label{eq:intPsiaMenor0}
\int_{\bar{J}}\psia(y)\sigma_J(dy)\leq 0,
\end{equation}
with equality if $\bar{b}<r$, proving (\ref{ii}) in \autoref{cond:continuation}. To prove (\ref{iv}) we split the integral as follows:
\begin{align}
\label{eq:splitintegralGa}
\int_{\bar{J}}\Ga(x,y)\sigma_J(dy)&=\int_{{\bar{J}_1}}\Ga(x,y)\sigma_{J_1}(dy)+\int_{{\bar{J}_2}}\Ga(x,y)\sigma_{J_2}(dy)\\
&\quad -\int_{{\bar{J}_1}\cap {\bar{J}_2}}\Ga(x,y)\sigma_J^+(dy)+\int_{\bar{J}\setminus({\bar{J}_1}\cup {\bar{J}_2})}\Ga(x,y)\sigma_{J}(dy) \notag 
\end{align}
where $\sigma_J^+$ is the positive part of $\sigma_J$. Considering the same decomposition as in \eqref{eq:splitintegralGa} with $\psia(y)$, instead of $\Ga(x,y)$, and also considering: equation \eqref{eq:intPsiaMenor0}; $\int_{{\bar{J}_1}}\psia(y)\sigma_{J_1}(dy)=0$; and  $\int_{{\bar{J}_2}}\psia(y)\sigma_{J_2}(dy)=0$, we obtain
\begin{equation}
-\int_{{\bar{J}_1}\cap {\bar{J}_2}}\psia(y)\sigma_J^+(dy)+\int_{\bar{J}\setminus({\bar{J}_1}\cup {\bar{J}_2})}\psia(y)\sigma_{J}(dy)\leq 0.
\end{equation}
For every $x$ consider $k(x)\geq 0$ such that $k(x)\psia({\bar{b}_2})=\Ga(x,{\bar{b}_2})$.
We have $k(x)\psia({\bar{b}_2})\leq \Ga(x,{\bar{b}_2})$ for $y\leq {\bar{b}_2}$ and $k(x)\psia({\bar{b}_2})\geq \Ga(x,{\bar{b}_2})$ for $y\geq {\bar{b}_2}$ and therefore
\begin{align*}
&-\int_{{\bar{J}_1}\cap {\bar{J}_2}}\Ga(x,y)\sigma_J^+(dy)
+\int_{\bar{J}\setminus({\bar{J}_1}\cup {\bar{J}_2})}\Ga(x,y)\sigma_{J}(dy) \\ 
&\quad = k(x) \left(-\int_{{\bar{J}_1}\cap {\bar{J}_2}}\psia(y)\sigma_J^+(dy)
+\int_{\bar{J}\setminus({\bar{J}_1}\cup {\bar{J}_2})}\psia(y)\sigma_{J}(dy)\right) \leq 0.
\end{align*}
The first two terms on the right-hand side of equation \eqref{eq:splitintegralGa} are also non-positive, and we conclude that (\ref{iv}) in \autoref{cond:continuation} holds. To prove (\ref{ii}) we consider the decomposition in \eqref{eq:splitintegralGa} with $\phia(y)$ instead of $\Ga(x,y)$ and $k\geq 0$ such that $k\psia({\bar{b}_2})=\phia({\bar{b}_2})$; the same considerations done to prove (\ref{iv}) conclude the result in this case.

\end{proof}

\begin{lem}\label{lem:casoder}
Under the assumptions of this section, consider $J_1=(a_1,b_1)$, $J_2=(a_2,r)$ such that: $b_1<a_2$; and $\al g(x)\geq 0$ for $x$ in $(b_1,a_2)$. Let ${\bar{J}_1}=({\bar{a}_1},{\bar{b}_1})$ and ${\bar{J}_2}=({\bar{a}_2},r)$ intervals such that: ${\bar{a}_1}>\ell$; ${\bar{b}_1}<r$; and ${\bar{a}_2}>\ell$. Suppose that the two pairs of intervals $(J_1,{\bar{J}_1})$, $(J_2,{\bar{J}_2})$ satisfy \autoref{cond:continuation}. If ${\bar{J}_1}\cap {\bar{J}_2}\neq \emptyset$ then, considering $J=(a_1,r)$, there exists $\bar{a}$ such that $(J,\bar{J}=(\bar{a},r))$ satisfies \autoref{cond:continuation}.
\end{lem}
\begin{proof}
Analogous to the previous lemma.

\end{proof}

\begin{lem} \label{lem:caso2lados}
Under the assumptions of this section, consider $J_1=(\ell,b_1)$, $J_2=(a_2,r)$ such that: $b_1<a_2$; and $\al g(x)\geq 0$ for $x$ in $(b_1,a_2)$. Let ${\bar{J}_1}=(\ell,{\bar{b}_1})$ and ${\bar{J}_2}=({\bar{a}_2},r)$ intervals such that the two pairs of intervals $(J_1,{\bar{J}_1})$, $(J_2,{\bar{J}_2})$ satisfy \autoref{cond:continuation}. If ${\bar{J}_1}\cap {\bar{J}_2}\neq \emptyset$ then for all $x\in \I$,
$$\int_{\I}\Ga(x,y)\sigma(dy)\leq 0.$$
\end{lem}
\begin{proof}
Consider the following decomposition of the integral
\begin{align*}
\int_{\I}\Ga(x,y)\sigma(dy)&=\int_{{\bar{J}_1}}\Ga(x,y)\sigma_{J_1}(dy)+\int_{{\bar{J}_2}}\Ga(x,y)\sigma_{J_2}(dy)\\
&\qquad -\int_{{\bar{J}_1}\cap {\bar{J}_2}} \Ga(x,y)\sigma^+(dy).
\end{align*}
Observing that the three terms on the right-hand side are non-positive, the lemma is proved.

\end{proof}

Now we state the algorithm to find the continuation region in the OSP corresponding to \eqref{eq:ospChap3}.
\begin{algorithm} \label{algor}
(Starting from a subset of the continuation region, in subsequent steps, increase the considered subset until finding the actual continuation region)

\begin{itemize}
\item[BS.]\emph{(base step)} Consider disjoint intervals $J_1,\ldots,J_n \subseteq \I$ such that
$$\left\{x\in \I:\al g(x)<0\right\}=\bigcup_{i=1}^n J_i.$$
Consider for each $i$,  ${\bar{J}_i}$ such that $(J_i,{\bar{J}_i})$ satisfies \autoref{cond:continuation} (this can be done in virtue of \autoref{lem:pasobase}). Define 
$$C=\left\{(J_i,{\bar{J}_i}):i=1\ldots n\right\},$$
and go to the iterative step (IS) with $C_1$.

\item[IS.]\emph{(iterative step)} At this step we assume given a set $C$ of pair of intervals satisfying \autoref{cond:continuation}. We assume the notation\footnote{We remark that at different moments the algorithm execute this step, the notation refers to different objects, e.g. the set $C$ is not always the same set.}
\bd
C=\{(J_i=(a_i,b_i),{\bar{J}_i}=({\bar{a}_i},{\bar{b}_i}))\colon i=1 \ldots n\},
\ed
with $a_i < a_j$ if $i<j$ (the intervals are ordered) and $b_i<a_{i+1}$ (the intervals are disjoint)

\begin{itemize}
\item If $C$ is empty, the algorithm is finished and the continuation region is empty.

\item Else, if for some $j$, ${\bar{J}_j}=\I$, the algorithm is finished and the continuation region is $\I$.

\item Else, if the intervals ${\bar{J}_i}$ are pairwise disjoint, the algorithm is finished and the continuation region is 
$$\CR=\bigcup_{i=1}^n {\bar{J}_i}$$

\item Else, if ${\bar{a}_j}=\ell$ for some $j>1$, add to $C$ the pair $(J=(\ell,b_j),\bar{J})$ satisfying \autoref{cond:continuation}, and remove from $C$ the pairs $(J_i,{\bar{J}_i})$ for $i=1\ldots j$. Observe that the existence of $\bar{J}$ is proved in \autoref{lem:caso1}. Return to the iterative step (IS).

\item Else, if ${\bar{b}_j}=r$ for some $j<n$, add to $C$ the pair $(J=(a_j,r),\bar{J})$ satisfying \autoref{cond:continuation}, and remove from $C$ the pairs $(J_i,{\bar{J}_i})$ for $i=j\ldots n$ (observe that the existence of $\bar{J}$ is proved in \autoref{lem:caso2}). Return to the iterative step (IS).

\item Else, if for some $j$, ${\bar{J}_j}\cap {\bar{J}_{j+1}}\neq \emptyset$, remove from $C$ the pairs $j$ and $j+1$, and add to $C$ the pair $(J=(a_j,b_j+1),\bar{J})$ satisfying \autoref{cond:continuation} (its existence is guaranteed, depending on the situation, \autoref{lem:casoIntervalosPropios}, \autoref{lem:casoizq}, \autoref{lem:casoder} or \autoref{lem:caso2lados}). Return to the iterative step (IS).
\end{itemize}
\end{itemize}
\end{algorithm}

We are now ready to prove \autoref{teo:diffusionGeneral}.

\begin{proof}[Proof of \autoref{teo:diffusionGeneral}]
Denote by $\CR=\{J_1,\ldots,J_n\}$ the set resulting from \autoref{algor}. It is clearly a disjoint union of intervals and it is easy to see that it satisfies all the conditions stated in the theorem. It remains to prove that this is in fact the continuation region associated with the optimal stopping problem. We use the Dynkin's characterization as the minimal $\desc$-excessive majorant to prove that
\bd
\Va(x):=\int_{\I\setminus \CR}\Ga(x,y)\sigma(dy)
\ed
is the value function. Since $\sigma(dy)$ is non-negative in $\I\setminus \CR$ we have that $\Va$ is $\desc$-excessive. 
%Call $J_i:i=1\ldots n$ the disjoint intervals compounding the set $\CR$. 
For $x\in \I$, we have
\begin{align}
\label{eq:gVa}
g(x)&=\int_{\I}\Ga(x,y)\sigma(dy)\\
&=\Va(x) + \sum_{i=1}^n \int_{J_i}\Ga(x,y)\sigma(dy). \notag
\end{align}
Observe that, on the one hand,
$$\int_{J_i}\Ga(x,y)\sigma(dy)=0 \quad (x \notin J_i),$$
due to the fact that, if $x<J_i$ then  $\Ga(x,y)=\wa^{-1}\psia(x)\phia(y)$ and $\int_{J_i}\phia(y)\sigma(dy)=0$ and, on the other hand,
$$\int_{J_i}\Ga(x,y)\sigma(dy)\leq 0 \quad (x \in J_i).$$
Combining this facts with equation \eqref{eq:gVa}, we conclude that 
$$\Va(x)\geq g(x)\quad (x \in \I),$$
and, in fact, the equality holds for $x \in \I\setminus\CR$,
what can be seen also as an application of \autoref{lem:waeqg}.
We have proved that $\Va$ is a  majorant of $g$. We have, up to now, $\Va(x)\geq \sup_{\tau}\E_x\left(\ea{\tau}g(X_\tau)\right)$. Finally observe that, denoting by $\SR$ the set $\I\setminus\CR$
$$\Va(x)=\E_x\left(\ea{\hit{S}}\Va(X_{\hit{\SR}})\right)=\E_x \left(\ea{\hit{\SR}}g(X_{\hit{\SR}})\right),$$
where the first equality is a consequence of \autoref{lem:Wa}. We conclude that $\Va$ is the value function and that $\SR$ is the stopping region, finishing the proof.

\end{proof}

\subsection{Implementation}
\label{sec:implementation}
To compute in practice the optimal stopping region, following the \autoref{algor}, it can be necessary a computational implementation of some parts of the algorithm. In fact, to solve our examples we have implemented a script in R \cite[see][]{Rsoftware} that receives as input:
\begin{itemize} 
\item function $\al g$;
\item the density of measure $m$;
\item the atoms of measure $m$;
\item functions $\phia$ and $\psia$;
\item two numbers $a$, $b$ that are interpreted as the left and right endpoint of an interval $J$
\end{itemize}
and produce as output two numbers $a'\leq a$, $b'\geq b$ such that $(J,(a',b'))$ satisfy \autoref{cond:continuation}. It is assumed that the interval $J$ given as input satisfies the necessary conditions to ensure the existence of $J'$.

To compute $a'$ and $b'$ we use a discretization of the given functions and compute the corresponding integrals numerically. We follow the iterative procedure presented in the proof of \autoref{lem:pasobase}.

Using this script the examples are easily solved following \autoref{algor}.

\subsection{Brownian motion and polynomial reward}

The previous results are specially suited for non-monotone reward functions. 

\begin{example}[$\desc=2$] Consider a standard Brownian motion $X$ as in \autoref{ex:Brownian}. Consider the reward function $g$ defined by
\bd
g(x):=-(x-2)(x-1)x(x+1)(x+2),
\ed
and the discount factor $\desc=2$. 
To solve the optimal stopping problem \eqref{eq:ospChap3}, by the application of \autoref{algor}, we start by finding the set $\al g(x)<0$. Remember that the infinitesimal generator is given by $Lg(x)=g''(x)/2$. After computations, we find that
$$\{x\colon \al g(x)<0\}=\bigcup_{i=1}^3 J_i,$$
with $J_1\simeq (-2.95,-1.15)$, $J_2\simeq (0,1.15)$ and $J_3 \simeq (2.95,\infty)$. Computing ${\bar{J}_i}$, as is specified in the (base step) of the algorithm in the proof of \autoref{teo:diffusionGeneral}, we find
${\bar{J}_1}\simeq (-3.23,-0.50)$, ${\bar{J}_2}\simeq (-0.36,1.43)$ and ${\bar{J}_3} \simeq (1.78,\infty)$. Observing that the intervals are disjoint we conclude that the continuation region is given by ${\bar{J}_1}\cup {\bar{J}_2} \cup {\bar{J}_3}$. Now, by the application of equation \eqref{eq:VaMasHumana}, we find the value function, which is shown in \autoref{fig:poly-alpha2}. Note that the smooth fit principle holds in the five contact point.
\begin{figure}
 \begin{center} 
\includegraphics[scale=.8]{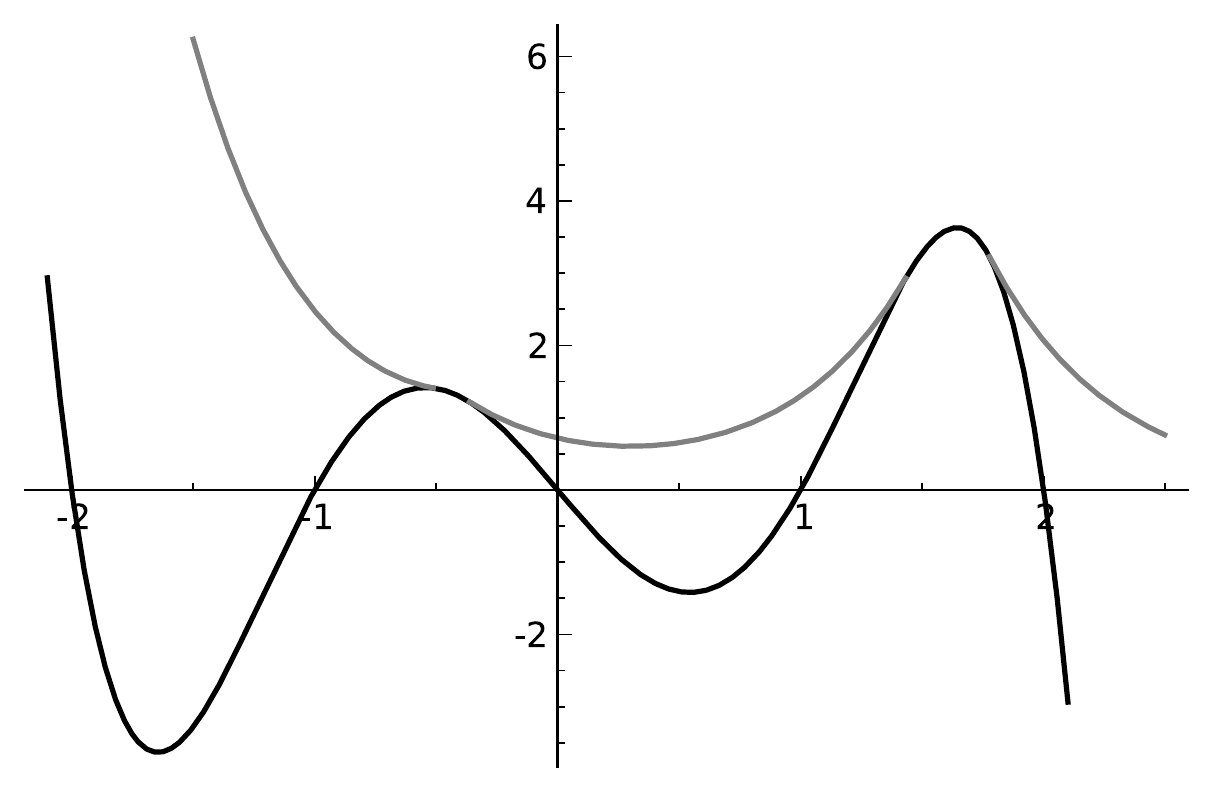}
\caption{\label{fig:poly-alpha2} OSP for the standard BM and a 5th. degree polynomial: $g$ (black), $\Va$ (gray, when different from $g$). Parameter $\desc=2$. In \autoref{fig:poly-alpha2-zoom} zooms of the interesting parts are shown.} 
 \end{center}
\end{figure}

\begin{figure}
 \begin{center} 
\includegraphics[scale=.5]{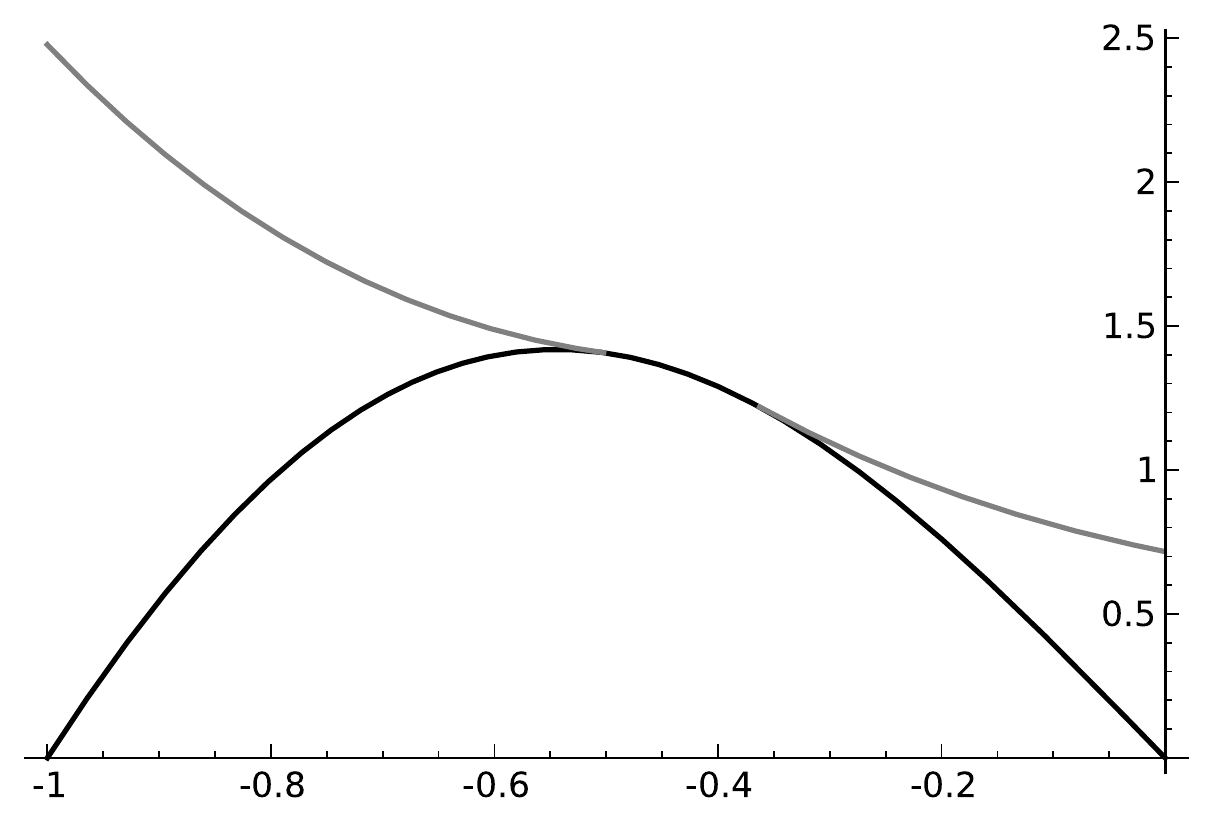}
\includegraphics[scale=.5]{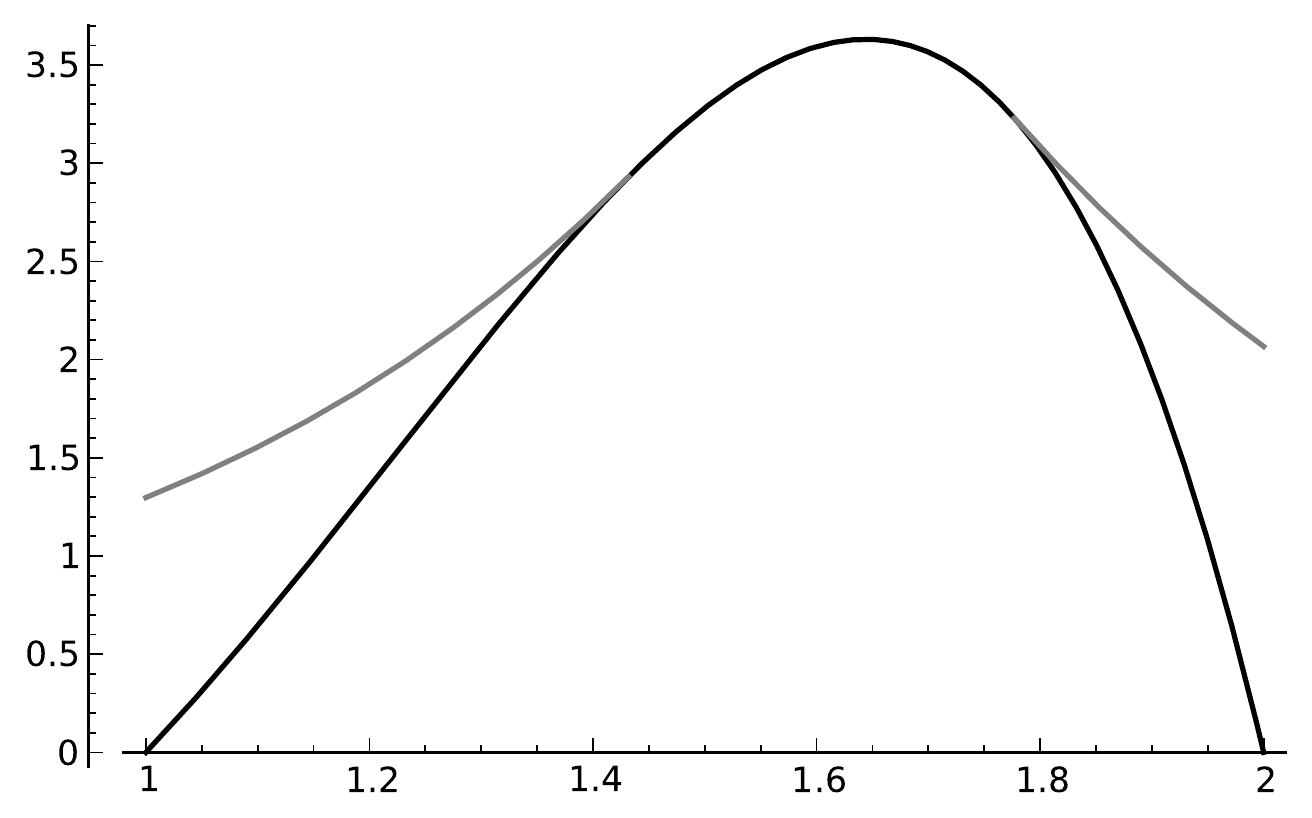}
\caption{\label{fig:poly-alpha2-zoom} Zooms of \autoref{fig:poly-alpha2} to appreciate the smooth fit principle.}
 \end{center}
\end{figure}

\end{example}

\begin{example}[case $\desc=1.5$]
Consider the process and the reward as in the previous example but with a slightly smaller discount, $\desc=1.5$. 
We have again
\[\{x\colon \al g(x)<0\}=\bigcup_{i=1}^3 J_i,
\]
but with $J_1\simeq (-3.21,-1.17)$, $J_2\simeq (0,1.17)$ and $J_3 \simeq (3.21,\infty)$. 
Computing ${\bar{J}_i}$ we obtain ${\bar{J}_1}\simeq (-3.53,-0.31)$, ${\bar{J}_2}\simeq (-0.39,1.46)$ 
and ${\bar{J}_3} \simeq (1.76,\infty)$. In this case ${\bar{J}_1}\cap {\bar{J}_2}\neq \emptyset$, therefore, 
according to the algorithm, we have to consider $J_1\simeq (-3.21,1.17)$, obtaining ${\bar{J}_1}\simeq (-3.53,1.46)$. 
Now we have two disjoint intervals and the algorithm is completed. The continuation region is 
\[ 
\CR\simeq (-3.53,1.46)\cup (1.76,\infty).
\]

\begin{figure}
 \begin{center} 
\includegraphics[scale=.8]{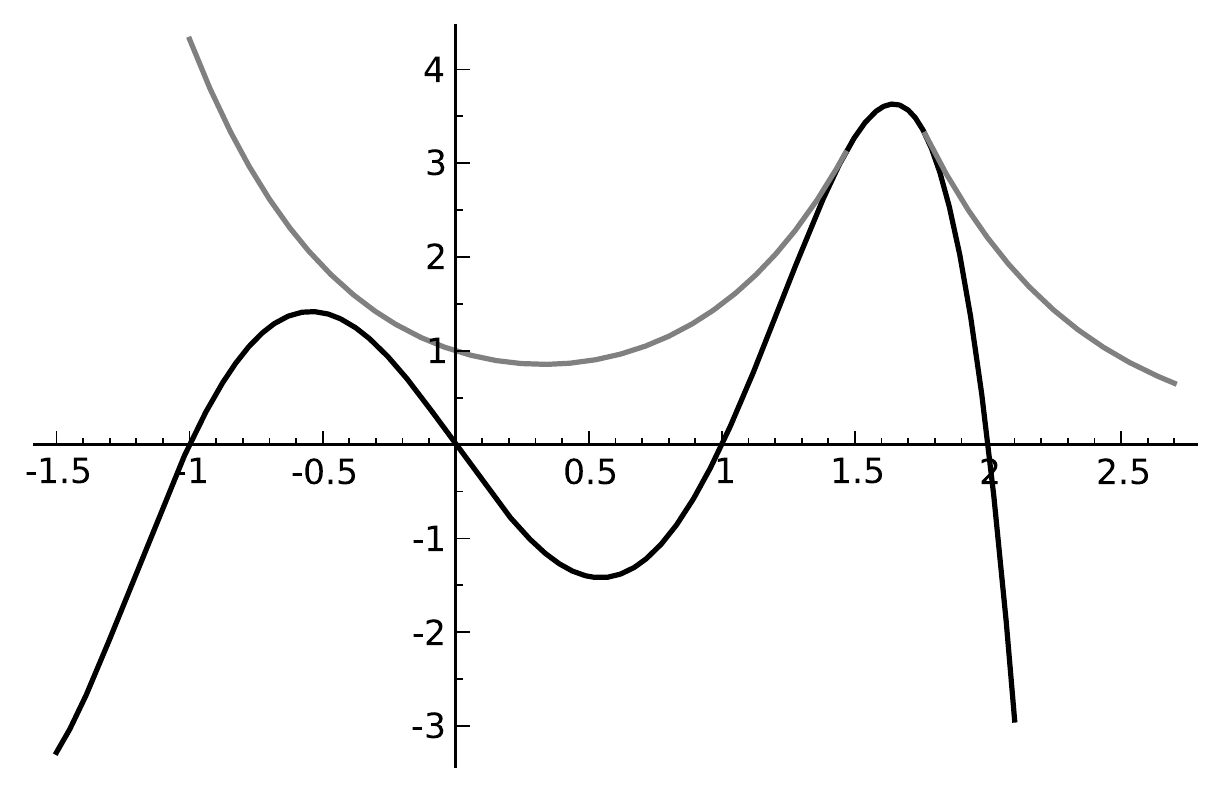}
\caption{\label{fig:poly-alpha1-5} OSP for the standard BM and a 5th. degree polynomial: $g$ (black), $\Va$ (gray, when different from $g$). Parameter $\desc=1.5$. }
 \end{center}
\end{figure}
It can be seen that for $\desc$ small enough the OPS will be left-sided.
\end{example}

\section{More general rewards}

In this section we consider one-dimensional diffusions, as in the rest of the chapter, but we allow less regular reward functions. Our assumption about $\g$ is that there exist a measure $\mug$ such that
\be
\label{eq:mug}
\g(x)=\int_\I \Ga(x,y) \mug(dy),
\ee
where $\Ga(x,y)$ is defined by \eqref{eq:Garepr}. This is motivated by different cases in which the reward $g$ is not regular enough to satisfy the inversion formula \eqref{eq:ginversionCap3}.
%, but considering the derivatives in the wider sense of measures it can be found the measure $\mug$ satisfying \eqref{eq:mug}. 
In these cases, considering the second derivative of the difference of two convex functions as a signed measure, it is possible to obtain a ``generalized'' inversion formula useful for our needs (see \cite{dudley2002real} Problems 11 and 12 of Section 6.3; and see also \cite{protter2005stochastic} p. 218--219).

Just to consider a very simple example, suppose that $X$ is a standard Brownian motion. Consider the function $g\colon \R \to \R$ given by
\bd
g(x):=
\begin{cases}
x,& x<1\\
-x+2,& 1\leq x \leq 2 \\
x-2 & x>2
\end{cases}
\ed
In this case, the differential operator is $Lf=\frac{f''}{2}$ when $f$ is in $\D_L$. The inversion formula \eqref{eq:ginversion} would be
\bd
g(x)=\int_{\R}\Ga(x,y)\al g(y) m(dy)
\ed
where $m(dy)=2 dy$, so the candidate to be $\mug$ is $\al g(y)2 dy$. The derivatives of $g$, in the general sense, would be
\bd
g'(x)=
\begin{cases}
1,& x<1\\
-1,& 1 < x < 2 \\
1 & x>2
\end{cases}
\ed
and the second generalized derivative is the measure $-2\delta_1(dx)+2\delta_2(dx)$. This lead us to consider
\bd
\mug(dy)= \desc g(y)\ind{\R\setminus \{1,2\}}(y) 2 dy +2 \delta_{\{1\}}(dy)-2\delta_{\{2\}}(dy)
\ed
The corresponding computations show that \eqref{eq:mug} holds with the considered measure $\mug$.

\begin{teo} \label{teo:general2}
Consider a one-dimensional diffusion $X$. Consider the function
$g\colon\estados \to \R$ such that  
\be \tag{\ref{eq:mug}} \label{eq:mugTeo}
g(x)=\int_{\I} \Ga(x,y)\mug(dy),
\ee
with $\mug$ a signed measure over $\sigalg$. Assume that $g$ satisfies the conditions for Dynkin's characterization (see \autoref{sec:DynkinChar}).
Suppose that $J_i\colon i=1,\ldots, N$ ($N$ could be $\infty$) are subintervals of $\I$, such that $J_i\cap J_j=\emptyset$ if $i\neq j$ and
\begin{itemize}
  \item $\int_{J_i}\phia(y)\mug(dy)=0$ if there is some $x\in \I$ such that $x<y$ for all $y\in J_i$,
  \item $\int_{J_i}\psia(y)\mug(dy)=0$ if there is some $x\in \I$ such that $x>y$ for all $y\in J_i$.
\end{itemize}
Define $\SR$ by 
\bd
\SR=\I\setminus \cup_{i=1}^N J_i.
\ed
and $\Va\colon \I \to \R$ by
\bd
\Va(x)=\int_{S} \Ga(x,y)\mug(dy). 
\ed
If $\mug(dy)\geq 0$ in $\SR$, and $\Va\geq g$ in $\CR=\cup_{i=1}^N J_i$, then
$\Va$ is the value function associated with the OSP, and $\SR$ is the stopping region.
\end{teo}

\begin{remark} \label{remark:formaVgeneral2} With the same arguments given in \autoref{remark:formaVgeneral} we obtain the alternative representation for $\Va$, given in \eqref{eq:VaMasHumana}:
\bd
\Va(x)=
\begin{cases}
g(x)&\text{for $x \notin \CR$},\\
k_1^i\phia(x)+k_2^i\psia(x) &\text{for $x \in J_i\colon i=1\ldots N$};
\end{cases}
\ed
where, denoting $a_i=\inf J_i$ and $b_i=\sup J_i$
\begin{itemize}
\item $k_1^i=0$ and $k_2^i=g(b_i)/\psia(b_i)$ if there is not $x\in \I$ such that $x<y$ for all $y\in J_i$;
\item $k_1=g(a_i)/\phia(a_i)$ and $k_2=0$ if there is not $x\in \I$ such that $x>y$ for all $y\in J_i$;
\item in the other cases
 $$k_1^i=\frac{g(b_i)\psia(a_i)-g(a_i)\psia(b_i)}{\psia(a_i)\phia(b_i)-\psia(b_i)\phia(a_i)},$$
and 
 $$k_2^i=\frac{g(a_i)\phia(b_i)-g(b_i)\phia(a_i)}{\psia(a_i)\phia(b_i)-\psia(b_i)\phia(a_i)}.$$
 \end{itemize}
\end{remark}

\begin{proof}
The strategy for the proof is to verify that $\Va$ is the minimal $\desc$-excessive function that dominates the reward function $\g$, then, from Dynkin's characterization, follows that $\Va$ is the optimal expected reward.

By the definition of $\Va$, and taking into account that $\mug$ is a non-negative measure in $S$, we conclude that $\Va$ is an $\desc$-excessive function. Applying \autoref{lem:waeqg} with $\Wa:=\Va$, we conclude that $\Va(x)$ and $g(x)$ are equal for $x\in S$, which in addition to the hypothesis $\Va(x)\geq g(x)$ for all $x\in S^c$ allow us to conclude that $\Va$ is a majorant of the reward. So far, we know
\bd
\sup_\tau \Ex{x}{\ea{\tau}g(X_\tau)}\leq \Va(x). 
\ed
From \autoref{lem:waeqg} --in the first equality-- we get
\begin{align*}
\Va(x)&=\Ex{x}{\ea{\hit{\SR}}g(X_{\hit{\SR}})}\\
& \leq \sup_\tau \Ex{x}{\ea{\tau}g(X_\tau)}, 
\end{align*}
that proves the other inequality holds as well. From the previous equation we also conclude that $\SR$ is the stopping region.

\end{proof}

Comparing \autoref{teo:diffusionGeneral} and \autoref{teo:general2}, it should be emphasized that the former gives a characterization of the solution and a method to find it, while the latter is just a verification theorem, which of course, also suggests a method to find the solution. 
However, \autoref{teo:general2} has less restrictive hypothesis and, although we do not include it here, an algorithm to find the continuation region may be developed, at least when the region in which the measure $\mug$ is negative, is a finite union of intervals; in fact, \autoref{algor} would be a particular case of this algorithm when considering $\mug(dy)=\al g(y)m(dy)$. 

\subsection{Brownian motion with drift and $g(x)=|x|$}
\label{ex:absx2}
As in \autoref{ex:absx1} we consider $X$ to be Brownian motion with drift $\mu$ and the reward function $g(x)=|x|$. This process has a Green function with respect to the reference measure $m(dx)=2e^{2\mu x} dx$ given by
\bd
\Ga(x,y)=
\begin{cases}
\wa^{-1}e^{-(\gamma+\mu)x}e^{(\gamma-\mu)y},\qquad y\geq x,\\
\wa^{-1}e^{-(\gamma+\mu)y}e^{(\gamma-\mu)x},\qquad y\leq x.
\end{cases}
\ed
where $\gamma=\sqrt{2\desc+\mu^2}$ and the Wronskian is $\wa=2\gamma$. The differential operator is $Lf(x)=\mu f'(x)+f''(x)/2$

For functions $f$ in the domain of the extended infinitesimal generator we would have 
\bd
f(x)=\int_{\estados}\Ga(x,y)(-\Aa f(y)) m(dy)
\ed
with $-\Aa g(x)=\desc f(x)-Lf(x)$. Suppose we can apply this formula to $f(x)=|x|$, interpreting the derivatives in the extended sense of measures, we would have
\bd
|x|=\int_{\R^*}\Ga(x,y) \mu(dy)
\ed
with
\bd
\mu(dy)= (-\desc y + \mu)2 e^{2\mu y}\ind{\{y<0\}}(y) dy +2 \delta_{\{0\}}(dy)+(\desc y - \mu)2 e^{2\mu y} \ind{\{y>0\}}(y) dy.
\ed
It can be checked that the previous formula actually holds. We can apply \autoref{teo:general2}. Assuming that the set $S$ is of the form $S=(-\infty,x_\ell)\cup(x_r,\infty)$ for some $x_\ell < 0 < x_r$, we need to find $x_\ell,x_r$ such that
\bd
\begin{cases}
\int_S \Ga(x_\ell,y)\mu(dy)=-x_\ell,\\
\int_S \Ga(x_\ell,y)\mu(dy)=x_r,
\end{cases}
\ed
or what is the same 
\bd
\begin{cases}
\int_{(x_\ell,x_r)} \Ga(x_\ell,y)\mu(dy)=0,\\
\int_{(x_\ell,x_r)} \Ga(x_r,y)\mu(dy)=0,
\end{cases}
\ed
which is also equivalent to
\be
\label{eq:absx}
\begin{cases}
\int_{(x_\ell,x_r)} \phia(y)\mu(dy)=0,\\
\int_{(x_\ell,x_r)} \psia(y)\mu(dy)=0.
\end{cases}
\ee
It can be seen that these equations are equivalent with the ones found in \autoref{ex:absx1}. When solving particular cases (with concrete parameter values) it is easy to verify that the region found is indeed the optimal stopping region.

We follow with some numerical examples. To do the numerical computations we have used the implementation in R, presented in \autoref{sec:implementation} with minor changes.

\begin{example}
Consider the discount $\desc=1$ and the drift $\mu=0$, solving numerically the system of equations \eqref{eq:absx} we find
$x_\ell\simeq -0.69264$, $x_r\simeq 0.69264$, a graphic of the solution is shown in \autoref{fig:absxmu0}.
\begin{figure}
\begin{center}
\includegraphics[scale=.8]{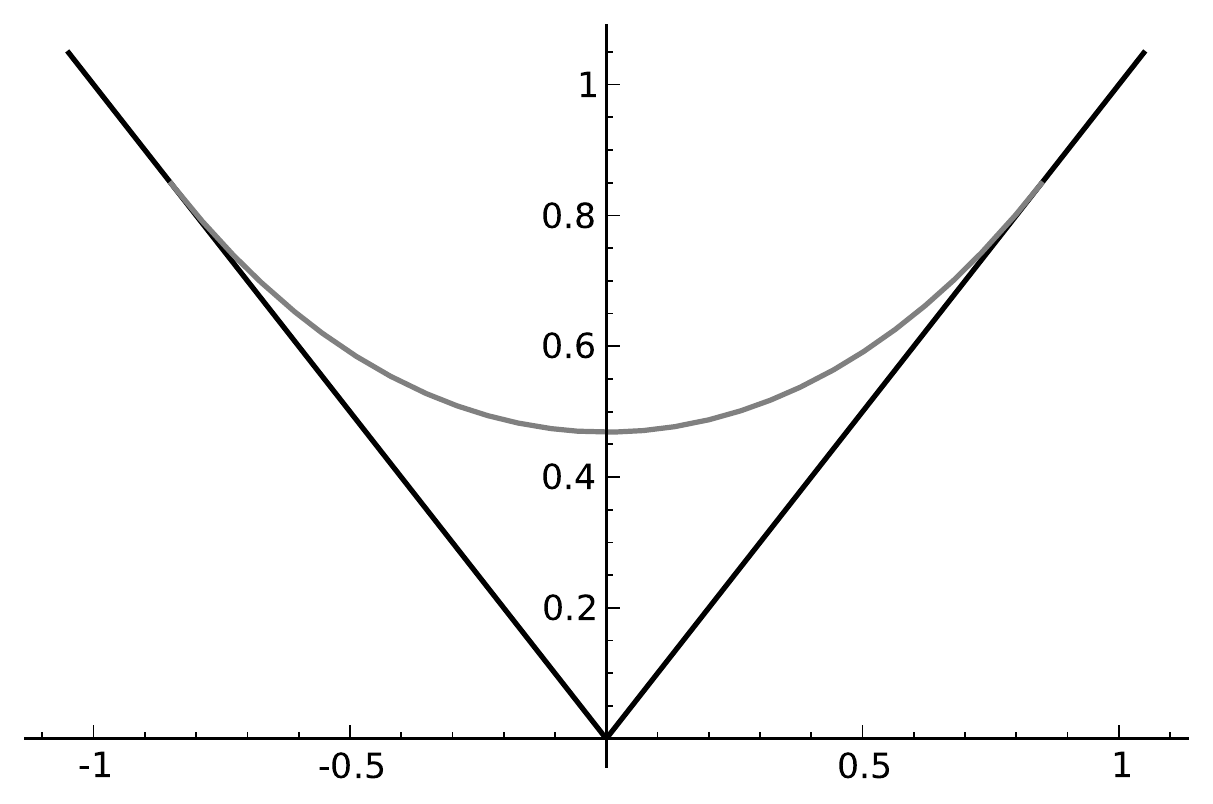}
\caption{\label{fig:absxmu0} OSP for the standard BM and $g(x)=|x|$: $g$ (black), $V_1$ (gray, when different from $g$).}
\end{center}
\end{figure}
\end{example}

\begin{example}
Now consider the same discount $\desc=1$ but a positive drift $\mu=1$, solving numerically the system of equations \eqref{eq:absx} we find 
$x_\ell\simeq -0.737$, $x_r\simeq 1.373$, a graphic of the solution is shown in \autoref{fig:absxmu1}.
\begin{figure}
\begin{center}
\includegraphics[scale=.8]{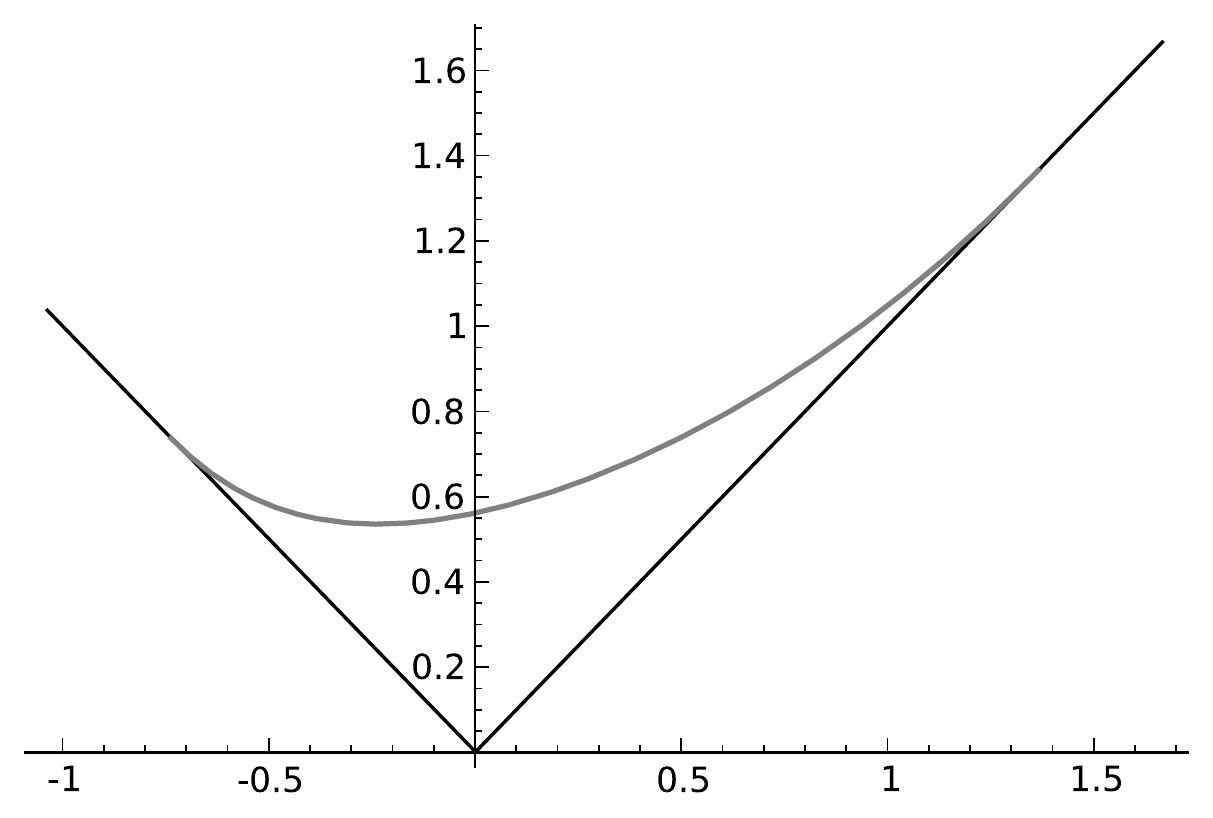}
\caption{\label{fig:absxmu1}OSP for the BM with drift $\mu=1$ and $g(x)=|x|$: $g$ (black), $V_1$ (gray, when different from $g$).}
\end{center}
\end{figure}
\end{example}

\begin{example}
Considering $\desc=1$ and a negative drift $\mu=3$ we find $x_\ell\simeq -3.158$ and $x_r\simeq 1.037$. \autoref{fig:absxmu-3} shows the solution.

\begin{figure}
\begin{center}
\includegraphics[scale=.8]{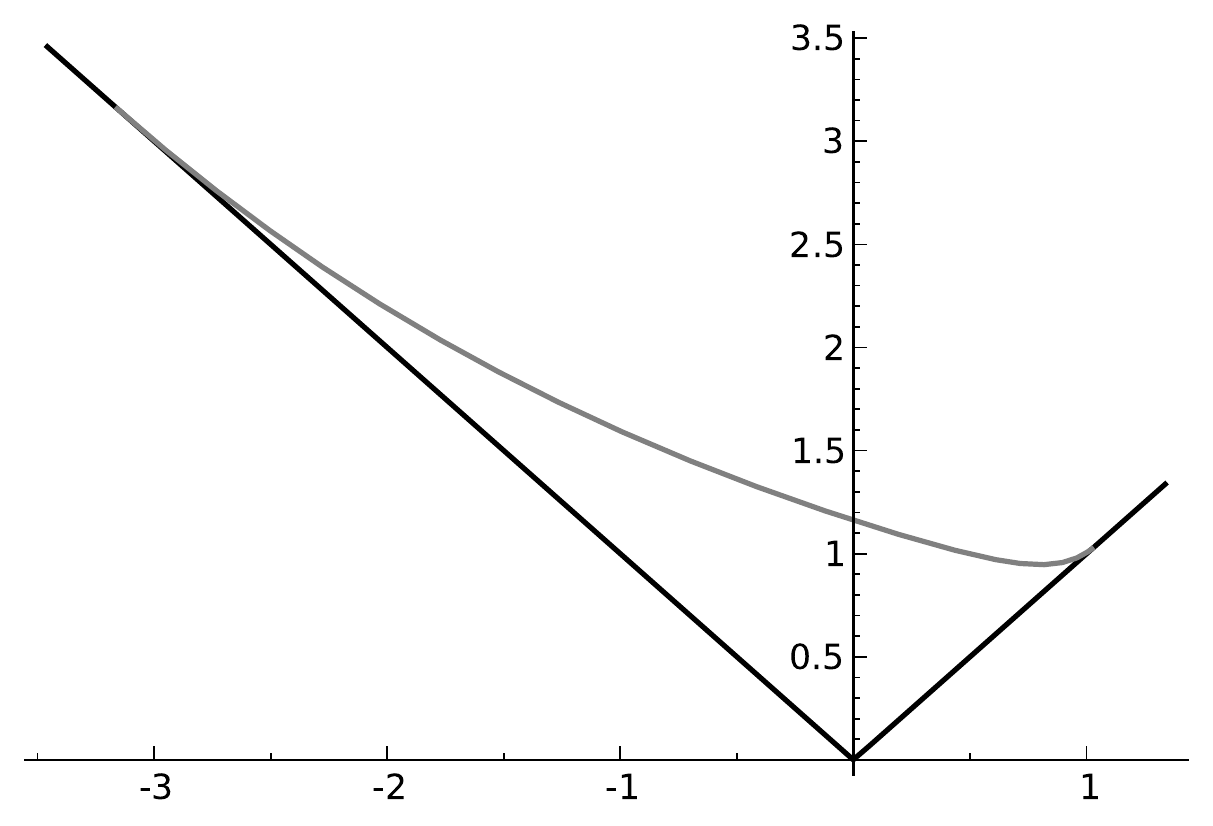}
\includegraphics[scale=.6]{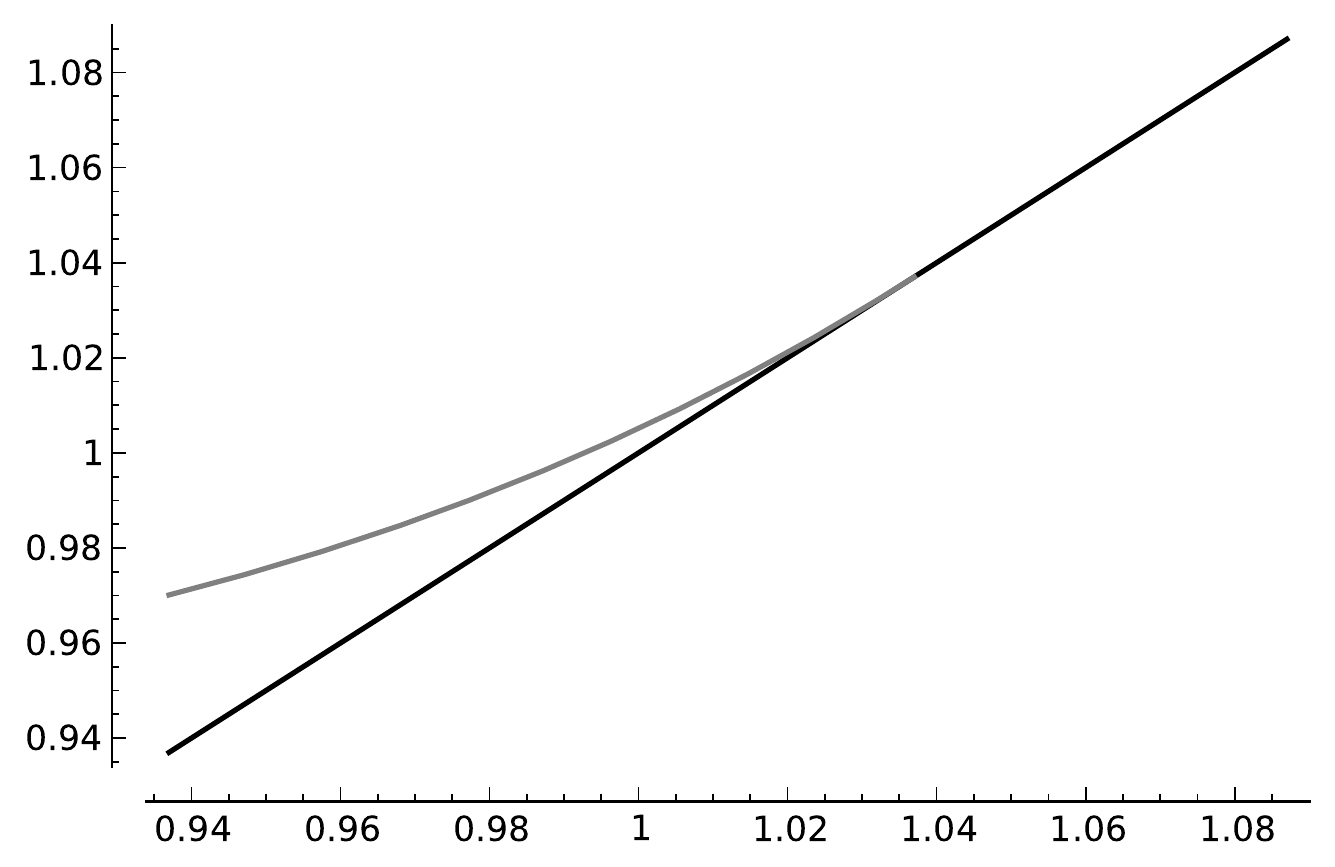}
\caption{\label{fig:absxmu-3}OSP for the BM with drift $\mu=-3$ and $g(x)=|x|$: $g$ (black), $V_1$ (gray, when different from $g$). The second graphic is a zoom to appreciate the smooth fitting.}
\end{center}
\end{figure}
\end{example}

\subsection{Example: Other non-differentiable reward}
Consider the OSP with reward $g\colon \R \to \R$ given by 
\bd
g(x)=
\begin{cases}
x,& x<1,\\
-x+2,& 1\leq x \leq 2, \\
x-2 & x>2.
\end{cases}
\ed
This is the function already presented in the introduction of this section and it satisfies \eqref{eq:mugTeo} with $\mug$ given by
\bd
\mug(dy)=
2 \desc g(y)  \ind{\R\setminus \{1,2\}}(y) dy +2 \delta_{\{1\}}(dy)-2\delta_{\{2\}}(dy).
\ed
Consider the discount factor $\desc=1$. The measure $\mug$ is negative in $(-\infty,0)$ and in $\{2\}$. Computing exactly in the first case, and by numerical approximation in the second (by following a variant of \autoref{algor}), we manage to find two disjoint intervals $J_1\simeq(-\infty,1/\sqrt2)$ and $J_2\simeq(1.15,2.85)$ that satisfy the conditions of \autoref{teo:general2}. 
For $\Va$, we have  the expression given in \autoref{remark:formaVgeneral2}, which considering $\psia(x)=e^{\sqrt{2\desc}x}$ and $\phia(x)=e^{-\sqrt{2\desc}x}$ in the particular case $\desc=1$, renders\footnote{We approximate the roots.}
\bd
V_1(x)=
\begin{cases}
k_2^1 e^{\sqrt2 x},& x<\frac1{\sqrt2},\\
x,& \frac1{\sqrt2} \leq x \leq 1,\\
-x+2,& 1<x\leq 1.15,\\
k_1^2 e^{-\sqrt2 x} + k_2^2 e^{\sqrt2 x},& 1.15<x<2.85,\\
x-2,&x\geq 2.85;
\end{cases}
\ed 
with $k_2^1= \frac1{e\sqrt2}\simeq 0.26$, $k_1^2\simeq 3.96$ and $k_2^2\simeq 0.013$.
\begin{figure}
\begin{center}
\includegraphics[scale=.9]{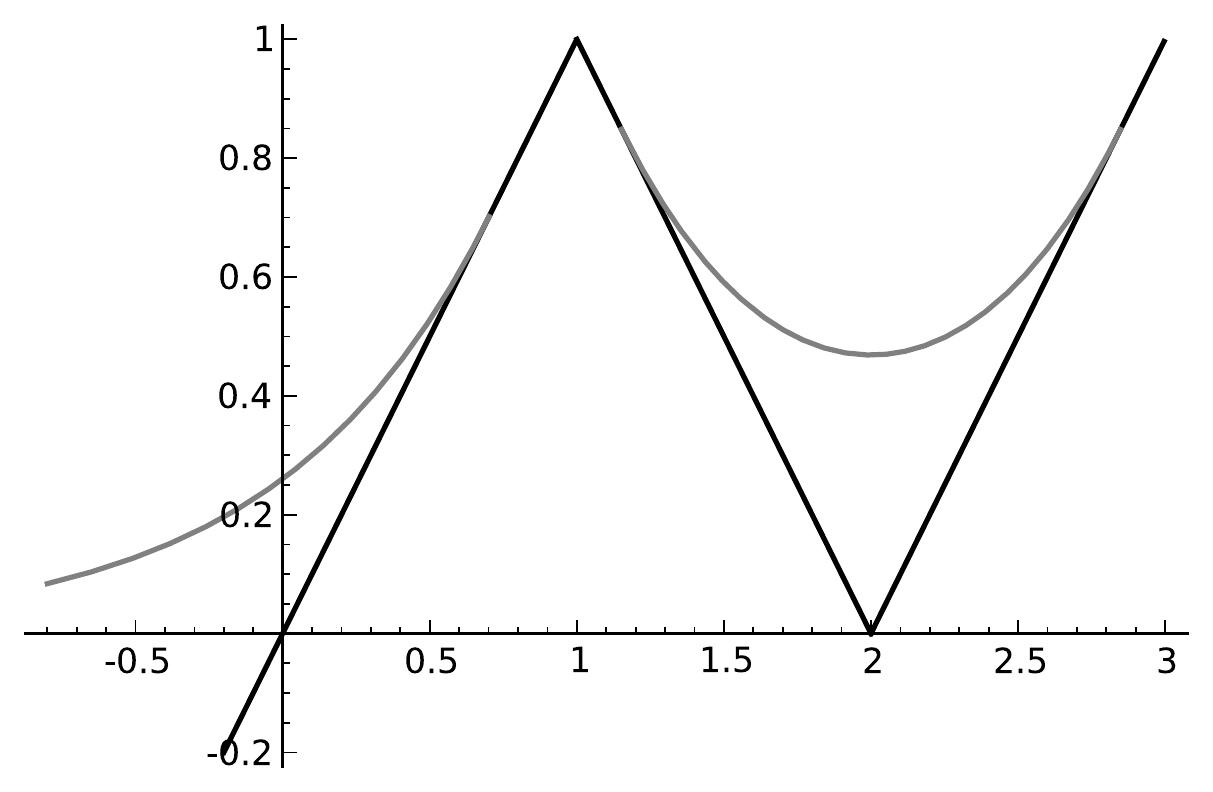}
\end{center}
\caption{\label{fig:mug} OSP for the standard BM and irregular reward: $g$ (black), $V_1$ (gray, when different from $g$).}
\end{figure}
In \autoref{fig:mug} we show the reward function $g$ and the value function $V_1$.
\chapter{Optimal stopping for multidimensional continuous Markov processes}
\label{chap:contMarkov}

\section{Introduction}
In this chapter we present results on optimal stopping for Markov processes with continuous sample paths taking values in general state spaces. We consider a standard Markov process, as 
%outlined, laid out
presented in \autoref{def:standardMarkov}, with the additional assumption of having continuous sample paths. We remember that the state space is an abstract topological semi-compact set $\estados$ equipped with the Borel $\sigma$-algebra $\sigalg$. 
The optimal stopping problem being considered, as in the previous chapters, is to find the stopping time $\tau^*$ and the value function $\Va$ satisfying
\be
\label{eq:ospChap4}
\Va(x)=\Ex{x}{\ea{\tau^*} \g(X_{\tau^*})} = \sup_{\tau}\Ex{x}{\ea{\tau} \g(X_{\tau})},
\ee
where the supremum is taken over all stopping times.

By the discounted version of the Dynkin's characterization, stated in \autoref{sec:DynkinChar}, the value function $\Va$ is the smallest $\desc$-excessive function dominating the reward function $g$, while the optimal stopping time is the hitting time of the stopping region $\SR$ given by:
\bd
\SR:=\{x\in \estados\colon \Va(x)=g(x)\}.
\ed
Combining this with the fact that 
\bd
x\mapsto \int_{\estados}f(y)\Ga(x,dy)
\ed
is an $\desc$-excessive function if $f\colon \estados\to \R$ is a non-negative function, we manage to prove a verification theorem to test the solution of the OSP when the reward function has a representation of the form:
\bd
g(x)=\int_{\estados}f(y)\Ga(x,dy).
\ed
Results in \autoref{chap:markov} suggest that the class of functions having this representation is fairly large.

Optimal stopping problems in which the process is a multidimensional diffusion \citep[see][]{stroock1979multidimensional} are a relevant particular case of the processes considered in this chapter. 
The optimal stopping problem of a multidimensional process is, in general, a hard issue, being the exception the problems for which a closed solution is known --usually problems in which the process and the reward function have certain type structure that allow to reduce it to a one-dimensional problem.--
One of the first work in this direction was \cite{margrabe1978value}, in which the author provides an explicit formula for the price of the right of changing an asset by another at the expiration date of the contract or at a time that can be choose, all this in the context of a Black-Scholes market. 
In the same line \cite{gerber2006martingale} solve the problem of pricing American  options on two stocks, where the prices are driven by geometric Brownian motions. 
A generalization of this result was provided in \cite{fajardo2006pricing}, considering Lévy driven stock prices. In some sense, the work by \cite{dubinsSheppSh} on optimal stopping for the maximum of a Bessel process, can be seen as a problem regarding multidimensional process, as the Bessel process itself has its origin in the multidimensional Brownian motion. In the pricing of Russian options, \cite{shepp93} deal with a two-dimensional diffusion; later, in \cite{shepp94}, the authors find the way of reducing the problem to a one-dimensional problem, obtaining an easier solution.
An interesting object to study, regarding optimal stopping of multidimensional processes, are the contracts on indexes that are constructed by linear combinations of different stock prices; similar to this is the problem considered in \cite{oksHu}.
The recent paper on optimal stopping by \cite{christensen2011harmonic} includes the treatment of multidimensional diffusions presenting a method based on harmonic functions. Another work that considers multidimensional processes is the Firth's PhD thesis \citep{firth2005thesis}, in which the problem of pricing multi-asset American options is studied.

\section{Main results}
%\subsection{Optimal stopping for general continuous paths processes}
We denote by $\hit{S}$ the hitting time of the set $S$, defined by
\bd \hit{S}:=\inf\{t\geq 0\colon X_t\in S\}. \ed
We start by stating our main theorem.

\begin{teo} \label{teo:general}
Consider a standard Markov process $X$ with continuous sample paths. Consider a reward function $g\colon\estados \to \R$ such that  there exists $f\colon \estados \to  \R$ that satisfies
\bd
%\label{eq:inversionRn}
g(x)=\int_{\estados} f(y)\Ga(x,dy)\quad (x \in \estados).
\ed
Assume as well that function $g$ satisfies the conditions for Dynkin's characterization (see \autoref{sec:DynkinChar}).
Consider $S\in \sigalg$ such that: $f(x)\geq 0$ for all $x \in S$; and $\Va\colon\estados \to \R$ defined by
\begin{equation*}
\Va(x):=\int_{S} f(y)\Ga(x,dy),
\end{equation*}
satisfies
\begin{itemize}
\item $\Va(x)\geq g(x)$ for all $x\in \estados \setminus S$, and
\item $\Va(x)=g(x)$ for all $x \in \partial S.$
\end{itemize}
Then, $S$ is the stopping region of the OSP, being $\hit{S}$
the optimal stopping time and $\Va$ the value function.
\end{teo}

\begin{remark}
This theorem is a generalization, to topological (including multidimensional) spaces, of \autoref{teo:verif1} in the case in which the inversion formula \eqref{eq:ginversion} holds.
\end{remark}
In the practical use of the previous theorem, one can think the function $f$ as $\desc g-Ag$ with $Ag$ as the infinitesimal generator of $g$.
Before proving the theorem, we state and prove a previous lemma, whose result we find interesting in itself.
\begin{lem}
\label{multimainlemma}
Consider a standard Markov process $X$ with continuous sample paths. Let $\Wa\colon\estados \to \R$ be defined by
\begin{equation}
\label{eq:defW}
\Wa(x):=\int_{S} f(y) \Ga(x,dy).
\end{equation}
and $g\colon\estados \to \R$ be defined by
\bd
g(x):=\int_{\estados} f(y) \Ga(x,dy),
\ed
where $f\colon \estados \to \R$ is some $\sigalg$-measurable function. If $g(x)=\Wa(x)$ for all $x\in \partial S$, then, $g(x)=\Wa(x)$ for all $x\in S$. Furthermore, $\Wa$ satisfies
\be \label{eq:WA}
\Wa(x)=\Ex{x}{ e^{-\desc \hit{S}}g(X_{\hit{S}})}. 
\ee
\end{lem}
\begin{proof}
First observe that from the fact that $g(x)=\Wa(x)$ for all $x\in \partial S$ it follows that
\begin{equation}
\label{eq:auxborde}
\int_{S^c} f(y)\Ga(x,dy)=0 \quad (x \in \partial{S}),
\end{equation}
where $S^c:=\estados\setminus S$. We can now move on to prove
 that $g(x)=\Wa(x)$ for $x\in S$. To confirm this, consider the following equalities:
\begin{align*}
g(x)&= \int_{\estados} f(y)\Ga(x,dy)\\
&=\int_{S} f(y) \Ga(x,dy)+\int_{S^c} f(y)\Ga(x,dy)\\
&= \Wa(x) +\int_{S^c} f(y) \Ga(x,dy),
\end{align*}
and note that we must prove that the second term on the right-hand side vanishes. Given that $\int_{S^c} f(y) \Ga(x,dy)$ is $F_{S^c}$, in the notation of \autoref{lem:generalMin}, and by the application of this result, we are able to conclude that for $x\in S$
\bd
 \int_{S^c} f(y) \Ga(x,dy) = \Ex{x}{e^{-\desc \hit{S^c}} \int_{S^c} f(y)  \Ga(X_{\hit{S^c}},dy)}.
\ed
Bearing \eqref{eq:auxborde} in mind, the integral inside the expected value above on the right-hand side vanishes, as $X_{\hit{S^c}}$ is in $\partial{S}$. 
Therefore, the formula on the right-hand side vanishes, thus completing the proof. 

The validity of \eqref{eq:WA} is a direct consequence of \autoref{lem:generalMin}, considering that $\Wa=F_S$ and 
taking into account that $g(x)=\Wa(x)$ for $x\in \partial S$.

\end{proof}

\begin{proof}[Proof of \autoref{teo:general}]
%as in the proof for the one-dimensional case
As in the proof for the one-dimensional case, this one entails proving that $\Va$ is the minimal $\desc$-excessive function that dominates the reward function $g$. By Dynkin's characterization, this implies that $\Va$ is the optimal expected reward.

By the definition of $\Va$, and taking into account that $f$ is a non-negative function in $S$, we deduce that $\Va$ is an $\desc$-excessive function (see \autoref{sec:excessive}). 
Applying \autoref{multimainlemma} with $\Wa:=\Va$, we get  $\Va(x)=g(x)$ for $x\in S$, which in addition to the hypothesis $\Va(x)\geq g(x)$ for all $x\in S^c$, yield that $\Va$ is a majorant of the reward. 
So far, we have established that
\bd
\sup_\tau \Ex{x}{ e^{-\desc \tau}g(X_\tau)}\leq \Va(x). 
\ed
By \autoref{multimainlemma} we get 
\begin{align*}
\Va(x)&=\Ex{x}{e^{-\desc \hit{S}}g(X_{\hit{S}})}\\
& \leq \sup_\tau \Ex{x}{e^{-\desc \tau}g(X_\tau)}, 
\end{align*}
then, we conclude that the desired equality holds.

\end{proof}

The following corollary of \autoref{teo:general} provides a practical way of using this result.

\begin{cor}
\label{mainmulti}
Consider a standard Markov process $X$ with continuous sample paths. 
Consider a reward function $g\colon\estados \to \R$ that belongs to the domain $\Da$
of the extended infinitesimal generator associated with the $\desc$-killed process and satisfying
\begin{itemize}
\item $\lim_{t\to \infty}\Ex{x}{\ea{t}g(X_t)}=0$ and
\item $\Ex{x}{\int_0^{\infty} \ea{s}|\Aa g(X_s)| ds}<\infty$.
\end{itemize} 
Assume as well that function $g$ satisfies the conditions for Dynkin's characterization (see \autoref{sec:DynkinChar}).
Suppose that $S\in \sigalg$ is a set satisfying $\Aa g(x)\leq 0$ for all $x \in S$
and define $\Va\colon\estados \to \R$ by:
\begin{equation*}
\Va(x):=\int_{S} (-\Aa g(y))\Ga(x,dy). 
\end{equation*}
If
\begin{itemize}
\item $\Va(x)\geq g(x)$ for all $x\in \estados \setminus S;$
\item $\Va(x)=g(x)$ for all $x \in \partial S;$
\end{itemize}
then $S$ is the stopping region of the OSP, being $\hit{S}$ the optimal stopping time and $\Va$ the value function.
\end{cor}

\begin{proof}
Since $g$ is in the domain of the extended infinitesimal generator associated with the killed process and the additional hypotheses we made about $g$, we know that \eqref{eq:invExtended} holds, that is 
\bd
g(x)=\int_{\estados}-\Aa g(y) \Ga(x,dy)\qquad (x\in \estados).
\ed
Therefore, we are in conditions to apply the previous theorem with $f(y):=-\Aa g(y)$ to complete the proof.

\end{proof}

%\section{Examples}
\section{3-dimensional Brownian motion}
Consider a three dimensional Brownian motion $X$ starting from $\ve$ (i.e. $X_t=\ve +B_t$ with $\{B_t\}$ a three dimensional Brownian motion). The differential operator of $X$ for $g\in  \CC^2(\R^3)$ is given by
\bd
Lg(\ve)=\sum_{i=1}^{3} g_{i}(\ve) + \frac{1}{2} \sum_{i=1}^3 g_{ii}(\ve) 
\ed
where $g_i$ and $g_{ii}$ denote the first order and second order partial derivative with respect to the $i$-st coordinate. The Green measure of $X$ is given by 
\bd
\Ga(\ve,d \we)=\frac{c}{\|\ve-\we\|}e^{-\sqrt{2\desc}\|\ve-\we\|}d\we.
\ed
for some constant $c$ \citep[see][p.306]{pinsky}.

\begin{example}
Consider the OSP, consisting in finding the stopping time $\tau^*$ such that, for every $\ve \in \R^3$
\bd
\Ex{\ve}{\ea{\tau^*}g(X_{\tau^*})}=\sup_{\tau}\Ex{\ve}{\ea{\tau}g(X_{\tau})},
\ed
with $g\colon\R^3 \to \R$ defined by
\bd
g(\ve)=\|\ve\|^2=\vx^2+\vy^2+\vz^2,
\ed
where $\ve$ is the vector $(\vx,\vy,\vz)$ in $\R^3$. For our particular reward function we have $Lg(\ve)=3$. It can be seen that $g$ satisfies the hypothesis of \autoref{mainmulti} and $-\Aa g(\ve)=\al g(\ve)$. 
We need to find the stopping region $\SR$, or equivalently the continuation region $\CR=\R^3\setminus \SR$. 
Observe that, in order to verify the hypothesis of $\Va(x)=g(x)$ in $\partial \SR$, we need $\CR$ to fulfill
\be \label{eq:sphere}
\int_\CR \al g(\we) \Ga(\ve,d\we)=0
\ee
for all $\ve \in \partial \CR$. 
The symmetry of the problem suggests us to consider $\CR$ as a ball centred in (0,0,0). We need to find out the radius. This lead us to look for $r>0$ such that, considering $\CR=\{\ve\colon\|\ve\|<r\}$ the equation \eqref{eq:sphere} holds for $\ve\colon \|\ve\|=r$.
To simplify the computations we consider $\ve=(0,0,r)$, but it is easy to see that this selection does not change the problem.
Now we solve the equation
\bd
\int_{\{\|\we\|<r\}} \frac{1}{\|\ve-\we\|}e^{- \sqrt{2\desc} \|\ve-\we\| } (\desc \|\we\|^2-3) d\we=0.
\ed
Considering $\se=\we-\ve$ the integral becomes
\bd
\int_{\{ \se \in \CR'\}} \frac{1}{\|\se\|}e^{- \sqrt{2\desc} \|\se\| } (\desc \|\se\|^2+2\desc r\sz+\desc r^2-3) d\se.
\ed
where $\CR'$ is the interior of a sphere with center in $(0,0,-r)$ and radius $r$. 
We consider spherical coordinates $(\rho,\theta,\phi)$, where $\rho=\|s\|$ is the radius, 
$\theta=\arccos(\sz/\rho)$ is the inclination, and $\phi=\arctan(\sy/\sx)$ is the azimuth. 
We have $\sz=\rho \cos(\theta)$, and $d\se=\rho^2 \sin (\theta) d\rho\, d\theta\, d\phi$; 
we obtain that the previous integral is
\bd
\int_0^{2\pi} d\phi \int_{\pi/2}^{\pi} d\theta \int_0^{-2\,r\cos(\theta)} \frac{1}{\rho}e^{- \sqrt{2\desc} \rho } (\desc \rho^2+ 2 \desc r \rho \cos(\theta)+\desc r^2-3)\rho^2 \sin(\theta) d\rho.
\ed
and doing the computations we conclude that this is equal to
\bd
{4 \pi r^3}\left(3-\sqrt{2\desc}\, r - e^{-2 \sqrt{2\desc}\,r}(\sqrt{2a}\,r+3) \right).
\ed
In order to find the positive solution of the equation we need to solve
\bd
3-\sqrt{2\desc}\, r - e^{-2 \sqrt{2\desc}\,r}(\sqrt{2a}\,r+3)=0;
\ed
calling $z=\sqrt{2\desc}\,r$ and doing computations we obtain the equivalent equation
\bd
(1-e^{-2 z})z=(1+e^{-2 z})3;
\ed
multiplying by $e^z$ the previous equation it can be easily concluded that $z$ is a solution of
\bd
\tanh(z)=\frac{z}{3}.
\ed
We conclude that the continuation region associated with the optimal stopping is the sphere centered in (0,0,0)
with radius $r=z/\sqrt{2\desc}$. This problem is equivalent to the one solved in \ref{ex:bessel}, since the process 
$\|X\|$ is a 3-dimensional Bessel process; in fact, the obtained solutions  agree.
\end{example}

\section{The ideas for a converse result} 
\label{sec:converse}

The main result of this chapter, \autoref{teo:general}, gives a number of sufficient conditions on a certain function $\Va$ in order to be the value function associated with the OSP. We would like to find out when those conditions are actually necessary.

For example, in the context of one-dimensional diffusions, we have \autoref{teo:diffusionGeneral}, which can also be written as follows:

\noindent \textbf{Theorem 3.8.} (Alternative statement.)
\emph{Consider a one-dimensional diffusion $X$ satisfying the corresponding assumptions. Assume that the reward function $g\colon \I \to \R$ satisfy
\bd
g(x)=\int_{\I} \Ga(x,y)\al g(y) m(dy).
\ed
Then:
$\SR$ is the stopping region and $\Va$ is the value function for the OSP, if and only if:
\begin{itemize}
\item $\Va(x)=\int_{\SR}\Ga(x,y)\al g(y) m(dy)$;
\item $\al g(x) \geq 0$ for  $x\in \SR$;
\item $\Va(x)=g(x)$ for  $x\in  \partial  \SR$;
\item $\Va(x) \geq g(x)$ for $x\in  \CR,\ (\CR:= \I\setminus \SR)$.
\end{itemize}
}

This alternative statement shows that for one-dimensional diffusions a converse result holds.

We say that the process $X$ \emph{has a Green function} if there exist a measure $m$ over $(\estados,\sigalg)$ and a function $\Ga(x,y)$, jointly measurable in $x$ and $y$, such that 
$$\Ga(x,H)=\int_H \Ga(x,y)m(dy).$$
The measure $m$ is called a \emph{reference measure} and $\Ga(x,y)$ is the \emph{Green function}. It is not an easy problem to determine whether a given process has Green function. Nevertheless, this happens in the most important cases, particularly in all the examples considered in this work. \cite[see][]{blumenthal-ge,kw,dynkin:1969}. 
%we have changed (with Eddie) it happens by this happens

The definition of infinitesimal generator $L$ given for one-dimensional diffusion also makes sense for standard Markov processes \cite[see, for instance][]{revuz}. In fact, for functions in the domain $\D_L$ of the infinitesimal generator (and also for other functions) the following inversion formula holds:
\begin{equation}
\label{eq:firstInvFormula}
\int_\estados\Ga(x,y) \al g(y)m(dy)=g(y).
\end{equation}

We say that the operator $L$ is \emph{local} if $Lf(x)=Lg(x)$ provided that $f=g$ in a neighbourhood of $x$. This is the case when $L$ is a differential operator \cite[see][]{oksendal}. In this situation we can extend the meaning of $L$, as we did for one-dimensional diffusions.

Riesz decomposition states that any $\alpha$-excessive function $\Va$ can be represented by 
\be \label{eq:converseVari} 
\Va(x)=\int \Ga(x,y) \sigma(dy)+h(x),
\ee
where $\sigma$ is a positive Radon measure and $h$ is an $\desc$-harmonic function, which under mild regularity conditions are unique. \cite[see][Theorem 2, and Proposition 13.1]{kw}.

It also can be seen, under the made assumptions, that if $u$ is given by
\be \label{eq:alu}
u(x):=\int f(y) \Ga(x,dy) +h(x),
\ee
and $f$ is continuous at $x$, then $\al u(x)=f(x)$.

Consider then a standard Markov process $X$ such that:
\begin{itemize}
\item has a Green function and the reference measure $m(dx)$ has no atoms;
\item the infinitesimal operator $L$ is local;
\item Riesz representation holds.
\end{itemize}
Consider a reward function $g$ such that:
\begin{itemize}
\item satisfies \eqref{eq:firstInvFormula};
\item $\al g$ is a continuous function; 
\item  satisfies the hypotheses for Dynkin's characterization.
\end{itemize}
Under the made assumptions, if the solution to the OSP is given by the stopping region $\SR$ and the value function $\Va$ we know that
\begin{itemize}
\item $\Va=g$ in $\SR$ (and also in $\partial \SR$);
\item $\Va\geq g$ in $\CR$;
\item $\al g$ is non-negative in $\SR$.
\end{itemize}

Therefore, if we could prove that 
\bd
\Va(x)=\int_{\SR}\al g(y) \Ga(x,y)m(dy)
\ed
we would have the desired converse result. We know that $\Va$ is an $\desc$-excessive function, then \eqref{eq:converseVari} holds for some measure $\sigma$. 
If $\sigma(dy)$ is absolutely continuous with respect to $m(dy)$ we could express
\bd
\Va(x)=\int_{\estados} \Ga(x,y) f_\sigma (y) m(dy)+h(x),
\ed
where $f_\sigma$ is the Radon-Nikodim derivative of $\sigma$ with respect to $m$.
We also know that $\Va$ is $\desc$-harmonic in the continuation region, and then $\sigma$ does not charge that set \cite[see][Theorem 12.1]{dynkin:1969}. Therefore, we may choose $f_\sigma$ to be 0 in the continuation region. Finally using the fact that $L$ is a local operator and $g=\Va$ in $\SR$ we obtain (assuming that $\SR$ is an open set without loss of generality)
\bd
\al \Va(x)= \al g(x) \quad (x\in \SR).
\ed
On the other hand, if we may assume $f_\sigma$ to be continuous in $\SR$, we obtain, by \eqref{eq:alu}, that
\bd
\al \Va(x) = f_\sigma(x) \quad (x\in \SR);
\ed
concluding that $f_\sigma=\al g$ in $\SR$. We still need to prove that $h$ in the representation of $\Va$ vanishes.

Although we do not have a complete proof of the fact that $h$ vanishes, we think it may be not difficult to prove it --perhaps with some additional hypothesis--. The assumption we find hard to justify, but we conjecture its validity, is the existence of $f_\sigma$.

In the one-dimensional case the representation of the Green function as the product of the fundamental solutions of $\al g(x)=0$ provides an alternative way to obtain this same result.
\chapter{Optimal stopping for strong Markov processes with one-sided jumps}
\label{chap:jumps}

\section{Introduction}
Through this chapter we consider standard Markov processes with only positive jumps or with only negative jumps. We denote by $\I$ its state space, which we assume to be an interval of $\R$ with left endpoint $\ell$ and right endpoint $r$. We use the notation $\Ig{a}$ to refer to the set $\I\cap \{x\colon x>a\}$, and $\Ige{a}$, $\Il{a}$, $\Ile{a}$ are used in the same sense.

In the case of positive jumps we study optimal stopping problems in which the stopping region is of the form $\Ige{x^*}$ (right-sided) and we develop the theory in detail. 
In the case of negative jumps we consider optimal stopping problems in which the 
stopping region is of the form $\Ile{x^*}$ but we only state the main theorem, 
the proofs being analogous to the previous case. 

The most studied subclass of Markov processes with jumps is the class of Lévy processes, which are processes with independent, stationary increments. Comprehensive treatment of Lévy processes can be found in the books by \cite{bertoin,applebaum,kyprianou2006introductory,sato1999levy}.
In recent years, several works on optimal stopping for jump processes have been developed, mostly regarding application to option pricing. In this direction it can be remarked the books by \cite{cont-tan,boyarchenko2002non,kyprianou2005exotic}, and the Surya's PhD thesis \citep{surya2007thesis}. The book \cite{boy-lev} also includes optimal stopping for Lévy processes. 

Consider a standard Markov process $X$ with state space $\I$, an interval of $\R$. Denoting by $\Delta X_t$ the difference between $X_t$ and $X_{t^-}$ ($\Delta X_t=X_t-\lim_{t\to s^-}X_s$),
we consider processes satisfying $\Delta X_t\geq 0$ for all $t$, also called spectrally-positive or processes satisfying $\Delta X_t\leq 0$ for all $t$, called spectrally-negative. 
A spectrally-positive process satisfies the following condition: if at some time $t$ the process is in the state $x$ and in a posterior time $s$ the process is in the state $y$, with $y$ less than $x$, 
then for every $z$ such that $y<z<x$ there exists an intermediate time such that the process is in the state $z$; in other words the process do not have negative jumps. 
%We also assume that $X$ has Green function $\Ga(x,y)$ and denote by $m$ its reference measure.

The optimal stopping problem we consider in this chapter is the same already considered in the previous ones: to find the stopping time $\tau^*$ and the value function $\Va$ satisfying
\be
\label{eq:ospChap5}
\Va(x)=\Ex{x}{\ea{\tau^*} \g(X_{\tau^*})} = \sup_{\tau}\Ex{x}{\ea{\tau} \g(X_{\tau})},
\ee
where the supremum is taken over all stopping times. 

Particular cases of this problem --with specifics reward or with specifics processes-- were solved. About optimal stopping for spectrally one-sided processes we may cite \cite{avram2004exit,chan}. 
Some works as \cite{darling,mordeckiOSPOLP,boyarchenko2002perpetual,nov-shir} solve optimal stopping problems expressing its solution in terms of the maximum of the process.
The works by \cite{alili-kyp} \cite{christensen2009note} study the validity of the smooth fit principle for jump processes. 
The article \cite{mordecki-salminen} provides a verification theorem for optimal stopping of Hunt processes departing from the Riesz representation of $\desc$-excessive functions.
Wienner-Hopf factorization techniques are used by \cite{surya} and \cite{deligiannidis2009optimal} to characterize the solution of the OSP. The articles \cite{pham1997optimal} and \cite{mordeckiOSDJ} consider the problem of pricing American options for diffusions with jumps.
In the articles \cite{kyp-surya} and \cite{novikov2007solution} are solved the problems with reward $g(x)=({x^+})^n$, and $g(x)=({x^+})^k\colon k>0$ respectively, giving its solution in terms of the roots of the Appel polynomial.
In the recent article by \cite{ChristensenSalminenBao} the authors characterize the solution of a general optimal stopping problem for strong Markov processes using the characterization of $\desc$-excessive functions as expected supremum.

The main theorem we present in this chapter has the following consequence: Assume that:
\begin{itemize}
\item X is a spectrally positive standard Markov process.
\item The reward function $g$ satisfies 
\bd
\g(x)=\int_{\I}-\Aa g(y)\Ga(x,dy)\qquad (x\in\I).
\ed
\item $x^*$ is a solution of 
\bd
\g(x^*)=\int_{\Ig{x^*}}\Ga(x^*,y)\al \g(y)dy
\ed
such that
\begin{itemize}
\item for $x\geq x^*$, $-\Aa g(y)\geq 0$ and
\item $\Va$ defined by 
\bd
\Va(x):=\int_{\Ig{x^*}}-\Aa g(y)\Ga(x,dy),
\ed
satisfies $\Va(x)\geq g(x)$ for $x<x^*$.
\end{itemize}
\end{itemize} 
Then the optimal stopping problem is right-sided with optimal threshold $x^*$ and $\Va$ is the actual value function.
Note that this result is analogous to \autoref{teo:verif1}, being the main differences that for one-dimensional diffusions:
\begin{itemize}
\item $\Ga(x,dy)=\Ga(x,y)m(dy)$, where $\Ga(x,y)$ can be represented in terms of $\phia$ and $\psia$; and
\item $-\Aa g(y) = \al g(y)$.
\end{itemize}

As an application of the obtained results, we consider the optimal stopping problem associated with the pricing of an American put option on a spectrally-negative Lévy market finding a generalization of the results obtained by \cite{chan}. 

At the end of the chapter we solve an OSP whose underlying process is a diffusion with jumps, actually it is a Lévy-driven Ornstein-Uhlenbeck process. The consideration of this process is motivated on prices of energy \cite[see][]{electricity}. Up to our knowledge, this is the first concrete optimal stopping problem solved for a jump-process that is not a Lévy process.

\section{Main results}
We start with a useful lemma concerning the Green kernel of spectrally-positive processes.
\begin{lem} \label{lem:Gsaltos}
Let $X$ be a standard Markov process without negative jumps. If $z<x$ and $H$ is a Borel set such that  $y< z$ for all $y$ in $H$, then
$$\Ga(x,H)=\Ex{x}{ \ea{\hit{z}}} \Ga(z,H),$$
in other words
the ratio between $\Ga(x,H)$ and $\Ga(z,H)$ is independent of $H$.
\end{lem}
\begin{proof}
Since the process does not have negative jumps every path hits any intermediate state to go from $x$ to $H$. In other words, we know that for any trajectory
beginning from $x$ and such that $X_t \in H$ there exists some $s<t$ satisfying
$X_s=x^*$; hence
\begin{align*}
\P_x(X_t \in H)&=\int_0^t  \P_{x}(X_{t}\in H|X_s=x^*)\P_x (\hit{x^*}\in ds)\\
&=\int_0^t  \P_{x^*}(X_{t-s}\in H)\P_x (\hit{x^*}\in ds).
\end{align*}
Using the previous formula we obtain that
 \begin{align*}
  \Ga(x,H)&=\int_0^\infty \ea{t}\P_x(X_t \in H) dt\\%\P_x(X_t\in dy)dt 
&= \int_0^\infty \ea{t} \left(\int_0^t  \P_{x^*}(X_{t-s}\in H)\P_x (\hit{x^*}\in ds)\right) dt\\ 
&= \int_0^\infty   \left( \int_s^\infty \ea{t}\P_{x^*}(X_{t-s}\in H) dt\right) \P_x(\hit{x^*}\in ds),
\end{align*}
where in the last equality we have changed the integration order; for the integral on the right-hand side we have
\begin{align*}
\int_s^\infty \ea{t}\P_{x^*}(X_{t-s}\in H) dt&= \ea{s} \int_s^\infty \ea{(t-s)}\P_{x^*}(X_{t-s}\in H) dt\\
&=\ea{s} \int_0^\infty \ea{t}\P_{x^*}(X_{t}\in H) dt\\
&=\ea{s} \Ga(x^*,H),
\end{align*}
obtaining that
\begin{align*}
  \Ga(x,H)&= \Ga(x^*,H) \int_0^\infty \ea{s} \P_x(\hit{x^*}\in ds)\\
  &=  \Ga(x^*,H) \Ex{x}{\ea{\hit{x^*}}}
\end{align*}
to conclude the proof
 
\end{proof}

\begin{lem}\label{lem:hunt}
Consider a spectrally-positive standard Markov process  $X$. Assume for all $x\in \I$
\begin{equation*}
 g(x)=\int_{\I}f(y)\Ga(x,dy)
\end{equation*}
and suppose $x^*$ is such that
\begin{equation*}
 g(x^*)=\int_{\Ig{x^*}}f(y)\Ga(x^*,dy).
\end{equation*}
Then for all $x\in \Ig{x^*}$ we have
\begin{equation*}
 g(x)=\int_{\Ig{x^*}}f(y) \Ga(x,dy).
\end{equation*}
\end{lem}
\begin{proof}
First observe that, from the definition of $\g$ and the equation defining $x^*$ we conclude that
\begin{equation}
\label{eq:hip}
\int_{\Ile{x^*}}f(y)\Ga(x^*,dy)=0.
\end{equation}
Using the definition of $\g$ we get
\begin{align*}
 g(x)&=\int_{\I}f(y) \Ga(x,dy)\\
&=\int_{\Ile{x^*}}f(y)\Ga(x,dy)+\int_{\Ig{x^*}}f(y)\Ga(x,dy).
\end{align*}
It remains to be proven that, if  $x>x^*$, the first term on the right-hand side of the previous equation vanishes; to do this, consider $x>x^*$, by \autoref{lem:Gsaltos}, we deduce that 
$$H\mapsto \Ga(x,H) \quad \mbox{and} \quad H\mapsto \Ex{x}{\ea{\hit{x^*}}}\Ga(x^*,H)$$
are the same measure in $\Ile{x^*}$; therefore
\begin{align*}
\int_{\Ile{x^*}}f(y)\Ga(x,dy)=\Ex{x}{\ea{\hit{x^*}}} \int_{\Ile{x^*}}f(y)\Ga(x^*,dy),
\end{align*}
and vanishes by equation \eqref{eq:hip}.

\end{proof}

\begin{teo}
\label{teo:huntVerif} Consider a spectrally-positive standard Markov process $X$, 
and $\gi:\I\mapsto \R$ such that
\be \label{eq5:gintf}
 \gi(x)=\int_{\I}f(y)\Ga(x,dy).
\ee
Assume $x^*$ is a root of 
\be \label{eq:x*saltos}
\gi(x^*)=\int_{\Ig{x^*}}f(y)\Ga(x^*,dy),
%\gi(x^*)=\int_{\Ig{x^*}}f(y)\Ga(x^*,dy),
\ee
such that $f(x)\geq 0$ for all $x$ in $\Ig{x^*}$. Define 
\bd 
\Va(x):=\int_{\Ig{x^*}} f(y) \Ga(x,dy).
\ed
Consider the reward function $\g$ that satisfies the conditions for Dynkin's characterization (see \autoref{sec:DynkinChar}), such that $\g(x)=\gi(x)$ for $x\geq x^*$. If $\Va(x) \geq g(x)$ for all $x$ in $\Ile{x^*}$, then the optimal stopping
problem \eqref{eq:ospChap5} with reward function $\g$ and discount rate $\desc$ is right-sided, 
$x^*$ is an optimal threshold and $\Va$ is the value function.
\end{teo}
\begin{remark}
This theorem is analogous to \autoref{teo:verif1}. Observe that function $f$ takes the part of $\al g$ and function $\gi$ is given in the \autoref{cond:RRC}. The difference in this case is that $\Ga(x,dy)$ no necessarily has a representation as $\Ga(x,y)m(dy)$ and there are not functions $\phia$ and $\psia$.
\end{remark}
\begin{remark}
It is also interesting to compare this result with Theorem 3.1 in \cite{mordecki-salminen}. In that theorem there is a condition $\Va=g$ for $x \geq x^*$, while we just need $\Va(x^*)=g(x^*)$. This difference is a consequence of the fact that in our theorem we consider spectrally-positive processes (see \autoref{lem:hunt}).
\end{remark}

\begin{proof}
By hypothesis $f(y)$ is non-negative for $y$ in $\Ig{x^*}$, then we have that $\Va$ is an $\desc$-excessive function. 
By the application of \autoref{lem:hunt} we deduce that $\Va(x)$ coincides with $\gi(x)$ for $x$ in $\Ig{x^*}$, therefore also coincides with $\g$. 
By hypothesis we obtain that $\Va$ dominates $\g$ in $\Ile{x^*}$. So $\Va$ is a majorant of $\g$ and, by Dynkin's characterization of the value function, we conclude that 
\begin{equation*}
\Va(x) \geq \sup_{\tau} \Ex{x}{e^{-\desc \tau}\g(X_\tau)}.
\end{equation*}
To conclude that the other inequality also holds, we apply \autoref{lem:generalMin} with $B:=\Ig{x^*}$ and $F_B:=\Va$ 
obtaining that 
\bd
\Va(x)=\Ex{x}{e^{-\desc \hit{\Ig{x^*}}}\Va(X_{\hit{\Ig{x^*}}})}.
\ed
Since the trajectories are right continuous, it gathers that $X_{\hit{\Ig{x^*}}}$ belongs to $\Ige{x^*}$, 
the region in which $\Va$ and $\g$ coincide; therefore
\bd
\Va(x)=\Ex{x}{e^{-\desc \hit{\Ig{x^*}}}g(X_{\hit{\Ig{x^*}}})},
\ed
proving 
\bd
\Va(x) \leq \sup_{\tau} \Ex{x}{e^{-\desc \tau}\g(X_\tau)}.
\ed
We have proved the desired equality concluding that the optimal stopping
problem is right-sided with threshold $x^*$.

\end{proof}

\begin{cor}
\label{cor:mainHuntVerif} Consider a strong Markov process $X$ with no negative jumps
and a reward function $\g\colon\estados \to \R$ that belongs to the domain $\Da$
of the extended infinitesimal generator associated with the $\desc$-killed process and satisfying
\begin{itemize}
\item $\lim_{t\to \infty}\Ex{x}{\ea{t}g(X_t)}=0$ and
\item $\Ex{x}{\int_0^{\infty} \ea{s}|\Aa g(X_s)| ds}<\infty$.
\end{itemize} 
Assume as well that function $g$ satisfies the conditions for Dynkin's characterization (see \autoref{sec:DynkinChar}). Suppose that $x^*\in \I$ is a solution of
\begin{equation*}
g(x^*)=\int_{\Ig{x^*}}(-\Aa \g(y))\Ga(x^*,dy),
\end{equation*}
such that $\Aa \g(x) \leq 0$ for all $x$ in $\Ig{x^*}.$
Define 
\begin{equation*}
\Va(x)=\int_{\Ig{x^*}} (-\Aa \g(y)) \Ga(x,dy).
\end{equation*}

If $\Va(x) \geq \g(x)$ for all $x\leq x^*$ then the OSP \eqref{eq:osp} is right-sided, $x^*$ is an optimal threshold and $\Va$ is
the value function.
\end{cor}

\begin{proof}
Observe that, by the assumptions on $\g$, for all $x$ in $\I$, \eqref{eq:invExtended} holds, that is,
\bd
g(x)=\int_{\I}(-\Aa \g(y))\Ga(x^*,dy).
\ed
So, all the hypotheses of the previous theorem are fulfilled with $f(x):=-\Aa \g(x)$ and $\gi:=g$,  this result leading to the thesis.

\end{proof}

The analogous result of \autoref{teo:huntVerif} for spectrally negative processes is as follows:
\begin{teo}
\label{teo:huntVerifnegative} Consider a spectrally-negative standard Markov process $X$, and $\gi:\I\mapsto \R$ such that \eqref{eq5:gintf} holds. Assume $x^*$ is a root of 
\begin{equation*}
\gi(x^*)=\int_{\Il{x^*}}f(y)\Ga(x^*,dy),
%\gi(x^*)=\int_{\Ig{x^*}}f(y)\Ga(x^*,dy),
\end{equation*}
such that $f(x)\geq 0$ for all $x$ in $\Il{x^*}$. Define 
\begin{equation*}
\Va(x)=\int_{\Il{x^*}} f(y) \Ga(x,dy).
\end{equation*}
Consider the reward function $\g$ that satisfies the conditions for Dynkin's characterization (see \autoref{sec:DynkinChar}), such that $\g(x)=\gi(x)$ for $x\leq x^*$. If $\Va(x) \geq g(x)$ for all $x$ in $\Ige{x^*}$, then the optimal stopping
problem \eqref{eq:ospChap5} with reward function $\g$ and discount rate $\desc$ is left-sided, 
$x^*$ is an optimal threshold and $\Va$ is the value function.
\end{teo}

\section{Applications}
\subsection{American put option on a Lévy market}
A particularly interesting subclass of the kind of processes with which we are dealing in this chapter are the Lévy processes in $\R$. A right continuous with left hand limits process $X$ is said to be a Lévy process provided that for every $s,t\geq 0$ the increment $X_{t+s}-X_t$ 
\begin{itemize}
\item is independent of the process $\{X_v\}_{0\leq v\leq t}$ and
\item has the same law as $X_s$.
\end{itemize}
It can be seen, as a consequence of the definition, that every Lévy process satisfy $\P(X_0=0)=1$.

Lévy-Khintchine representation for Lévy processes in $\R$ states that every Lévy process can be characterized by a triplet 
$(a, \sigma, \Pi)$, with $a\in \R$, $\sigma\geq 0$ and $\Pi$ a measure supported in $\R\setminus 0$ that satisfies
\be \label{eq:jumpmeasurecond}
\int_{\R\setminus\{0\}}\min\{x^2,1\}\Pi(dx)<\infty.
\ee
The relation between the Lévy process and the \emph{characteristic triplet} is the fact that for every $z\in i\R$
\be
\Ex{}{e^{z X_t}}=e^{t \Psi(z)}
\ee
where the so called \emph{characteristic exponent} $\Psi(z)$ is given by
\bd
\Psi(z)= a z + \frac12 \sigma^2 z^2 + \int_{\R}\left(e^{z x}-1-z x\ind{|x|<1}\right) \Pi(dx).
\ed
In this case we have
\bd X_t=X_0 + a t+\sigma B_t +J_t \ed
with
\bd J_t=\int_{(0,t]\times\{|x|\geq 1\}}x\mu^J(\omega,dt,dx)+\int_{(0,t]\times\{|x|<1\}}x\left(\mu^X(\omega,dt,dx)-\nu(dt,dx)\right). \ed
The random measure $\mu^X$ used in the previous formula is defined by \citep[see][Proposition 1.16]{jacod1987limit}
\be \label{eq:defPoissRandMeasure}
\mu^X(\omega,dt,dx)=\sum_s \ind{\{\Delta X_s(\omega)\neq 0\}}\delta_{(s,\Delta X_s(\omega))}(dt,dx)
\ee
($\Delta X_t=X_t-lim_{t\to s^-}X_s$), and $\nu$, the compensator measure, is in this case $\nu(dt,dx)=dt\Pi(dx)$.

\begin{lem} \label{lem:levy}
Given a Lévy process $\{X_t\}$ with characteristic triplet $(a,\sigma,\Pi)$ and a function $h\colon \R \to \R$ such that $h$ is bounded, twice differentiable with continuous and bounded derivatives. Then
\bd
h(x)=\int_{(0,\infty)}\ea{t}\Ex{x}{\desc h(X_t)-Lh(X_t)}dt
\ed
where $Lh$ is given by
\be
\label{eq:LLevy}
Lh(x)=a h'(x) + \frac{\sigma^2}2 h''(x) + \int_{\R}\left(h(x+y)-h(x)-\ind{|y|<1}yh'(x)\right) \Pi(dy).
\ee

\end{lem}
\begin{proof}
We apply Ito's formula \cite[see][p. 82]{protter2005stochastic} to $f\colon f(s,x)=\ea{s}h(x)$ and $Y_t=(t,X_t)$, obtaining that
\begin{align} \label{eq:itosemimartingale}
\notag  \ea{t}h(X_t)-h(X_0)&=\int_{(0,t]}-\desc \ea{s}h(X_{s^-})ds+\int_{(0,t]}\ea{s}h'(X_{s^-})dX_s\\
\notag &\quad+\frac12 \int_{(0,t]}\ea{s}h''(X_{s^-})\sigma^2 dt\\
 &\quad+\sum_{0< s\leq t}\ea{t}\left(h(X_s)-h(X_{s^-})-h'(X_{s^-})\Delta X_s \right).
\end{align}
Before taking expectation in the previous formula we analyse the second and the last term on its right-hand side. For the second term we have
\begin{align*} \int_{(0,t]}\ea{s}h'(X_{s^-})dX_s&=\int_{(0,t]}\ea{s}a h'(X_{s^-})ds+\int_{(0,t]}\ea{s}\sigma h'(X_{s^-})dB_s\\
&\qquad +\int_{(0,t]}\ea{s}h'(X_{s^-})dJ_s\\
&=\int_{(0,t]}\ea{s}a h'(X_{s^-})ds+\int_{(0,t]}\ea{s}\sigma h'(X_{s^-})dB_s\\
&\qquad +\int_{(0,t]\times \{0<|y|<1\}}\ea{s}h'(X_{s^-})y \left(\mu^X(\omega,dt,dy)-\nu(dt,dy)\right)\\
&\qquad +\int_{(0,t]\times \{|y|\geq 1\}}\ea{s}h'(X_{s^-})y \mu^X(\omega,dt,dy);
\end{align*}
while for the last last term we have
\begin{align*}
&\sum_{0< s\leq t}\ea{t}\left(g(X_s)-h(X_{s^-})-h'(X_{s^-})\Delta X_s \right)\\
&\quad =\int_{(0,t]\times \R^*}\ea{s}\left(h(X_{s^-}+y)-h(X_{s^-})-h'(X_{s^-})y\right)\mu^{X}(\omega,ds,dy)
%&\quad =\int_{(0,t]\times \R^*}\ea{s}\left(h(X_{s^-}+x)-h(X_{s^-})-h'(X_{s^-})x\right)\left(\mu^{X}(\omega,ds,dx)-\nu(dt,dx)\right)\\
%&\qquad + \int_{(0,t]\times \R^*}\ea{s}\left(h(X_{s^-}+x)-g(X_{s^-})-h'(X_{s^-})x\right)\nu(dt,dx).
\end{align*}
Going back to \eqref{eq:itosemimartingale} with the previous computations in mind, we obtain that
\begin{align} \label{eq:itosemimartingale2}
\notag & \ea{t}h(X_t)-h(X_0)=\int_{(0,t]} \ea{s}\left(-\desc h(X_{s^-})+ah'(X_{s^-})+\frac{\sigma^2}2 h''(X_{s^-})\right)ds \\
&\qquad +\int_{(0,t]}\ea{s}\sigma h'(X_{s^-})dB_s\\
\notag &\qquad +\int_{(0,t]\times \{0<|y|<1\}}\ea{s}h'(X_{s^-})y \left(\mu^X(\omega,dt,dy)-\nu(dt,dy)\right)\\
%&\qquad +\int_{(0,t]\times \{|x|>1\}}\ea{s}h'(X_{s^-})x \mu^X(\omega,dt,dx);\\
\notag &\qquad +\int_{(0,t]\times \R^*}\ea{s}\left(h(X_{s^-}+y)-h(X_{s^-})-h'(X_{s^-})y\ind{\{0<|y|<1\}}\right)\mu^{X}(\omega,ds,dy)%\left(\mu^{X}(\omega,ds,dx)-\nu(dt,dx)\right)\\
%&\qquad + \int_{(0,t]\times \R^*}\ea{s}\left(g(X_{s^-}+x)-g(X_{s^-})-g'(X_{s^-})x\right)\nu(dt,dx).
\end{align}
Now we take the expectation and then we take the limit as $t\to \infty$: Regarding to the left-hand side, we have
\bd
\lim_{t\to \infty} \Ex{x}{\ea{t}h(X_t)-h(X_0)}=-h(x)
\ed
as $h$ is bounded, $\ea{t}\to 0$, and $\Ex{x}{X_0}=x$. About the right-hand side in \eqref{eq:itosemimartingale2}, we analyse each term:
\begin{align*}
& \lim_{t\to \infty}\Ex{x}{\int_{(0,t]} \ea{s}\left(-\desc h(X_{s^-})+ah'(X_{s^-})+\frac{\sigma^2}2 h''(X_{s^-})\right)ds }\\
& \qquad \qquad \qquad=\int_{(0,t]} \ea{s}\Ex{x}{-\desc h(X_{s^-})+ah'(X_{s^-})+\frac{\sigma^2}2 h''(X_{s^-})}ds,
\end{align*}
since $h$, $h'$, and $h''$ are bounded functions; the second term vanishes, as it is an integral with respect to a martingale of a predictable integrable function; with respect to the sum of the third and the fourth terms, observe that, before taking the expectation and the limit, we can rewrite it as
\begin{align*}
&\int_{(0,t]\times \R^*}\ea{s}\left(h(X_{s^-}+y)-h(X_{s^-})-h'(X_{s^-})y\ind{\{0<|y|<1\}}\right)\nu(ds,dy)\\
&\qquad +\int_{(0,t]\times \R^*}\ea{s}\left(h(X_{s^-}+y)-h(X_{s^-})\right) \left(\mu^X(\omega,dt,dy)-\nu(dt,dy)\right),
%&\qquad +\int_{(0,t]\times \{|x|>1\}}\ea{s}h'(X_{s^-})x \mu^X(\omega,dt,dx);\\
\end{align*}
which after taking the expectation and the limit, we will see, becomes
\bd
\int_{(0,\infty)}\ea{s}\Ex{x}{\int_{\R^*}\left(h(X_{s^-}+y)-h(X_{s^-})-h'(X_{s^-})y\ind{\{0<|y|<1\}}\right)\Pi(dy)}dt
\ed
and we would have completed the proof.

We still have to justify
\begin{align*}
&\lim_{t\to \infty} \Ex{x}{\int_{(0,t]\times \R^*}\ea{s}\left(h(X_{s^-}+y)-h(X_{s^-})-h'(X_{s^-})y\ind{\{0<|y|<1\}}\right)\nu(ds,dy)}\\
&\quad= \int_{(0,\infty)}\ea{s}\Ex{x}{\int_{\R^*}\left(h(X_{s^-}+y)-h(X_{s^-})-h'(X_{s^-})y\ind{\{0<|y|<1\}}\right)\Pi(dy)}dt,
\end{align*}
and
\bd
\lim_{t\to \infty}\Ex{x}{\int_{(0,t]\times \R^*}\ea{s}\left(h(X_{s^-}+x)-h(X_{s^-})\right) \left(\mu^X-\nu\right)}=0.
\ed
To prove the former equality one can observe that, by Taylor formula and because of the boundedness of $h$ and $h'$, there exist a constant $K$ such that
\bd
|h(X_{s^-}+y)-h(X_{s^-})-h'(X_{s^-})y\ind{\{0<|y|<1\}}|<K \min\{1,y^2\};
\ed
then, considering \eqref{eq:jumpmeasurecond}, and $\nu(ds,dy)=\Pi(dy)ds$, the equality follows by application of Fubini's theorem. The latter equality is a direct consequence of  
\bd
M_t=\int_{(0,t]\times \R^*}\ea{s}\left(h(X_{s^-}+y)-h(X_{s^-})\right) \left(\mu^X-\nu\right)
\ed
being a martingale, fact that can be seen as an application of Theorem 1.33 in \cite{jacod1987limit}, since 
\bd
H_t=\int_{(0,t]\times \R^*}\left(\ea{s}\left(h(X_{s^-}+y)-h(X_{s^-})\right)\right)^2 \nu(ds,dy)
\ed
is an integrable increasing process; to see that $H$ is in fact integrable we need to check that
\bd
\int_{\R^*}\left(h(X_{s^-}+y)-h(X_{s^-})\right)^2\Pi(dy)
\ed
is bounded, what can be done similarly as above, by observing that there exist some constant $K$ such that $\left(h(X_{s^-}+y)-h(X_{s^-})\right)^2<K\min\{1,x^2\}$.

\end{proof}

%Observe that in this case $Lf(x)$ depends only on values of $f(y)$ for $y\in (x-\delta, \infty)$. 
%This fact give us the possibility of defining $Lf(x)$ for functions that are not in the domain of the infinitesimal generator as we did in the case of one-dimensional diffusions. We say that $\g$ satisfies the right regularity condition (RRC) for $x$ if there exists $\tilde{g}$ such that $\tilde{g}(y)=g(y)$ for $y\geq x$ and 
%\bd
%\int_{\R}\al \tilde{g}(y) \Ga(x,dy)=\tilde{g}(x) \qquad (x\in \R).
%\ed
%Now we state a relaxed version of \autoref{teo:huntVerif}
%\begin{teo}
%\label{teo:LevyVerif} Consider a Levy process $X$ without negative jumps,
%and $\g:\I\mapsto \R$ satisfying the RRC for $x^*$ a solution of 
%\begin{equation*}
%g(x)=\int_{(x,\infty)}\al \g(y)\Ga(x,dy),
%\end{equation*}
%such that $\al \g(x)\geq 0$ for all $x>x^*$. Define 
%\begin{equation*}
%\Va(x)=\int_{(x^*,\infty)}\al \g(y) \Ga(x,dy).
%\end{equation*}
%If $\Va(x) \geq \g(x)$ for all $x<x^*$, then the optimal stopping
%problem \eqref{eq:osp} with reward function $\g$ and discount $\desc$ is right-sided, 
%$x^*$ is an optimal threshold and $\Va$ is the value function.
%\end{teo}

%\subsection{American put option on Lévy market}
Consider the optimal stopping problem \eqref{eq:ospChap5} where $\{X_t\}_{t\geq 0}$ is a Lévy process with only negative jumps, i.e. $\Pi$ is supported in $(-\infty,0)$, and the reward function $\g\colon \R \to \R$ defined by $\g(x)=(K-e^x)^+$. 
We remark that this is the kind of optimal stopping problems one has to solve to price an American put option.

First, we observe that the optimal stopping region is contained in the set $\{x:g(x)>0\}=(-\infty,\ln(K))$, where the reward function is $\g(x)=K-e^x$. 
We may construct a function $\gi$ bounded and with continuous second derivatives, the derivatives also bounded, such that $\gi(x)=\g(x)$ for $x<\ln(K)$. 
Applying \autoref{lem:levy} to $\gi$ and using Fubini's theorem we conclude that \eqref{eq5:gintf} is fulfilled with $f\colon f(x)=\desc \gi(x)-L \gi(x)$. Assume $x^*<\ln(K)$ is a solution of
$$\gi(x^*)=\int_{(-\infty,x^*)} \al \gi(y) \Ga(x^*,y)dy,$$
such that $\al \g(x)\geq 0$ for $x<x^*$. 
According to \autoref{teo:huntVerifnegative}, 
we have that $\Va$, defined by
$$\Va(x)=\int_{(-\infty,x^*)} \al \gi(y) \Ga(x,y)dy,$$
is the value function of the optimal stopping problem providing that $\Va(x)\geq \g(x)$ for $x>x^*$. 
Even though the definition of $\Va$ is given in terms of $\gi$, it is in fact independent on the extension of $\g$ chosen, since, due to the restriction on the jumps of the process, the measure $\Pi$ is supported in $(-\infty,0)$, so, for $x<\ln(K)$
\begin{align*}
\al \gi(x)&= \desc \g(x) - a \g'(x) - \frac{\sigma^2}2 g''(x) \\
&\qquad - \int_{(-\infty,0)}\left(g(x+y)-g(x)-y\ind{|y|<1}g'(x)\right) \Pi(dy).
\end{align*}
Substituting in the previous equation $g(x)$ by $K-e^x$, we obtain that
\begin{align*}
\al \gi(x)&= \desc K -\desc e^x + a e^x + \frac{\sigma^2}2 e^x \\
&\qquad - e^x \int_{(-\infty,0)}\left( 1-e^y+y\ind{|y|<1} \right)\Pi(dy)\\
&=\desc K-(\desc-\Psi(1))e^x.
\end{align*}
Considering this previous equality, we have the representation for the value function
\bd
\Va(x)=\int_{(-\infty,x^*)}\left(\desc K-(\desc-\Psi(1))e^x \right) \Ga(x,y)dy.
\ed
For some financial applications it is assumed risk-neutral  market, i.e. the process $e^{-rt}X_t$ is a martingale, what is equivalent to $\desc-\Psi(1)=0$. Under this assumption the value function would be
\bd
\Va(x)=\desc K\int_{(-\infty,x^*)}\Ga(x,y)dy,
\ed
obtaining formula (9.5.2) in \cite[p. 207]{chan}.
\subsection{Lévy-driven Ornstein-Uhlenbeck with positive jumps}

Let $X$ be a Lévy-driven Ornstein-Uhlenbeck process; i.e. a process satisfying the stochastic differential equation
\be \label{eq:orn-uhl-sde}
dX_t=-\gamma X_{t^-}dt+dL_t,
\ee
where $\{L_t\}$ is a Lévy process.
The consideration of this process is motivated by its application to model electricity markets \cite[see][]{electricity}.
The only solution of the equation \eqref{eq:orn-uhl-sde} is \citep[see][]{novikov2006levyornstein}
\be \label{eq:orn-uhl-sde-sol}
X_t=e^{-\gamma t}\left(\int_0^t e^{\gamma s}dL_s+X_0\right).
\ee
In our example we consider $L_t= \sigma B_t+ J_t$ where $\{J_t\}$ is a compound Poisson process with rate $\lambda$ and jumps with exponential distribution of parameter $\beta$; i.e.
\bd
J_t=\sum_{i=1}^{N_t}Y_i,
\ed
with $\{N_t\}$ a Poisson process with rate $\lambda$ and $Y_i$ independent identically distributed random variables, with exponential distribution of parameter $\beta$. Observe that there are only positive jumps.

We aim to solve the optimal stopping problem \eqref{eq:ospChap5} with reward function $g\colon g(x)=x^+$, i.e. to find the stopping time $\tau^*$ and the value function $\Va$ such that
\bd
\Va(x)=\Ex{x}{\ea{\tau^*}{X_{\tau^*}}^+}=\sup_{\tau}\left(\Ex{x}{\ea{\tau}{X_{\tau}}^+}\right).
\ed
In order to solve this problem we apply \autoref{teo:huntVerif} with $\gi\colon \gi(x)=x$; hence we need to find $f$
satisfying \eqref{eq5:gintf}. Consider the following equalities
\begin{align}\label{eq:itoO-U}
\ea{t}X_t-X_0&=\int_{(0,t]}X_{s^-}(-\desc \ea{s})ds+\int_{(0,t]}\ea{s}dX_s\\
\notag &=-\int_{(0,t]}X_{s^-} (\desc+\gamma) \ea{s}ds +\int_{(0,t]}\ea{s} \sigma dB_s+\int_{(0,t]}\ea{s} dJ_s.
\end{align}
The expected value of the integral with respect to $\{B_s\}$ vanishes. Concerning the integral with respect to the jump process $\{J_s\}$, we can write it in terms of the jump measure $\mu$ -defined in \eqref{eq:defPoissRandMeasure}- as 
\begin{align*}
\int_{(0,t]}\ea{s} dJ_s&=\int_{(0,t]\times \R}\ea{s} y \mu(\omega,ds,dy)\\
&=\int_{(0,t]\times \R}\ea{s} y (\mu(\omega,ds,dy)-\nu(ds,dy)) +\int_{(0,t]\times \R}\ea{s} y \nu(ds,dy),
\end{align*}
where $\nu$, the compensator of $\mu$, in this case is given by 
\bd
\nu(ds,dy)=\lambda\beta \ind{\{y>0\}}e^{-\beta y}dy ds.
\ed
From the application of Corollary 4.6 in \cite{kyprianou2006introductory}, it follows that
\bd
M_t=\int_{(0,t]\times \R}\ea{s} y (\mu(\omega,ds,dy)-\nu(ds,dy))
\ed
is a martingale, then $\Ex{x}{M_t}=\Ex{x}{M_0}=0$. It follow that
\begin{align*}
\Ex{x}{\int_{(0,t]}\ea{s} dJ_s}&=\int_{(0,t]\times \R}\ea{s} y \nu(ds,dy)\\
&=\int_{(0,t]}\int_{\R^+} \ea{s} y \lambda\beta e^{-\beta y}dy ds\\
&=\int_{(0,t]} \ea{s} \frac\lambda\beta  ds
\end{align*}
Taking the expectation in \eqref{eq:itoO-U} we obtain that
\be \label{eq:itoO-U-expectation}
\Ex{x}{\ea{t}X_t}-x=-\Ex{x}{\int_{(0,t]}\left(X_{s^-} (\desc+\gamma)-\frac\lambda\beta \right)\ea{s}ds}.
\ee
Using \eqref{eq:orn-uhl-sde-sol} we compute $\Ex{x}{X_t}$:
\begin{align*}
\Ex{x}{X_t}&=\Ex{x}{e^{-\gamma t}\left(\int_0^t e^{\gamma s}dL_s+X_0\right)}\\
&= e^{-\gamma t}\left(\Ex{x}{\int_0^t e^{\gamma s}\sigma dB_s}+\Ex{x}{\int_0^t e^{\gamma s}\sigma dJ_s} + x\right)\\
&= (1-e^{-\gamma t})\frac\lambda{\beta\gamma}+ x e^{-\gamma t},
\end{align*}
concluding that $\lim_{t\to \infty} \ea{t}\Ex{x}{X_t}=0$. With similar arguments we obtain $\Ex{x}{|X_t|}\leq \frac{1}{\sqrt{\pi \gamma}} + \frac{\lambda t}{\beta}$. 
We can change the order between the expectation and the integral on the right-hand side of \eqref{eq:itoO-U-expectation}. Taking the limit as $t\to \infty$ in \eqref{eq:itoO-U-expectation} we obtain that
\bd
-x=-\int_0^\infty \Ex{x}{X_s(\desc+\gamma)-\frac\lambda\beta}\ea{s}ds.
\ed
The previous equality can be written in terms of the Green kernel by
\be \label{eq:inversionO-U}
x=\int_{\R} \left(y(\desc+\gamma)-\frac\lambda\beta\right)\Ga(x,dy)
\ee
which is \eqref{eq5:gintf} with 
\be \label{eq:fparainveOrnstein}
f(y)=y(\desc+\gamma)-\frac\lambda\beta.
\ee
Now we move on to find the Green kernel of the process.

It can be seen that for the considered process there exist a function $\Ga(x,y)$ such that $\Ga(x,dy)=\Ga(x,y)dy$. As we can not find $\Ga(x,y)$ explicitly we compute its Fourier transform,
\begin{align*}
\hat{\Ga}(x,z)&=\int_{-\infty}^\infty e^{izy}\Ga(x,y)dy\\
&= \int_0^\infty \ea{t} \int_{-\infty}^\infty e^{izy} \P_x(X_t\in dy) dt\\
&= \int_0^\infty \ea{t} \Ex{x}{e^{izX_t}} dt.
\end{align*}
We need to compute $\Ex{x}{e^{izX_t}}$. In order to do that we apply Dynkin's formula to $u(x)=e^{izx}$. We have
\[u'(x)=iz u(x)\quad \text{and} \quad u''(x)=-z^2 u(x).\]
and
\begin{align*}
Lu(x)&=-\gamma x iz u(x)+\frac{-z^2}{2} u(x)+u(x) \lambda \beta \int_0^\infty \left(e^{izy}-1\right)e^{-\beta y}dy\\
&= u(x)\left(-\gamma xiz+\frac{-z^2}{2}+\frac{\lambda \beta}{\beta-i z}-\lambda \right)\\
&= u(x)\left(-\gamma xiz+\frac{-z^2}{2}+\frac{iz \lambda}{\beta-i z} \right).
\end{align*}
By Dynkin's formula we obtain that
\begin{align*}
\Ex{x}{e^{izX_t}}-e^{izx}&=\Ex{x}{\int_0^t u(X_s)\left(-\gamma X_s iz+\frac{-z^2}{2}+\frac{iz \lambda}{\beta-i z} \right) ds} 
\end{align*}
Denoting by $h(x,t,z)=\Ex{x}{e^{i z X_t}}=\Ex{x}{u(X_t)}$ we have 
\[h_z(x,t,z)=\Ex{x}{i X_t u(X_t)}\]
 and the previous equation is
\be \label{eq:hxtz}
h(x,t,z)-e^{izx}=\int_0^t -\gamma z h_z(x,s,z) + \left(-\frac{z^2}{2}+\frac{\lambda i z}{\beta-iz}\right)h(x,s,z)ds.
\ee
Instead of solving the previous equation we try to find directly 
$\hat{\Ga}(x,z)$. Remember that
\begin{align*}
\hat{\Ga}(x,z)&= \int_0^\infty \ea{t} h(x,t,z) dt.
\end{align*}
Taking Laplace transforms in \eqref{eq:hxtz} we obtain that
\begin{align*}
\hat{\Ga}(x,z)-e^{izx}/\desc &= \int_0^\infty ds \int_s^\infty \left(-\gamma z h_z(x,s,z) \ea{t}+ \left(-\frac{z^2}{2}+\frac{\lambda i z}{\beta-iz}\right)h(x,s,z)\ea{t} \right) dt\\
&= \frac{1}{\desc} \int_0^\infty \left(-\gamma z h_z(x,s,z) \ea{s}+ \left(-\frac{z^2}{2}+\frac{\lambda i z}{\beta-iz}\right)h(x,s,z)\ea{s}\right)  ds,
\end{align*}
which is equivalent to
\bd
\desc \hat{\Ga}(x,z) - e^{izx} = -\gamma z \frac{\partial \hat{\Ga}}{\partial z}(x,z)+\left(-\frac{z^2}{2}+\frac{\lambda i z}{\beta-iz}\right) \hat{\Ga}(x,z)
\ed
and to 
\be
\label{eq:ecuacion}
 \left(\desc+\frac{z^2}{2}-\frac{\lambda i z}{\beta-iz}\right)  \hat{\Ga}(x,z)+\gamma z \frac{\partial \hat{\Ga}}{\partial z}(x,z)  = e^{izx}.
\ee

About the initial condition, observe that $\hat{\Ga}$ satisfies
\begin{align*}
\hat{\Ga}(x,0)&=\int_{-\infty}^\infty \Ga(x,dy)\\
&= \int_0^\infty \ea{t} \int_{-\infty}^\infty \P_x(X_t\in dy) dt\\
&= \int_0^\infty \ea{t} dt = \frac{1}{\desc}.
\end{align*}

We solve explicitly  \eqref{eq:ecuacion}: Let us start by solving the homogeneous equation
\bd
 \left(\desc+\frac{z^2}{2}-\frac{\lambda i z}{\beta-iz}\right)  H(z)+\gamma z H_z(z)  = 0,
\ed
obtaining that
\bd
\frac{H_z(z)}{H(z)}  = -\frac{\left(\desc+\frac{z^2}{2}-\frac{\lambda i z}{\beta-iz}\right)}{ \gamma z },
\ed
and
\bd
\log(H(z))=\frac{-\desc}{\gamma}\log(|z|)+\frac{-1}{4\gamma}z^2-\frac{\lambda}{\gamma}\log(\beta-iz),
\ed
then
\bd
H(z)=e^{-\frac{1}{4\gamma}z^2}|z|^{-\frac{\desc}{\gamma}}(\beta-iz)^{-\frac{\lambda}{\gamma}}.
\ed

The solution of \eqref{eq:ecuacion} is given by
\begin{align}
\notag \hat{\Ga}(x,z)&=\frac{H(z)}{\gamma}\int_{0}^z  \frac{e^{i\zeta x}}{\zeta H(\zeta)}d\zeta\\ 
\label{eq:Ggorro} &=\frac{1}{\gamma}\left(e^{-\frac{1}{4\gamma}z^2}|z|^{-\frac{\desc}{\gamma}}(\beta-iz)^{-\frac{\lambda}{\gamma}}\right)\int_{0}^z  e^{i\zeta x}e^{\frac{1}{4\gamma}\zeta^2} \frac{|\zeta|^{\frac{\desc}{\gamma}}}{\zeta}(\beta-i\zeta)^\frac{\lambda}{\gamma} d\zeta
\end{align}
for $z\neq 0$.
Observe that $H(z)\to \infty$ as $z\to 0$, being equivalent to $|z|^{-\frac{\desc}{\gamma}}\beta^{-\frac{\lambda}{\gamma}}$. On the other hand
\bd
\int_{0}^z  \frac{e^{i\zeta x}}{\zeta H(\zeta)}d\zeta \to 0\quad (z\to 0)
\ed
since the integral is convergent. We can use l'H\^opital rule to compute the limit of $\hat{\Ga}(x,z)$ when $z$ goes to 0. We obtain that
\begin{align*}
\lim_{z\to 0}\hat{\Ga}(x,z)&=\lim_{z\to 0} \frac{H(z)}{\gamma}\int_{0}^z  \frac{e^{i\zeta x}}{\zeta H(\zeta)}d\zeta \\
&=\lim_{z\to 0} \frac{1}{\gamma}\left(|z|^{-\frac{\desc}{\gamma}}\beta^{-\frac{\lambda}{\gamma}}\right)\int_{0}^z  \frac{e^{i\zeta x}}{\zeta H(\zeta)}d\zeta \\
&=\lim_{z\to 0} \frac{1}{\gamma}\frac{\int_{0}^z  \frac{e^{i\zeta x}}{\zeta H(\zeta)}d\zeta }{|z|^{\frac{\desc}{\gamma}}\beta^{\frac{\lambda}{\gamma}}} \\
&=\lim_{z\to 0} \frac{1}{\gamma}\frac{e^{i z x}e^{\frac{1}{4\gamma} z^2} |z|^{\frac{\desc}{\gamma}-1}(\beta-i z)^\frac{\lambda}{\gamma} }{\frac{\desc}{\gamma} |z|^{\frac{\desc}{\gamma}-1}\beta^{\frac{\lambda}{\gamma}}}  =\frac{1}{\desc}
\end{align*}
concluding that the solution we found satisfies the initial condition.

We have obtained an expression for $\hat{\Ga}(x,z)$, which allow us, for particular values of the parameters, to compute a discretization of $\hat{\Ga}(x,z)$. From this discretization, using the discrete Fourier transform, we find a discretization of $\Ga(x,y)$ (we have written an R script to do this, see \autoref{numerical}) necessary to solve equation \eqref{eq:x*saltos} in \autoref{teo:huntVerif}.

\begin{example}[$\beta=\desc=\gamma=\lambda=1$]
Consider the process already presented with parameters $\beta=\desc=\gamma=\lambda=1$. Equation \eqref{eq:Ggorro} is 
\begin{align*}
\hat{\Ga}(x,z)&=\left(e^{-\frac{1}{4}z^2}|z|^{-1}(\beta-iz)^{-1}\right)\int_{0}^z  e^{i\zeta x}e^{\frac{1}{4}\zeta^2}\frac{|\zeta|}{\zeta} (\beta-i\zeta) d\zeta\\
&=\left(e^{-\frac{1}{4}z^2}z^{-1}(\beta-iz)^{-1}\right) \\
&\qquad \left(i \sqrt{\pi} e^{x^2} (\beta-2x)\left(\erf\left(x-\frac{iz}{2}\right)-\erf(x)\right) 
-2 i (e^{izx+\frac{1}{4}z^2}-1)\right)
\end{align*}

Remember that we are considering the reward function $\g(x)=x^+$. To solve numerically equation \eqref{eq:x*saltos} we use: $\tilde{g}(x)=x$; function $f$ given in \eqref{eq:fparainveOrnstein}; and the discretization of $\Ga(x,y)$ obtained numerically as described above. The solution we found is $x^*= 1.1442$. \autoref{fig:ornstein-jumps} shows some points of the value function, obtained numerically by the formula
\bd
\Va(x)=\int_{x^*}^\infty \Ga(x,y)f(y) dy.
\ed
We also include in the plot the reward function (continuous line). Observe that for $x<x^*$ (in the continuation region) $\Va>g$ and the hypothesis of \autoref{teo:huntVerif} is fulfilled.
\begin{figure}
\begin{center}
\includegraphics[scale=.6]{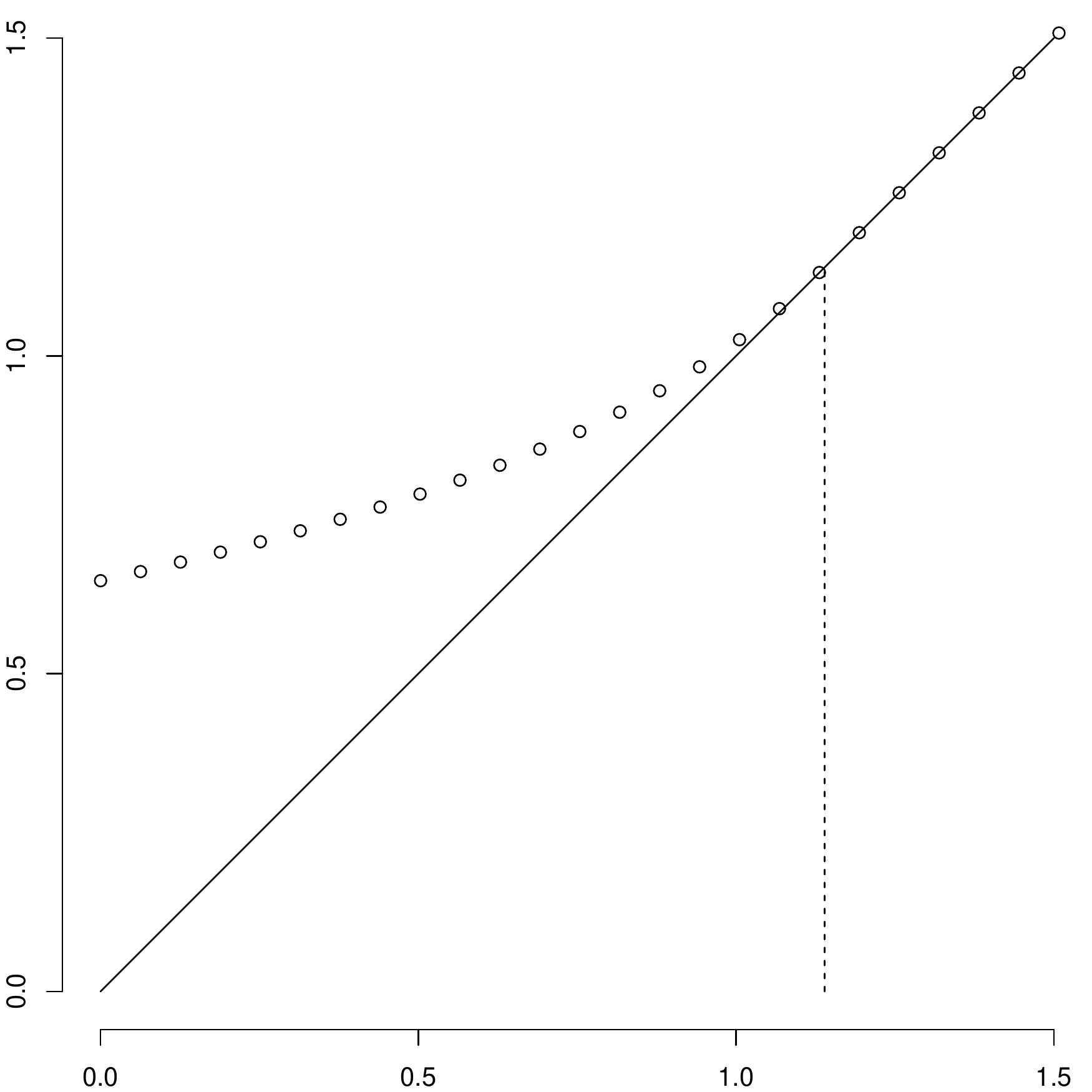}
\caption{\label{fig:ornstein-jumps} OSP for the Ornstein-Ulhenbeck with jumps. $g$ (continuous line), $V_1$ (circles).}
\end{center}
\end{figure}
\end{example}

\begin{remarks} (i) This example gives, up to our knowledge, the first explicit solution to an optimal stopping problem for a process with jumps that is not a L\'evy process. (ii) We find interesting in this example the way in which the theoretical results, the Fourier methods and computational power gathers.
\end{remarks} 

\begin{example}[$\desc=\gamma=1$ and $\lambda=0$] 
In this example we consider the process $X$ already presented with parameter $\lambda=0$, i.e. with no jumps and the same reward function $\g(x)=x^+$. This problem was solved by \cite{taylor}.

We have 
\begin{align*}
\hat{\Ga}(x,z)&=\left(e^{-\frac{1}{4}z^2}|z|^{-1}\right)\int_{0}^z  \frac{e^{i\zeta x}|\zeta|e^{\frac{1}{4}\zeta^2}}{\zeta} d\zeta\\
&=i \sqrt{\pi} e^{-\frac{1}{4}z^2}z^{-1} e^{x^2} \left(\erf\left(x-\frac{iz}{2}\right)-\erf(x)\right)
\end{align*}
As in the previous example, we solve numerically equation \eqref{eq:x*saltos} obtaining that $x^*\simeq 0.5939$.  \autoref{fig:ornstein} shows some points of the value function obtained numerically by the formula:
\bd
\Va(x)=\int_{x^*}^\infty \Ga(x,y)f(y) dy;
\ed
we also include in the plot the reward function (continuous line) to show that in the stopping region they coincide and also to verify that $\Va$ is a majorant of $g$ (hypothesis of \autoref{teo:huntVerif}). The obtained threshold is in accordance with the result obtained by Taylor.
\begin{figure}
\begin{center}
\includegraphics[scale=.6]{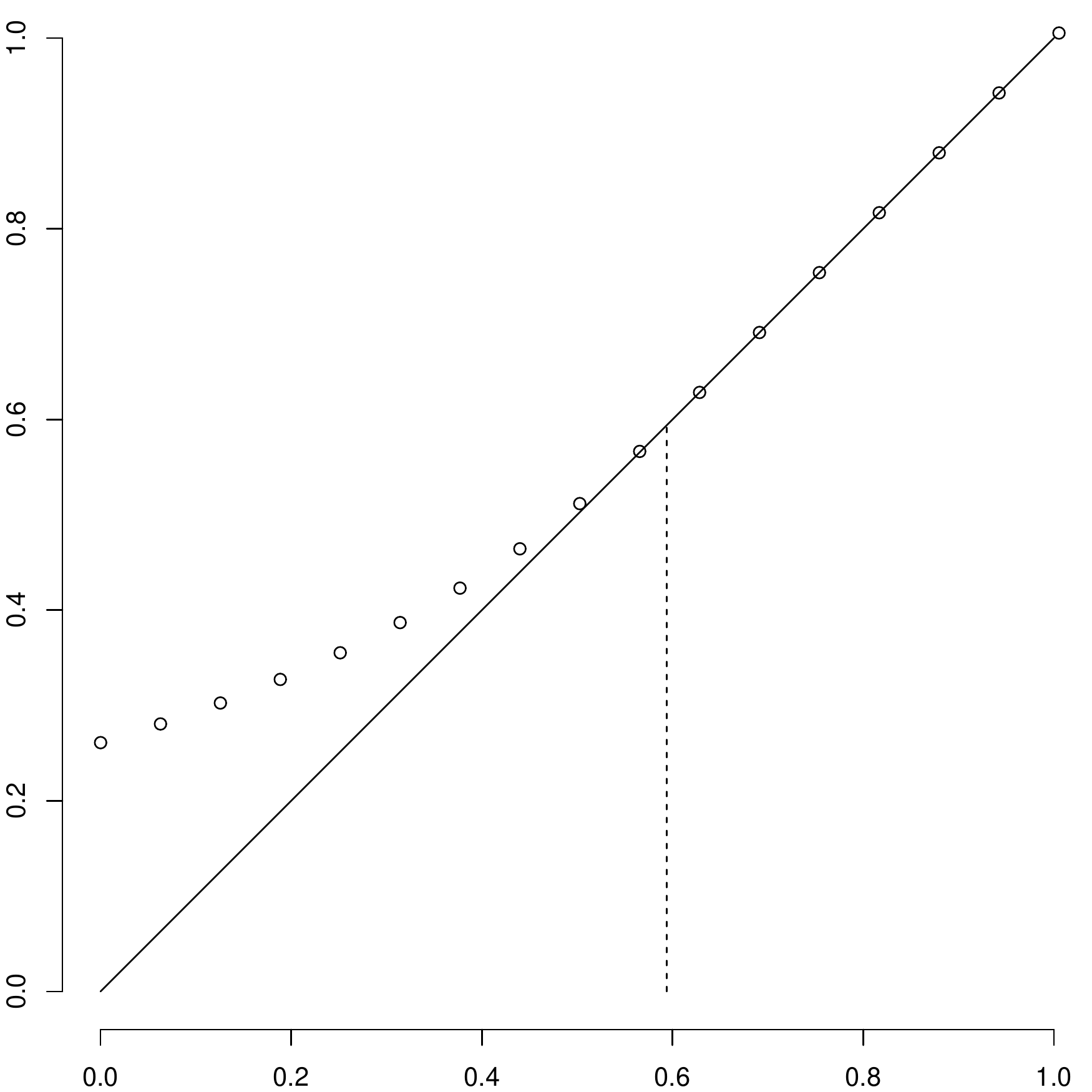}
\caption{\label{fig:ornstein} OSP for the Ornstein-Ulhenbeck process. $g$ (continuous line), $V_1$ (circles).}
\end{center}
\end{figure}
\end{example}

\appendix
\chapter{Numerical method to compute $\Ga(x,y)$}
\label{numerical}
In this section we show how to use Fourier methods to find a numerical approximation of the Green function. In our approach it is completely necessary to know $\Ga(x,y)$ in order to solve concrete optimal stopping problems. We remember the discrete Fourier transform of the vector $(v_0,\ldots , v_{n-1})$ is the vector $(w_0,\ldots , w_{n-1})$ such that
\be \label{eq:dft}
w_k=\sum_{j=0}^{n-1}v_j e^{-i2\pi \frac{j}{n}k}.
\ee

As we have seen in the examples, sometimes we do not know the Green function $\Ga(x,y)$ but we can find the transformed function 
\bd
\hat{\Ga}(x,z)=\int_{-\infty}^{\infty} e^{izy}\Ga(x,y)dy.
\ed
To recover $\Ga(x,y)$ from $\hat{\Ga}(x,z)$ we have
\bd
\Ga(x,y)=\frac{1}{2\pi}\int_{-\infty}^{\infty} e^{-izy}\hat{\Ga}(x,z)dz.
\ed

Departing from a discrete vector $\left(\hat{\Ga}(x,z_0),\ldots,\hat{\Ga}(x,z_{n-1})\right)$ we can use the discrete Fourier transform to find an approximation of $\Ga(x,y_0),\ldots,\Ga(x,y_{n-1})$. To do this observe that, assuming that the integral is convergent and with $A$ sufficiently large we have
\begin{align*}
\int_{-\infty}^{\infty} e^{-izy}\hat{\Ga}(x,z)dz &\simeq \int_{-A/2}^{A/2} e^{-izy}\hat{\Ga}(x,z)dz\\
&\simeq \frac{A}{n} \sum_{j=0}^{n-1} e^{-iz_j y}\hat{\Ga}(x,z_j) 
\end{align*}
where $z_j=-\frac{A}{2}+j\frac{A}{n}.$  Consider $y_k=2\pi \frac{k}{A}$; we have
\begin{align*}
\Ga(x,y_k)&\simeq \frac{1}{2\pi}\frac{A}{n} \sum_{j=0}^{n-1} e^{-iz_j y_k}\hat{\Ga}(x,z_j) \\
&\simeq \frac{1}{2\pi}\frac{A}{n} \sum_{j=0}^{n-1} e^{-i(-\frac{A}{2}+j\frac{A}{n})(2\pi \frac{k}{A})}\hat{\Ga}(x,z_j)\\
&\simeq \frac{1}{2\pi}\frac{A}{n} e^{i\pi k} \sum_{j=0}^{n-1} e^{-i 2\pi \frac{j}{n}k} \hat{\Ga}(x,z_j)
\end{align*}
From equation \eqref{eq:dft}, considering $v_j=\hat{\Ga}(x,z_j)$, for $j=0,\ldots,n-1$, we have
\bd
\Ga(x,y_k)\simeq \frac{1}{2\pi}\frac{A}{n} e^{i\pi k} w_k
\ed
for $k=0,\ldots,n-1$, where $(w_0,\ldots,w_{n-1})$ is the discrete Fourier transform of $(v_0,\ldots,v_{n-1})$

%\addcontentsline{toc}{chapter}{Bibliography} 
%\bibliographystyle{plainnat}
%\bibliographystyle{chicago}
%\bibliography{tesis.bib}

\end{document}